\documentclass[11pt]{article}
\usepackage[letterpaper, total={6.5in, 9in}]{geometry}
\usepackage[english]{babel}
\usepackage[utf8]{inputenc} 
\usepackage[T1]{fontenc}    
\usepackage[colorlinks=true,linkcolor=blue]{hyperref}       
\usepackage{url}            
\usepackage{booktabs}       
\usepackage{amsfonts}       
\usepackage{nicefrac}       
\usepackage{microtype}      
\usepackage{geometry}
\usepackage{subcaption}
\usepackage{graphicx}
  
\usepackage[ruled,linesnumbered]{algorithm2e}
\usepackage{amssymb}
\usepackage{amsfonts}
\usepackage{amsmath}
\numberwithin{equation}{section}
\numberwithin{equation}{section}
\usepackage{amsthm}
\usepackage{latexsym}  
\usepackage{epsfig}
\newtheorem{theorem}{Theorem}[section]
\newtheorem{corollary}[theorem]{Corollary}
\newtheorem{lemma}[theorem]{Lemma}
\newtheorem{fact}{Fact}[section]
\newtheorem{proposition}[theorem]{Proposition}

\newtheorem{definition}{Definition}[section]

\newtheorem{remark}[theorem]{Remark}
\newtheorem{mainresult}{Main Result} 
\newtheorem{property}{Property} 
\newtheorem{assumption}{Assumption}

\usepackage{enumitem}

\usepackage[utf8]{inputenc}
\usepackage{pifont}

\usepackage{adjustbox}
\usepackage{pdfcomment}
\usepackage{todonotes}
\usepackage{setspace}
\usepackage{array}
\usepackage{wrapfig}

\makeatletter
\patchcmd{\thebibliography}{\settowidth}
  {\setlength{\itemsep}{0pt plus 0.5ex} \setlength{\parskip}{0pt} \settowidth}
  {}{}
\makeatother

\newcommand{\linV}{\vec{V}}
\newcommand{\linVp}{\vec{V}_p}
\newcommand{\linVd}{\vec{V}_d}
\newcommand{\soln}{\operatorname{SOLN}}
\newcommand{\opt}{\operatorname{OPT}}
\newcommand{\dist}{\operatorname{Dist}}
\newcommand{\gap}{\operatorname{Gap}}
\newcommand{\thetau}{\theta^\star}
\newcommand{\thetax}{\theta_p^\star}
\newcommand{\thetas}{\theta_d^\star}
\newcommand{\limitingeratio}{limiting error ratio}
\newcommand{\eratio}{error ratio}
\newcommand{\distinfeas}{\mathrm{DistInfeas}}
\newcommand{\LPsharp}{LP sharpness}
\newcommand{\er}{ER}
\newcommand{\limitinger}{LimitingER}

\newcommand{\ignore}[1]{}

\usepackage{xcolor}

\newcommand{\condN}{\mathcal{N}}

\newcommand{\eps}{\varepsilon}

\newcommand{\calD}{{\cal D}}
\newcommand{\calE}{{\cal E}}
\newcommand{\calF}{{\cal F}}

\newcommand{\calH}{{\cal H}}
\newcommand{\calI}{{\cal I}}

\newcommand{\calM}{{\cal M}}
\newcommand{\calN}{{\cal N}}

\newcommand{\calR}{{\cal R}}
\newcommand{\calS}{{\cal S}}

\newcommand{\calU}{{\cal U}}
\newcommand{\calV}{{\cal V}}

\newcommand{\calX}{{\cal X}}
\newcommand{\calY}{{\cal Y}}
\newcommand{\calZ}{{\cal Z}}

\newcommand{\e}{{\boldsymbol e}}
\newcommand{\f}{{\boldsymbol f}}

\title{Computational Guarantees for Restarted PDHG for LP based on ``Limiting Error Ratios'' and LP Sharpness}
\author{Zikai Xiong\thanks{MIT Operations Research Center, 77 Massachusetts Avenue, Cambridge, MA 02139, USA. 
\href{mailto:zikai@mit.edu}{zikai@mit.edu}.  Research supported by AFOSR Grant No. FA9550-22-1-0356.
}  
\and Robert M. Freund\thanks{MIT Sloan School of Management, 77 Massachusetts Avenue, Cambridge, MA 02139, USA. 
\href{mailto:rfreund@mit.edu}{rfreund@mit.edu}. Research supported by AFOSR Grant No. FA9550-22-1-0356.}}
\date{\today}

\begin{document}

\maketitle

\begin{abstract}
In recent years, there has been growing interest in solving linear optimization problems -- or more simply ``LP'' -- using first-order methods in order to avoid the costly matrix factorizations of traditional methods for huge-scale LP instances. The restarted primal-dual hybrid gradient method (PDHG) -- together with some heuristic techniques -- has emerged as a  powerful tool for solving huge-scale LPs. However, the theoretical understanding of the restarted PDHG and the validation of various heuristic implementation techniques are still very limited. Existing complexity analyses have relied on the Hoffman constant of the LP KKT system, which is known to be overly conservative, difficult to compute (and hence difficult to empirically validate), and fails to offer insight into instance-specific characteristics of the LP problems. These limitations have limited the capability to discern which characteristics of LP instances lead to easy versus difficult LP instances from the perspective of computation. With the goal of overcoming these limitations, in this paper we introduce and develop two purely geometry-based condition measures for LP instances: ``\limitingeratio''~and \LPsharp. We provide new computational guarantees for the restarted PDHG based on these two condition measures. For \limitingeratio, we provide a computable upper bound and show its relationship with the data instance's proximity to infeasibility under perturbation. For \LPsharp, we prove its equivalence to the stability of the LP optimal solution set under perturbation of the objective function. We validate our computational guarantees in terms of these condition measures via specially constructed instances. Conversely, our computational guarantees validate the practical efficacy of certain heuristic techniques (row preconditioners and step-size tuning) that improve computational performance in practice. Finally, we present computational experiments on LP relaxations from the MIPLIB dataset that demonstrate the promise of various implementation strategies.
\end{abstract}

\section{Introduction, Motivations, and Main Results}\label{sec:intro}

The focus of this paper is on solving huge-scale instances of linear optimization problems -- or more simply ``LP''.  LP problems abound across a wide variety of applications from manufacturing, transportation, service sciences, to computational science and engineering. Up until very recently, the most successful methods for solving LP problems have been simplex and pivoting methods~\cite{dantzig1963linear} and interior-point methods \cite{wright}; these methods have been extensively studied and implemented in state-of-the-art commercial solvers~\cite{gurobi}. In most cases, they are able to obtain a high-accuracy solution, but the success of these methods relies on repeatedly solving a linear system in each iteration. For LP instances of a huge scale, the matrix factorizations required for solving the linear systems can be prohibitively costly. Moreover, the matrix factorizations are often unable to exploit the natural sparsity of a given LP instance and can have prohibitively large memory requirements. In contrast, first-order methods (FOMs) -- and in particular the primal-dual hybrid gradient method (PDHG) \cite{chambolle2016ergodic} -- are emerging as an alternative for solving huge-scale LP problems because they do not require the repeated solution of linear equations nor do they impose large memory requirements, thus reducing per-iteration costs. The primary task within each iteration of a FOM for LP is the gradient computation, which typically only requires matrix-vector multiplications (and so can fully take advantage of the sparsity of the LP instance). Moreover, FOMs are more suitable for distributed and parallel computation, and can benefit from modern computational architectures that accelerate computation through distributed systems and graphics processing units (GPUs). And indeed this compatibility with modern hardware architectures underscores the growing importance of FOMs for solving larger-scale LP instances.

Perhaps the best-known implementation of an FOM for solving LP is the solver PDLP~\cite{applegate2021practical}, which is based on the primal-dual hybrid gradient method (PDHG)~\cite{chambolle2016ergodic} to solve the saddlepoint formulation of LP. In the experiments reported in~\cite{applegate2021practical}, PDLP was able to outperform the commercial solver Gurobi when the LP problem was large-scale. A recent GPU implementation of PDLP further outperforms traditional algorithms implemented in the state-of-art commercial solvers on more LP instances~\cite{lu2025cupdlp}. Furthermore, \cite{pdlpnews} presents a distributed version of PDLP that is used to solve practical LP problems with $92$ billion non-zeros in the constraint matrix -- which is far beyond the capability of any simplex or interior-point method. PDLP is based on PDHG~\cite{chambolle2016ergodic} -- often referred to as the Chambolle-Pock method.  PDHG is an operator-splitting method with alternating updates between the primal and dual variables. On top of running the base algorithm PDHG, schemes for restarting PDHG have also been proven in theory to help PDHG achieve faster linear convergence~\cite{applegate2023faster,lu2022infimal} and this theory has yielded impressive speed-ups in practice as well. And in addition to using restarts, PDLP also utilizes various heuristic techniques such as presolving, row preconditioning, and step-size tuning~\cite{applegate2021practical}.
 
We work with LP in the standard form:
\begin{equation}\label{pro: general primal LP}
	\min_{x\in\mathbb{R}^n}  \ c^\top x \quad	\text{s.t.}  \ Ax = b, \  x \ge 0 \ ,
\end{equation} 
where the constraint matrix $A \in \mathbb{R}^{m \times n}$, the right-hand side vector $b \in \mathbb{R}^{m}$, and the objective vector $c \in \mathbb{R}^n$, whose standard dual problem is:
\begin{equation}\label{pro: standard dual LP}
		\max_{y \in \mathbb{R}^m}  \ b^\top y  \quad	\text{s.t.}  \ A^\top y \le  c \ .
\end{equation}
The problem \eqref{pro: general primal LP} can also be expressed as the saddlepoint problem:
\begin{equation}\label{pro: saddle point LP}
	\min_{x \in \mathbb{R}^n_+ } \max_{y\in \mathbb{R}^m} L(x,y) : = c^\top x + b^\top y - x^\top A^\top y
\end{equation}
where $L(x,y)$ is the Lagrangian function and $y$ is the vector of multipliers on the equation system $Ax=b$, where exchanging the $\min$ and $\max$ operations in \eqref{pro: saddle point LP} leads to \eqref{pro: standard dual LP}. In this paper we let $\calX^\star$ and $\calY^\star$ denote the optimal solution sets of \eqref{pro: general primal LP} and \eqref{pro: standard dual LP}, let $\calZ^\star := \calX^\star \times \calY^\star$ denote the solution set of the saddlepoint problem \eqref{pro: saddle point LP}, and let $z:=(x,y)\in\mathbb{R}^{m+n}$ denote the combined primal/dual iterates.

The family of PDHG algorithms (using various step-sizes, and with/without overlaid restart schemes) is designed to directly tackle the saddlepoint problem \eqref{pro: saddle point LP} rather than the original problem \eqref{pro: general primal LP} or/and its dual \eqref{pro: standard dual LP}. One step of PDHG for \eqref{pro: saddle point LP} at the point $z=(x,y)$ is defined as follows:
\begin{equation}\label{eq: one PDHG}
	z^+  = \textsc{PDHGstep}(z) := \left\{
		\begin{array}{l}
			x^+ := P_{\mathbb{R}^n_+}\left(x-\tau\left(c-A^{\top} y\right)\right)   \\
			y^+:=y+\sigma\left(b-A\left(2 x^+ -x \right)\right)  \ , 
			\end{array}
	\right. 
\end{equation}
in which $\tau$ and $\sigma$ are the primal and dual step-sizes, respectively, and $P_{\mathbb{R}^n_+}$ is the projection operator onto the non-negative orthant $\mathbb{R}^n_+$ (which is the computationally trivial task of taking the nonnegative parts of the components).  
The vanilla PDHG tackles LP by generating iterates according to: $z^{k+1} \leftarrow \textsc{PDHGstep}(z^k)$ for $k=0,1,2,\ldots.$ PDHG with restarts, which is denoted by rPDHG, is a variant of PDHG that regularly restarts PDHG using the average of the previous $\ell$ iterates where $\ell$ is chosen according to some rule, see Section \ref{restarts} for a detailed description of rPDHG.

In this paper we seek to more deeply understand the performance of rPDHG applied to LP problems, both in theory and in practice, and to improve the theory where possible, as well as to explore the extent to which the theory is aligned with computational practice of rPDHG. The starting point of our work is the algorithm and analysis of rPDHG in the paper \cite{applegate2023faster}, which contains many new and important ideas both in terms of methods and analysis, and whose main results for rPDHG we now attempt to summarize in a brief and cogent manner. Along with other primal-dual methods,  \cite{applegate2023faster} analyzes rPDHG for solving \eqref{pro: saddle point LP}. Let the iterates of rPDHG be denoted by $z^{k}=(x^k,y^k)$. Taken together, Theorems 1 and 2 of \cite{applegate2023faster} state that rPDHG requires at most 
\begin{equation}\label{intro:alpha_complexity}
	O\left(\frac{\|A\|}{\alpha} \cdot \ln\left(\frac{\|A\|}{\alpha\eps}\right)\right) 
\end{equation} 
iterations in order to obtain an iterate $z^k$ for which $\dist(z^k,\calZ^\star)\le \eps$, where the notation $\dist(z,\calZ)$ denotes the Euclidean distance from a point $z$ to the set $\calZ$. Here the notation $O(\cdot)$ hides only absolute constants, $\|A\|$ is the spectral norm of $A$, and $\alpha$ is a positive scalar related to the sharpness \cite{polyak1979sharp, burke1993weak} of a particular functional called the ``normalized duality gap'' function, that we now describe. The normalized duality gap function is denoted $\rho(r;z)$ and is defined parametrically for a given positive ``radius'' $r$ as:
\begin{equation}\label{fearofheights_intro}
	\rho(r;z) := \left(\frac{1}{r}\right)\max_{  \hat{z} \in \widetilde{B}(r;z) }  \big[ L(x,\hat{y}) - L(\hat{x},y) \big] \ ,
\end{equation} 
where $z = (x,y)\in \mathbb{R}^n_+\times \mathbb{R}^m$, $\widetilde{B}(r;z) := \{\hat{z} := (\hat{x},\hat{y}):  \hat{x}\ge 0 \text{ and } \|\hat{z} -z\|_M \le r\}$, and the norm $\| \cdot \|_M$ is a carefully selected matrix norm constructed using $A$ and the step-size parameters of PDHG, where $M = \begin{pmatrix} I & -\eta A^\top \\ -\eta A & I \end{pmatrix}$ in the simple case when the primal and dual step-sizes of PDHG are identically equal to $\eta$ and  $\eta \le 1/\|A\|$. The scalar $\alpha$ in \eqref{intro:alpha_complexity} is related to the ``sharpness'' of $\rho(r;z)$, which we now describe as well. As developed in \cite{polyak1979sharp} and extended in \cite{burke1993weak}, the sharpness of a function $f$ essentially measures how fast $f$ grows away from its optimal solution set, and we say that $f$ is $\beta$-sharp if $f(v) - f^\star \ge \beta \cdot \dist(v, \calV^\star)$ for any $v$, where $\calV^\star$ is the set of minimizers of $f$, and $f^\star$ is the minimum value of $f$.  The quantity $\alpha$ in 
\eqref{intro:alpha_complexity} is related to these sharpness notions applied to the normalized duality gap function and is defined to be a constant that satisfies
\begin{equation}\label{intro:sharpness}
	 \rho(r^k;z^k) \ge \alpha\cdot \dist(z^k,\calZ^\star) 
\end{equation} 
for all iterates $z^k$ of the algorithm. (Note that this is a bit weaker than the actual sharpness of $\rho(r; \cdot)$ as it only needs to hold for the iterates $z^k$.) Here the radius parameter $r^k$ depends on the iteration $k$ and is dynamically and adaptively defined by the algorithm's iterates, see \cite{applegate2023faster} for the precise details. In summary, if there exists a positive value $\alpha$ for which \eqref{intro:sharpness} holds for all iterates $z^k$ of rPDHG, then the algorithm has an overall iteration complexity bound given by \eqref{intro:alpha_complexity}.   


Furthermore, Property 3 and Lemma 5 of \cite{applegate2023faster} show that \eqref{intro:sharpness} always holds for 
\begin{equation}\label{medicareb} \alpha = \frac{1}{\calH(K)\sqrt{1+16\dist(0,\calZ^\star)^2}} \ , 
\end{equation}
where $\calH(K)$ is the Hoffman constant\footnote{The Hoffman constant of a matrix $W$ is a global error bound that bounds the distance of any point $v$ to a non-empty set of solutions of a system of linear inequalities $ Wv \le g$ in terms of the norm of the residual vector $\|[W v -g]^+ \|$, see \cite{hoffman1952approximate}.} \cite{hoffman1952approximate} of the matrix $K$ of the Karush-Kuhn-Tucker linear inequality system that defines the optimal solution set, namely 
$$
K:=   
\begin{pmatrix}
	I & -A^\top & A^\top & 0 & -c \\
	0 & 0 & 0 & -A & b  
	\end{pmatrix}^\top \ .
$$ Combining \eqref{intro:alpha_complexity} with \eqref{medicareb} one also obtains the following iteration bound for rPDHG for LP:
\begin{equation}\label{intro:complexity}
	O\left(\|A\|\cdot \calH(K) \cdot \max\{1,\dist(0,\calZ^\star)\}  \cdot \ln\left( \|A\|\cdot \calH(K)\cdot \max\{1,\dist(0,\calZ^\star)\}\cdot\frac{1}{\varepsilon}\right) \right)  \ .
\end{equation} 

\subsection{Motivating Issues}\label{subsec:motivating issues}

\noindent The rPDHG algorithm and iteration bounds \eqref{intro:alpha_complexity} and/or \eqref{intro:complexity} in \cite{applegate2023faster} are quite significant in at least several ways, including but not limited to the algorithm design (the restart scheme for rPDHG both theoretically and practically), the proof of linear convergence of rPDHG, the use of and the development of properties of the normalized duality gap function $\rho(r;z)$, and the use of the Hoffman constant $\calH(K)$ of the KKT system matrix $K$ to bound the constant $\alpha$ in \eqref{intro:sharpness}.

Nevertheless, there are certain issues with the bounds \eqref{intro:alpha_complexity} and/or \eqref{intro:complexity} that are not very satisfactory, and which we seek to overcome. One issue has to do with the reliance on the sharpness constant $\alpha$ of the normalized duality gap function $\rho(r;z)$ in \eqref{intro:alpha_complexity}. The normalized duality gap function $\rho(r;z)$ itself is not such a natural metric of LP behavior, as there is (at best) only a partial equivalence between $\rho(r;z)$ and more typical metrics and stopping criteria that are used in LP solvers such as primal and dual infeasibility and non-optimality measures. Also, the definition of $\rho(r;z)$ depends on the step-sizes of the algorithm through the matrix $M$. Therefore the sharpness of $\rho(r;z)$ and consequently the value of $\alpha$ depend -- at least partially -- on the magnitude and ratio of the primal and dual step-sizes of the algorithm and hence depend on more than just the intrinsic properties of the LP.  It is thus unclear from the iteration bound \eqref{intro:alpha_complexity} what natural/intrinsic properties of the LP problem itself (such as geometric properties, data-perturbation metrics, error bounds, etc.) contribute to the performance -- theoretically or practically -- of rPDHG.

Another issue concerns the iteration bound \eqref{intro:complexity}.  This bound replaces $\alpha$ by two other metrics, namely $\dist(0,\calZ^\star)$ and $\calH(K)$.  The reliance on $\dist(0,\calZ^\star)$, which measures the norms of primal and dual optimal solutions, seems both natural and appropriate, as in the very least $\dist(0,\calZ^\star)$ bounds changes in optimal solution values under perturbations of $b$ and $c$ and thus is tied to general notions of condition number theory for LP much more broadly, see \cite{renegar1994some}.  However, the reliance in \eqref{intro:complexity} on the Hoffman constant $\calH(K)$ of the KKT system matrix $K$ is not desirable for a number of reasons.  For one, the Hoffman constant of a matrix $W$ is typically extremely large by definition, as it accounts for the largest error bound of every linear inequality system $\calV_g = \{ v : Wv \le g\}$ over all possible right-hand side vectors $g$ for which $\calV_g \ne \emptyset$.  Furthermore, it is typically an extremely conservative measure since in academic applications one is typically only concerned with the error bound for a single given $g$.  And in the KKT system that single $g$ is given by $(0, -b, b, -c, 0)^\top$ and thus has its own special structure by itself and also is tied to the data $b$ and $c$ which appear in the matrix $K$ as well.  

A third issue also concerns the non-local nature of the error bounds embedded in the Hoffman constant.  If the primal or the dual feasible region has an extreme point whose active constraint system is badly ill-conditioned, then $\calH(K)$ will be very large, even if this extreme point is not related at all to the optimal solution in terms of active constraints and/or distance to the optima or objective function value.  Indeed, it would be better to have an iteration bound that does not depend so globally on properties of all points in the primal and the dual feasible sets.  

The fourth issue with the iteration bounds \eqref{intro:alpha_complexity} and/or \eqref{intro:complexity} has to do with computability in order to test whether or not the bounds align with computational practice.  From both the practical and theoretical perspectives, it is important to ask whether an iteration bound aligns with computational practice.  In other words, when applied to problems that arise in practice do problems with smaller sharpness constant values $\alpha$  require more iterations of rPDHG than problems with larger such values of $\alpha$?, and similarly for $\calH(K)$?  In order to study these questions one has to be able to actually compute (or approximately compute) $\alpha$ and/or $\calH(K)$.  However, it is not known (at least by us) how to compute $\alpha$ for LP problems in general.  And regarding the Hoffman constant, using the results in \cite{pena2021new} we know that $\calH(K)$ has the following characterization:
$$
\calH(K):=\max _{\substack{J \subseteq\{1, \ldots, 2 m+2 n+1\} \\ K_J \text { has full row rank }}} \frac{1}{\min _{v \in \mathbb{R}_{+}^J,\|v\|=1}\left\|K_J^{\top} v\right\|} \  ,
$$
where $K_J$ denotes the submatrix of $K$ formed by selecting the rows indexed by $J$. Even with this novel characterization, it is still a very difficult task to compute (or even just reliably estimate) $\calH(K)$ as it requires enumerating exponentially many submatrices \cite{pena2021new}. Thus neither of the bounds \eqref{intro:alpha_complexity} nor \eqref{intro:complexity} are amenable to testing the extent to which they might align with computational practice.

To illustrate the importance of alignment between theory and practice, consider the following extremely simple LP instance with $m=1$ and $n=2$, for which computing $\calH(K)$ is doable:
\begin{equation}\tag{$\mathrm{LP_{\gamma}}$}\label{pro sample 2dim LP}
	\begin{aligned}
			\underset{(x_1,x_2)\in\mathbb{R}^2_+}{\operatorname{min}}   \ \cos(\gamma) \cdot x_1 - \sin(\gamma) \cdot x_2   \quad
			 \operatorname{s.t.} \   \sin(\gamma) \cdot x_1 + \cos(\gamma) \cdot x_2 = 1 
	\end{aligned}
\end{equation} for $\gamma \in (0,\pi/2)$, which is illustrated in 
the left subfigure of Figure \ref{fig intro}. 
In the subfigure the feasible set is the blue line segment.  We note that for all $\gamma \in (0,\pi/2)$ the distance of the feasible set to $(0,0)$ is $1$. The optimal solution is the red point $(x_1^\star,x_2^\star) = (0, 1/\cos(\gamma))$, and the direction of the objective vector $c = [\cos(\gamma),-\sin(\gamma)]$ is denoted by the red dashed arrow.
The right subfigure shows the values of the Hoffman constant $\calH(K)$ of the KKT system, the iteration bound \eqref{intro:complexity} based on \cite{applegate2023faster}, and the actual iteration count of rPDHG for \eqref{pro sample 2dim LP} for $\gamma \in (0,\pi/2)$.  We see that as $\gamma \searrow 0$, \eqref{pro sample 2dim LP} becomes more ill-conditioned in terms of $\calH(K)$ and this is reflected in the iteration bound \eqref{intro:complexity}.  However, this family of LP instances is very easy for rPDHG to solve for arbitrarily small values of $\gamma$.  It would be better to have an iteration bound that is more aligned with actual computational practice.  (The blue line in the right subfigure is a spoiler: it shows the bound that we develop in Theorem \ref{thm overall complexity} of this paper, which for this family of problems is well-aligned with actual iteration counts.) Details of this experiment are presented in Section \ref{subsec:five simple experiments}.
\begin{figure}[htbp]
    \centering 
    \begin{subfigure}[b]{0.34\textwidth}
        \centering
        \includegraphics[width=\textwidth]{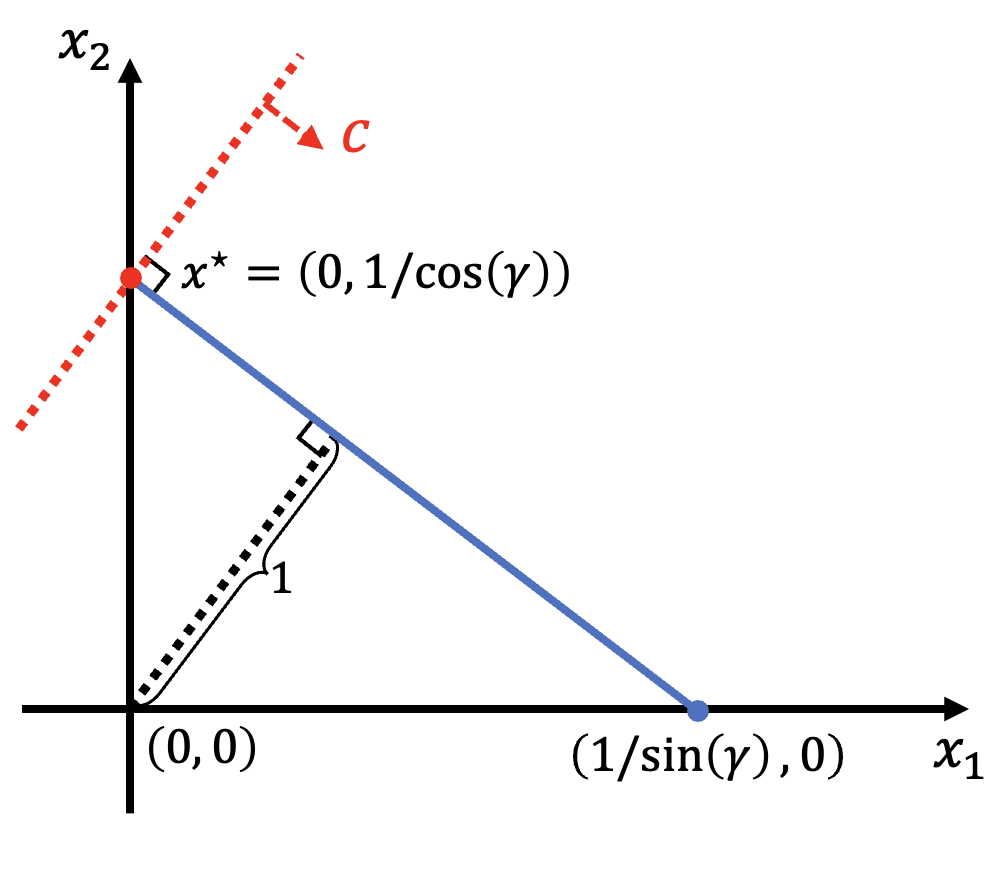} 
    \end{subfigure}
    \begin{subfigure}[b]{0.37\textwidth}
        \centering
        \includegraphics[width=\textwidth]{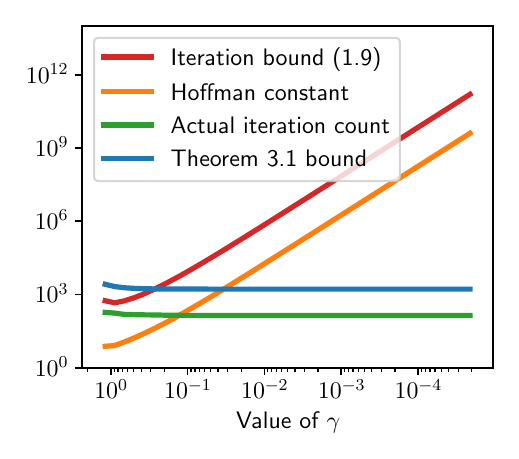} 
    \end{subfigure}
    \caption{\footnotesize (left) The feasible set (blue line) and the optimal solution (red point) of \eqref{pro sample 2dim LP}. (right) The values of the theoretical iteration bound \eqref{intro:complexity} based on \cite{applegate2023faster}, the Hoffman constant $\calH(K)$ of the KKT system, and the actual iteration count of rPDHG to achieve $z=(x,y)$ for which $\dist(z,\calZ^\star)\le 10^{-10}$.  The blue line shows our new bound in Theorem \ref{thm overall complexity} of this paper.  } \label{fig intro}
\end{figure}

\subsection{New computational guarantees based on ``Limiting Error Ratios'' and LP Sharpness}\label{subsec: main results}

We now describe our new computational guarantees.  Consider the original LP primal problem \eqref{pro: general primal LP}, and let $\calF_p$ denote the feasible set, defined as the intersection of the nonnegative orthant $\mathbb{R}^n_+$ and the affine subspace $V_p := \{x \in \mathbb{R}^n : Ax = b\}$, namely $\calF_p = V_p \cap  \mathbb{R}^n_+$.  The primal optimal solution set is denoted by $\calX^\star$. Our new computational guarantees involve two types of condition measures.  The first condition measure is denoted by $\thetax$, and is called the ``\limitingeratio,'' or ``\limitinger'' for short, which we now define. 
\begin{definition}[Error ratio and limiting error ratio]\label{def:limitinger}
	For any $x\in V_p \setminus \calF_p$ (namely $x$ satisfies the linear equality constraints but lies outside the nonnegative orthant and is thus infeasible), the \textit{\eratio~(\er)} of $\calF_p$ at $x$ is defined as:
	\begin{equation}\label{overcast}
		\theta_p(x) : =\frac{\dist(x, \calF_p)}{\dist(x,\mathbb{R}^n_+)} 
			\ ,
	\end{equation}
	which is the ratio of the distance to the feasible set $\calF_p$ to the distance to the nonnegative orthant.
	Let $\theta_p(x) := 1$ for $x \in \calF_p$ for notational completeness.  The  \textit{\limitingeratio~(\limitinger)} is the quantity $\thetax$ defined as:
	\begin{equation}\label{kansas}
			\thetax : =   \lim_{\eps \to 0} \left( \sup_{x\in V_p, \hspace{.05cm}\dist(x, \calX^\star) \le \eps} \theta_p(x) \right)
		\ ,
	\end{equation}
	which is the supremum of $\theta_p(x)$ for all $x \in V_p$ approaching $\calX^\star$.
\end{definition}
\noindent

The measure $\thetax$ (and also $\theta_p(x)$) is similar to error bounds proposed in the literature such as those in \cite{pang1997error,lewis1998error}. We call $\theta_p(x)$ an \textit{\eratio} because it represents the ratio of two types of errors: the distance to the feasible region and the distance to the nonnegative orthant $\mathbb{R}^n_+$. For any $x \in V_p \setminus \calF_p$ it always holds that $\theta_p(x) \ge 1$ because the numerator in \eqref{overcast} is always at least as large as the denominator (since $\calF_p \subset \mathbb{R}^n_+$).
Indeed, many papers have studied different formulations of global upper bounds for the \eratio s of linear inequality systems from different perspectives, starting with the celebrated Hoffman bound \cite{hoffman1952approximate}, including \cite{mangasarian1981condition,bergthaller1992distance,luo1994perturbation,goffin1980relaxation} among many others, see \cite{pang1997error} for a comprehensive survey of relevant results. 
	Note that $\thetax$ is also bounded above by a Hoffman bound $\calH(\bar A)$ on the linear inequality system $Ax \le b, \ -Ax \le -b, \ -x \le 0$, where $\bar A :=[A^\top \ -A^\top \ -I]^\top$, because $\theta_p(x) \le \calH(\bar A)$ for all $x$.  However, unlike $\calH(\bar A)$, $\thetax$ only depends on local information about the feasible set $\calF_p$ near $\calX^\star$, and hence $\calH(\bar A)$ is likely to be an excessively conservative bound on $\thetax$, since (i) $\calH(\bar A)$ is a global bound whereas $\thetax$ is a local bound, and (ii) $x$ must satisfy $x \in V_p$ in \eqref{kansas} of Definition \ref{def:limitinger} and only allows zero error in the system $Ax=b$.

It should be noted that $\thetax$ is a \emph{local} quantity, as it is the supremum of $\theta_p(x)$ over infeasible points $x\in V_p\setminus\calF_p$ that approach $\calX^\star$. In contrast, the global supremum $\sup_{x\in V_p\setminus\calF_p}\theta_p(x)$ can be arbitrarily larger than $\thetax$.  To see this, consider again the family of LP instances  \eqref{pro sample 2dim LP} for $\gamma \in (0,\pi/2)$.  Here $\calF_p$ is the line segment with endpoints $x^\star=(0,\tfrac{1}{\cos\gamma})$ and $\bar x:=(\tfrac{1}{\sin\gamma},0)$. Now parameterize $V_p$ by
$x(t)=t\,x^\star+(1-t)\,\bar x=(\tfrac{1-t}{\sin\gamma},\tfrac{t}{\cos\gamma})$ for $t \in (-\infty, +\infty)$.
Then for $t>1$ we have $\dist(x(t),\mathbb{R}^2_+)= \frac{t-1}{\sin \gamma}$ and $\dist(x(t),\calF_p)=(t-1)\|(\frac{1}{\sin \gamma}, \frac{1}{\cos \gamma})\|$, whereby $\theta_p(x(t)) = \tfrac{1}{\cos\gamma}$ and hence $\thetax=\lim_{t\downarrow 1}\theta_p(x(t))=\tfrac{1}{\cos\gamma}$.  However, for $t<0$ we have $\dist(x(t),\mathbb{R}^2_+)=\frac{|t|}{\cos \gamma}$ and $\dist(x(t),\calF_p)= |t|\|(\frac{1}{\sin \gamma}, \frac{1}{\cos \gamma})\|$, which yields $\theta_p(x(t))=\tfrac{1}{\sin\gamma}$.  Therefore $\sup_{x\in V_p\setminus\calF_p}\theta_p(x) = \max\{ \tfrac{1}{\cos \gamma}, \tfrac{1}{\sin \gamma} \}$ and the ratio of $\sup_{x\in V_p\setminus\calF_p}\theta_p(x)$ to $\thetax$ is equal to $\max\{1, \tfrac{\cos \gamma}{\sin \gamma} \} \rightarrow +\infty$ as $\gamma \downarrow 0$.

	The second condition measure is the ``\LPsharp''~and is denoted by $\mu_p$, and is defined as follows.
	\begin{definition}[\LPsharp]\label{def: sharpness mu} Let $H_p^\star$ denote the optimal objective hyperplane, i.e., $H_p^\star := \{\hat{x} \in \mathbb{R}^n:c^\top \hat{x}  = c^\top x^\star\}$ for an $x^\star \in \calX^\star$. The \LPsharp~$\mu_p$ of the primal problem is defined as:
	\begin{equation}\label{topsyturvy}
	\mu_p : = \inf_{x \in \calF_p\setminus\calX^\star} \frac{\dist(x,\ V_p\cap H_p^\star  )}{\dist(x,\ \calX^\star)} 
		\ ,
	\end{equation}
	which is the infimum of the ratio between the distance from $x$ to $V_p \cap H_p^\star$ and the distance from $x$ to the optimal set $\calX^\star$. 
	\end{definition}
	\noindent 
	Intuitively speaking, the \LPsharp~measures how quickly the objective function grows away from the optimal solution set $\calX^\star$ among all feasible points. In the case of LP it is easy to see that $\mu_p >0$.

Notions of sharpness were perhaps first introduced by Polyak in \cite{polyak1979sharp} as a useful analytical tool in convex minimization. For example, sharpness plus some mild smoothness assumptions can lead to linear convergence of the subgradient descent method via the use of restarts, see \cite{yang2018rsg}. \cite{applegate2023faster} generalizes the sharpness concept from convex optimization to primal-dual saddlepoint problems by defining sharpness for the normalized duality gap functional \eqref{fearofheights_intro}. Here we apply the notion of sharpness directly (and naturally) to the LP optimization problem itself.

The two condition measures $\thetax$ and $\mu_p$ are defined for the primal problem in the nonnegative (``cone'') variable $x$.  We define analogous condition measures for the dual problem in terms of the slack variables $s$ of the dual constraints, namely the variables $s := c - A^\top y$. Specifically, let $\calS^\star$ denote the set of optimal slack variable values in the dual LP.  Then it is straightforward to define dual counterparts $\thetas$ and $\mu_d$ of the primal condition measures $\thetax$ and $\mu_p$; see Section~\ref{restarts} for the formal definitions. Additionally, we let $\kappa$ denote the standard condition number of $A$, which is defined to be the ratio of the largest to the smallest positive singular value of $A$.

The main result of this paper is a new computational guarantee for rPDHG that depends on the above condition measures, which we now describe. We suppose that $c$ satisfies $Ac = 0$ (which can be easily enforced by projecting $c$ onto the nullspace of $A$ as part of a presolve scheme), and we suppose that the step-sizes $\tau$ and $\sigma$ are chosen in a certain way that is described in detail in Section \ref{restarts}.  Our main result (Theorem \ref{thm overall complexity}) is as follows:
\begin{mainresult}\label{main result}
(Less formal restatement of Theorem \ref{thm overall complexity}) Under the above conditions rPDHG requires at most 
\begin{equation}\label{eq overall complexity intro}
		O\left( \calN \cdot \ln\left[  \frac{ \calN \calD\eps_0}{ \eps}\right]\right)
	\end{equation}
iterations of \textsc{PDHGstep} in order to compute a pair $(x,s)$ of primal solution and dual slack that satisfies $\max\{\dist(x,\calX^\star),\dist(s,\calS^\star)\} \le \eps$, where $\calN$ and $\calD$ are defined as follows: 
\begin{equation}\label{eq:def_N}
		\calN := \kappa \cdot \left(\frac{1}{\mu_p}+\frac{1}{\mu_d}\right) \left(\thetax+\thetas+\frac{\|x^\star\|}{\|b\|_{Q}}+\frac{\|s^\star\|}{\|c\|}\right) 
	\end{equation}
and 
	\begin{equation}\label{eq:def_D}
		\calD:=32e\cdot \kappa \cdot \max\left\{\frac{\|c\|}{\|b\|_{Q}}, \frac{\|b\|_{Q}}{\|c\|}\right\} \ , 
	\end{equation}
and $\eps_0 := \max\{\dist(x^0,\calX^\star),\dist(s^0,\calS^\star)\}$ is the error of the initial iterate $(x^0,s^0)$ measured using the distance to the set of optimal solutions. Here $x^\star$ and $s^\star$ are the least norm optimal primal solution and dual slack solution, respectively.  Also $\|b\|_Q$ denotes $\|A^\top (AA^\top)^\dagger b\|$, in which $Q := (AA^\top)^{-1}$ if $A$ has full row rank. The quantity $e$ is the base of the natural logarithm.
\end{mainresult}

Let us now examine the components of the iteration bound in Main Result \ref{main result} a bit closer. First notice that the ratio of the initial error to the target error $\eps_0/\eps $ appears inside the logarithm term and is a consequence of the global linear convergence property of the algorithm.  The quantity $\condN$ appears both inside and outside of the logarithm term and itself involves the condition number $\kappa$ of the matrix $A$, plus six other quantities.  These are $\mu_p$ and $\mu_d$ (the \LPsharp~for both the primal and dual problems), $\thetax$ and $\thetas$ (the \limitinger~for both the primal and dual problems), as well as $\frac{\|x^\star\|}{\|b\|_Q}$ and $ \frac{\|s^\star\|}{ \|c\|} $. Notice that higher values of $\mu_p^{-1}$, $\mu_d^{-1}$, $\thetax$, and/or $\thetas$ result in a higher value of $\condN$.  Also notice that all four cross-terms between primal and dual \limitinger s  and the reciprocals of the \LPsharp~are present in $\condN$, as well as all four cross-terms between $\frac{\|x^\star\|}{\|b\|_{Q}},\frac{\|s^\star\|}{\|c\|}$ and the reciprocals of the \LPsharp. The numerator in $\frac{\|x^\star\|}{\|b\|_Q}$ is the norm of the least-norm primal optimal solution, which measures the stability of the dual problem under perturbation of $c$. In Fact \ref{fact: symmetric LP formulation} of Section \ref{secpreliminaries} we show that the denominator $\|b\|_Q$ is equal to $\dist(0,V_p)$, which is the distance from $0$ to $V_p$, and so $\|b\|_Q$ can be interpreted as a lower bound on the norm of any (and every) feasible or optimal solution of the primal. Therefore the quotient $\frac{\|x^\star\|}{\|b\|_Q}$ is equal to the geometric measure $\frac{\dist(0,\calX^\star)}{\dist(0,V_p)}$ and can be interpreted as a relative measure of stability or relative distance to optima.  A similar interpretation holds for $ \frac{\|s^\star\|}{ \|c\|} $. 

We note that $\calN$ is the only quantity outside of the logarithm term in the iteration bound and hence is the quantity characterizing the rate of linear convergence.  We will shortly focus on $\calN$ in our discussion and comments.

The quantity $\calD$, which appears only in the logarithm term of the iteration bound, is a (positive) scale-invariant measure of the problem data, and is less important since it only appears inside the logarithm term in \eqref{eq overall complexity intro}. The proofs leading to Main Result \ref{main result} actually use a similar strategy as in \cite{applegate2023faster}.  Roughly speaking, these proofs proceed by showing that a sharpness condition akin to \eqref{intro:sharpness} (but using a different norm) holds for $\alpha = \frac{1}{\condN}$ for all iterates, see Lemma \ref{lm: Mdistance upper bounded by normalized duality gap} in Section \ref{restarts}.

\subsection{Remarks on Main Result \ref{main result}}\label{subsec: comments}

\paragraph{Geometric nature of the new iteration bound.}
The iteration bound \eqref{eq overall complexity intro} in Main Result \ref{main result}, as well as the specific quantity $\calN$ that is outside the logarithm term, essentially relies on the primal and dual \limitinger~values $\thetax$ and $\thetas$, the primal and dual \LPsharp~values $\mu_p$ and $\mu_d$, and the minimum norms of solutions $x^\star$ and $s^\star$, as well as on data metrics involving the condition number $\kappa$ of $A$ and other norms of the data $(A, b, c)$.  The primal and dual \limitinger~values and the \LPsharp~values are all geometric in nature and (in our view) arise naturally in the study of the behavior or ``conditioning'' of an LP instance at its optima. Recall that $\thetax$ and $\thetas$ are local supremums of the error ratios $\theta_p(x)$ and $\theta_d(s)$ near $\calX^\star$ and $\calS^\star$, and $\mu_p$ and $\mu_d$ quantify the sharpness of the objective function near the optimal solution set. In this regard $\thetax$, $\thetas$, $\mu_p$, and $\mu_d$ depend only on the optimal solution set and an arbitrarily small neighborhood around the optimal solution set. Especially because the bound in Main Result \ref{main result} is described only by local information around the optima, it is likely to be significantly tighter than \eqref{intro:complexity} which depends on the Hoffman constant $\calH(K)$. (An ill-conditioned basis at a non-optimal extreme point may significantly increase $\calH(K)$, yet it will have no impact on $\calN$.) 

\paragraph{Dependence only on local behavior near the optima.} Because the iteration bound in Main Result \ref{main result} only depends on local behavior near the optimal solution set, it therefore shows that the performance of PDHG -- at least in theory -- is only tied to local properties of the LP instance local to the optimal solution set. 

\paragraph{Computability of the bound in Main Result \ref{main result}.}
From both a practical and theoretical perspective, it is important to ask whether the bound in Main Result \ref{main result} aligns with computational practice. In other words, when applied to problems that arise in practice do problems with smaller values of \eqref{eq overall complexity intro} require fewer iterations of rPDHG than problems with larger such values? (More generally it is important to ask this question for any algorithm for any optimization problem; however it is all the more important for LP given its pervasive use in practice, and it is important for rPDHG in order to better understand whether/where one might discover improvements in the theory or in practice in this nascent stage.)  In order to answer this and related questions it is necessary to efficiently compute the component quantities involved in \eqref{eq overall complexity intro}.  The quantities $\|b\|_Q$, $\|c\|$, and $\kappa$ are not difficult to compute (or estimate with high accuracy).  Also $\|x^\star\|$ and $\|s^\star\|$ are readily computable once an optimal solution has been computed.  Hence the challenge in computing the bound in Main Result \ref{main result} lies in computing $\theta_p^\star$ and $\mu_p$ and their analogous dual quantities. Let us first consider $\theta_p^\star$.  While we have not uncovered a direct way to compute $\theta_p^\star$ exactly, we have developed an efficient way to compute an upper bound $\theta_p^\star$ via the following result:  
\begin{property}\label{texas primal} {\bf (essentially Proposition \ref{cusco}.)} Suppose $x_a \in \calX^\star$ and there exists $R_a$ for which $\calX^\star \subset \{x: \|x-x_a \| \le R_a\}$, then it holds that $ \thetax \le G^\star$ for $G^\star$ defined as follows:
	\begin{equation}\label{rrr2d2 primal}  
			G^\star := \ \inf_{r >0, \ x\in \mathbb{R}^n}  \displaystyle\frac{R_a+\|x-x_a\|}{r} \quad \operatorname{s.t.} \ x \in V_p , \ x \ge r \cdot e \ . 
	\end{equation} 
\end{property}
\noindent
In order to compute $G^\star$ above one needs to know a ball containing the optimal solution set; if $\calX^\star$ is a singleton then it is sufficient to know an optimal solution, and if there are multiple optima then the analytic center of $\calX^\star$ can furnish such information, see Sonnevend \cite{son1}, also \cite{nesterov1994interior}.  Property \ref{texas primal} essentially states that $\thetax$ cannot be too large if (i) the radius of the optimal solution set of \eqref{pro: general primal LP} is not too large, and (ii) there is a feasible solution $x$ that is not too close to the boundary of $\mathbb{R}^n_+$ and not too far from the optimal solution set. Such a solution is related to the concept of a ``reliable solution'' in \cite{EpeFre00}. Similar results also hold for $\thetas$. This result is formally stated as the first assertion of Proposition \ref{cusco} in Section \ref{sec error ratio}, where we will also show that the optimization problem in \eqref{rrr2d2 primal} can be reformulated as a convex conic optimization problem with $m$ linear equalities, $n+1$ linear inequalities, and one second-order cone.

Let us now consider the computation of \LPsharp~$\mu_p$.  Not surprisingly, there is a nice polyhedral characterization of \LPsharp~$\mu_p$ that enables its computation, which we develop in Section~\ref{sec sharpness}.  If the optimal solution set is a singleton, then the LP sharpness $\mu_p$ is the smallest objective function growth rate along all of the edges of $\calF_p$ emanating from $x^\star$, which can be computed easily if the number of such edges is not excessive.  For  LP problems with multiple optimal solutions, computing $\mu_p$ requires computing the smallest sharpness along all edges of $\calF$ that intersect $\calX^\star$, which might be more challenging.  A similar approach applies to computing $\mu_d$.  For details see Section~\ref{sec sharpness}.

\paragraph{Relation to other condition numbers.}
Especially since the condition measures \limitinger~and \LPsharp~play the central role in Main Result \ref{main result}, it is useful to understand how they may be related to other more traditional condition measures for LP, such as Renegar's data-perturbation condition numbers \cite{renegar1994some}.  It turns out that the \limitinger~is upper-bounded by (and hence is tighter than) a simple quantity involving the data-perturbation condition number of Renegar \cite{renegar1994some}, see Corollary \ref{cor: theta with infeasibility}.  We also show that \LPsharp~$\mu_p$ is related to the stability of $\calX^\star$ under perturbation of the objective function vector $c$; in fact $\mu_p$ is equal to the least-norm relative perturbation $\Delta c$ of $c$ for which the new optimal solution set is not a subset of existing optimal solution set, see Theorem \ref{thm: sharpness and perturbation}. Similar arguments also hold for $\mu_d$.

\paragraph{Invariance under simple scalar rescaling.} It has been observed in practice that, with proper choice of step-sizes, rPDHG's performance is invariant under the scalar rescaling $(\alpha A, \beta b,\gamma c)$ of the LP instance data $(A,b,c)$ for $\alpha, \beta, \gamma >0$, see \cite{applegate2021practical}. Notice that this observation is in synch with the iteration bound in Main Result \ref{main result}. To see this, notice that the quantity $\calN$ in \eqref{eq:def_N}  is invariant under scalar rescaling, since in particular each of the condition measures -- $\mu_p$, $\mu_d$, $\thetax$, and $\thetas$ is invariant under the rescaling, as is the matrix condition number $\kappa$. Likewise, because $\frac{\|x^\star\|}{\|b\|_Q}$ and $\frac{\|s^\star\|}{\|c\|}$ can be interpreted as the relative distances to optima, they are also invariant under the rescaling.  Curiously, the quantity $\calD$ (which only appears inside the logarithm term) is not invariant under rescaling; while the first term of $\calD$ (which is $\kappa$) is scale invariant, the second term is always at least $1$ and setting $\alpha = \beta = \gamma = \|b\|_Q/\|c\|$ results in $\calD = \kappa$.

\vspace{15pt} 

In addition to the above desirable features of the iteration bound in Main Result \ref{main result}, the bound also suggests ways to think about practical enhancements of rPDHG to improve performance, in particular row-preconditioning of $A$ as well as tuning the step-sizes $\tau$ and $\sigma$, which we now discuss.
\vspace{-5pt} 

\paragraph{Row-preconditioning of $A$.} 
 It has been observed in practice that heuristic row- and column-preconditioning of $A$ can improve the practical performance of rPDHG, see \cite{applegate2021practical}. The iteration bound in Main Result \ref{main result} provides a theoretical justification of the value of row-preconditioning as follows. Observe from Main Result \ref{main result} that the iteration bound is at least linear in the matrix condition number $\kappa$ which appears outside the logarithm term (and inside the logarithm term as well).  Now consider the row-preconditioned system $HAx = Hb$ for some rank-$m$ matrix $H$.  Replacing $(A,b)$ with $(HA,Hb)$ does not change the geometry of the primal $x$ or dual $s$ variables, and so leaves $\mu_p$, $\mu_d$, $\thetax$, and $\thetas$ invariant.  However, it does change the value of $\kappa$, and thus heuristics to compute $H$ that will reduce the value of $\kappa$ will have the effect of reducing the theoretical iteration bound in Main Result \ref{main result}.  Therefore, row-preconditioning of $A$ is a natural way to improve the performance of the algorithm -- at least in theory. Furthermore, in Section \ref{sec experiments} we explore and confirm the practical effect of row-preconditioning.

\paragraph{Tuning the ratio of primal and dual step-sizes.}
The theory for PDHG is premised on the primal and dual step-sizes $\tau$ and $\sigma$ satisfying $\tau \cdot \sigma \le (\sigma_{\max} ^+(A))^{-2}$ where $\sigma_{\max} ^+(A)$ is the largest positive singular value of $A$, see \cite{chambolle2011first,applegate2023faster}.  However, there is leeway in the ratio of the stepsizes; notice that for $\gamma>0$ if we replace $(\tau,\sigma)\rightarrow (\gamma \tau, \sigma/\gamma) $ then the product $\tau \cdot \sigma$ is unchanged but the stepsize ratio $\tau/\sigma$ changes by $\gamma^2$.  Furthermore, it has been observed in practice that tuning the ratio $\tau/\sigma$ can significantly improve the performance of rPDHG, see \cite{applegate2021practical,applegate2023faster}. Our analysis points to theoretical benefits in the iteration bound of rPDHG if the stepsize ratio is tuned in a special way.  In Theorem \ref{thm special step size complexity} in Section \ref{restarts} we show that a specially chosen step-size ratio leads to an iteration bound with a similar structure as in \eqref{eq overall complexity intro} but with $\calN$ and $\calD$ replaced by: 

\begin{equation}\label{eq smart N intro}
	\widehat{\condN} := \kappa \cdot \left(\frac{\thetax }{\mu_p}  +\frac{\thetas}{\mu_d } +\frac{\|x^\star\|}{\mu_d\|b\|_{Q}}+\frac{\|s^\star\|}{\mu_p\|c\|}\right) \ 
\end{equation} and
\begin{equation}\label{eq smart D intro}
	\widehat{\calD} := 32e\cdot \kappa \cdot \max\left\{\frac{\mu_p\|c\|}{\mu_d\|b\|_{Q}}, \frac{\mu_d\|b\|_{Q}}{\mu_p\|c\|}\right\} \ .
\end{equation} 
Since the quantity  $\widehat{\calD}$ only appears in the logarithm term, let us focus on the quantity $\widehat{\calN}$ as compared to ${\calN}$.  Notice that $\widehat{\calN}$ contains fewer cross-terms involving $\mu_p$, $\mu_d$, $\thetax$, $\thetas$, $\frac{\|x^\star\|}{\|b\|_{Q}}$ and $\frac{\|s^\star\|}{\|c\|}$, and so has the potential to be significantly smaller than $\calN$.  This lends theoretical credence to the value of tuning the step-size ratio in practical implementations of rPDHG. Indeed, we confirm the benefit of this strategy in our experiments in Section~\ref{sec experiments}.  The special stepsize formula that leads to this theoretical improvement is presented in equation \eqref{eq smart step size 2} in Theorem \ref{thm special step size complexity}.  Notice that the stepsize formula in \eqref{eq smart step size 2} involves the LP sharpness quantities $\mu_p$ and $\mu_d$.  Unfortunately, these two quantities are typically not known {\em a priori} nor are they easy to estimate, and for this reason the improved iteration bound involving $\widehat{\calN}$ and $\widehat{\calD}$ is essentially just theoretical in nature.  Nevertheless, the improved bound points to the usefulness of heuristically tuning the step-size ratio in practical implementations of PDHG.

We end this section with a review of related works for large-scale LP.


\subsection{Related works for large-scale LP}

In addition to \cite{applegate2021practical,applegate2023faster} discussed earlier, there are several other investigations and analyses of the performance of PDHG and its variants for solving LP problems.
\cite{lu2025geometry} studies PDHG (without restarts) applied to LP instances, and uncovers a two-phase behavior of PDHG: the initial phase is characterized by sublinear convergence, followed by a second phase with linear convergence. The linear convergence of the latter phase is upper bounded using the Hoffman constant of a reduced linear system defined by the limiting point of the algorithm's trajectory.  The duration of the initial phase inversely depends on the smallest nonzero of the limiting point.  
\cite{lu2021nearly} introduces a stochastic variant of PDHG for solving LPs. \cite{applegate2021infeasibility} studies how to use PDHG for detecting infeasible LP instances, and \cite{lu2022infimal} shows that the PDHG without restarts also achieves linear convergence on LPs, though at a slower rate compared to rPDHG.
\cite{hinder2024worst} shows that the rPDHG has polynomial-time complexity for totally unimodular LPs, and 
\cite{lu2025practical} proposes to solve convex QP using PDHG-based methods.

Concurrent with our own efforts in revising the present work, several more papers on PDHG have been posted on arXiv. From a computational perspective, recent efforts include improved implementations of PDHG-based LP solvers \cite{lu2023cupdlp-c}, other extensions to convex quadratic and conic optimization  \cite{huang2025restartednew,lin2024pdcs}, as well as \cite{lu2024restarted,xiong2024role,li2024pdhg}. From a theoretical perspective, recent works have provided refined analyses for special families of LPs \cite{xiong2024accessible,lu2024pdot}, average-case complexity guarantees \cite{xiong2025high}, and convex conic optimization extensions \cite{xiong2024role}. 

In addition to PDHG, a number of other FOMs have also been studied for solving huge-scale LP instances. 
Early efforts included the steepest ascent method \cite{brown1951computational}, feasible direction methods \cite{zoutendijk1960methods}, the projected gradient algorithm \cite{lemke1961constrained}, and others.
More recently, several more practical FOM-based solvers have been proposed.
ABIP \cite{lin2021admm,deng2025enhancednew} solves conic linear programs (including linear programs) using an ADMM-based interior-point method applied to the homogeneous self-dual embedding. 
SCS \cite{o2016conic,o2021operator} employs a similar ADMM-based approach to solve the homogeneous self-dual embedding. 
OSQP \cite{osqp,osqp-gpu} uses an ADMM-based method to solve convex quadratic programs, which include LPs. 
HPR-LP \cite{chen2025hprnew} is a recently developed GPU-based solver for LP that uses the Halpern Peaceman-Rachford method with semi-proximal terms.
\cite{li2020asymptotically} proposes a semismooth Newton augmented Lagrangian method for LP problems and proves its superlinear convergence. 
ECLIPSE \cite{basu2020eclipse} is a distributed LP solver designed specifically for addressing large-scale LPs encountered in web applications.

\subsection{Notation}
For a matrix $A\in\mathbb{R}^{m\times n}$, let $\operatorname{Null}(A):=\{x\in\mathbb{R}^n:Ax = 0\}$ denote the null space of $A$ and $\operatorname{Im}(A) :=\{Ax:x\in\mathbb{R}^n\}$ denote the image of $A$. For any set $\calX\subset \mathbb{R}^n$,  let $P_\calX: \mathbb{R}^n \to \mathbb{R}^n$ denote the Euclidean projection onto $\calX$, namely, $P_\calX(x) := \arg\min_{\hat{x}\in \calX} \|x - \hat{x}\|$. Unless otherwise specified, $\|\cdot\|$ denotes the Euclidean norm. For $M \in \mathbb{S}_{+}^{n}$, the set of symmetric positive-semi-definite matrices in $\mathbb{R}^{n\times n}$, we use $\|\cdot\|_M$ to denote the semi-norm $\|z\|_M :=\sqrt{z^\top Mz}$. For any $x \in \mathbb{R}^n$ and $\calX\subset \mathbb{R}^n$, the Euclidean distance between $x$ and $\calX$ is denoted by $\dist(x,\calX):= \min_{\hat{x} \in \calX} \|x-\hat{x}\|$ and the $M$-norm distance between $x$ and $\calX$ is denoted by $\dist_M(x,\calX):= \min_{\hat{x} \in \calX} \|x-\hat{x}\|_M$. 
For simplicity of notation, we use $[n]$ to denote the set $\{1,2,\dots,n\}$. For $A \in \mathbb{R}^{n\times n}$, $A^\dagger$ denotes the Moore-Penrose inverse of $A$. For any matrix $A$,  $\sigma_{\max} ^+(A)$ and $\sigma_{\min}^+(A)$ denote the largest and smallest non-zero singular values of $A$. For an affine subset $V$, let $\linV$ denote the associated linear subspace of $V$, namely $V=\linV + v$ for every $v \in V$. 
Let $\mathbb{R}^n_+$ and $\mathbb{R}^n_{++}$ denote the nonnegative and strictly positive orthant in $\mathbb{R}^n$, respectively. Let $e$ denote the vector of ones, namely $e=(1, \ldots, 1)^\top$ whose dimension is dictated by context. For a vector $v \in \mathbb{R}^n$, $v^+$ and $v^-$ respectively denote the vector of positive parts and negative parts of $v$, i.e., the components of $v^+$ and $v^-$ are $(v^+)_i = \max\{v_i, 0\}$ and $(v^-)_i = \max\{-v_i, 0\}$ for $i \in [n]$. The operator norm $\|A\|$ of a matrix $A$ is defined as $\|A\| = \sup_{x \neq 0} \frac{\|Ax\|}{\|x\|}$.
For a symmetric matrix $A$, $A \succeq 0$ means $A\in\mathbb{S}^n_+$. 
For a linear subspace $\linV \subset \mathbb{R}^n$, $\linV^\bot$ denotes the orthogonal complement of $\linV$. 

\subsection{Organization}

The other sections of this paper are organized as follows. Section \ref{secpreliminaries} contains preliminaries for LP and presents a detailed review of PDHG. In Section \ref{restarts} we present our new computational guarantees for rPDHG for LP based on \limitinger~and \LPsharp. In Sections \ref{sec error ratio} and \ref{sec sharpness} we discuss and derive computable upper bounds for \limitinger~and a computable representation of \LPsharp, and we relate both of these condition measures to other condition numbers. Finally, in Section \ref{sec experiments} we present computational experiments that give credence to both our theoretical iteration bounds and the effectiveness of heuristic enhancements inspired by these iteration bounds.
 
\section{Preliminaries for LP and PDHG}\label{secpreliminaries}

As mentioned in Section \ref{subsec: main results}, we use the lens of focusing on the nonnegative variables $x$ of the primal and the nonnegative dual variables $s$ (the slack variables) of the dual. We first review the ``symmetric'' formulation of primal and dual LP problems from this perspective. 
 \subsection{Symmetric primal and dual formulations of LP}\label{subsec: symmetric primal dual LP}

The dual problem \eqref{pro: standard dual LP} can be formulated with explicit slack variables $s$ as follows: \begin{equation}\label{pro: general dual LP}
		 \max_{y \in \mathbb{R}^m, \ s\in\mathbb{R}^n}  \ b^\top y  \quad	\text{s.t.}  \ A^\top y + s = c, \ s\ge 0 \ .
 \end{equation}
Define $q := A^\top (AA^\top)^\dagger b$ ; then \eqref{pro: general dual LP} is equivalent to the following (dual) problem on $s$:
 \begin{equation}\label{pro: general dual LP on s}
	 \max_{s\in\mathbb{R}^n}  \ q^\top ( c- s) \quad \  \text{s.t.}  \ s \in c + \operatorname{Im}(A^\top), \ s\ge 0 \ ,
 \end{equation}
and the corresponding dual solution in the variable $y$ is any such $y$ for which $A^\top y = c-s$.
(This is because for any dual feasible solutions $y$ satisfying $A^\top y = c-s$, the corresponding objective value of $y$ is equal to $q^\top ( c- s)$: 
 $b^\top y =  q^\top A^\top y =  q^\top(c-s)$,
 where the first equality is due to $Aq = b$.)  
 
 Let $V_p := \{x \in\mathbb{R}^n: Ax = b\} = q + \operatorname{Null}(A)$ and $V_d := c + \operatorname{Im}(A^\top)$. 
 To summarize, we can rewrite the primal problem \eqref{pro: general primal LP} and the dual problem \eqref{pro: general dual LP on s} in the following symmetric formats:
 \begin{equation}\label{pro: primal dual reformulated LP}
	 \begin{aligned}
		 &\text{(P)}&\calX^\star :=\arg\min_{x\in\mathbb{R}^n}& \  c^\top x & \quad&\text{(D)}&\calS^\star :=\arg\max_{s\in\mathbb{R}^n}& \  q^\top (c - s) \\
		 &&\text{s.t.} & \ x \in \calF_p:= V_p\cap\mathbb{R}^n_+
		 &\quad & &\text{s.t.} & \ s \in \calF_d:= V_d\cap\mathbb{R}^n_+ \\
		 &&&\ \quad  \quad V_p:= q + \operatorname{Null}(A)& & & & \quad \quad \ V_d:= c + \operatorname{Im}(A^\top)
	 \end{aligned}
 \end{equation} 
 This reformulation of the dual was, to the best of our knowledge, first proposed in \cite{todd1990centered}. Here the sets of primal and dual optima are $\calX^\star$ and $\calS^\star$, respectively, and we use $\calY^\star$ to denote the corresponding optimal solutions $y$ associated with $\calS^\star$.  
 
 [From a computational perspective, none of the quantities specific to the symmetric reformulation actually need to be computed, i.e., we do not need to compute $q$ or $A^\dagger$.  We present these objects as they frame our analysis and our results.]
   
We now briefly review optimality conditions for \eqref{pro: primal dual reformulated LP}. Note that the duality gap in \eqref{pro: primal dual reformulated LP} is equal to $\gap(x,s):=c^\top x - q^\top (c - s) $, and  a solution pair $(x,s)$ is optimal for \eqref{pro: primal dual reformulated LP} if and only if the following conditions are met:
\begin{itemize}[nosep]
	\item Primal feasibility: $\dist(x,V_p) = 0$ and $\dist(x , \mathbb{R}^n_+) = 0$,
	\item Dual feasibility: $\dist(s, V_d) = 0$ and $\dist(s,\mathbb{R}^n_+) = 0$, and
	\item Nonpositive duality gap: $\gap(x,s): = c^\top x - q^\top (c - s) \le 0$.
\end{itemize}
The optimal primal-dual solution sets can be directly written as
$$
\calX^\star\times \calS^\star :=
\left\{
(x,s)
\left|
 x\in V_p, \ x \in \mathbb{R}^n_+, \  s\in V_d, \ s \in \mathbb{R}^n_+, \ \gap(x,s) \le 0
\right.
\right\} \ .
$$
In our theoretical development we will measure the error of a non-optimal pair $(x,s)$ using the distance to optima, defined as:
\begin{equation}\label{xmas}\calE_d(x,s) := \max \{ \dist(x,\calX^\star),\dist(s,\calS^\star)\} \ . \end{equation}
(The distance to optima is not conveniently computable, and so in practice it is more typical to work with the relative error defined as $\calE_r (x,y) := \frac{\|Ax^+ - b\|}{1+ \|b\|} + \frac{\|(c-A^\top y)^-\|}{1+\|c\|} + \frac{|c^\top x^+ - b^\top y|}{1 +|c^\top x^+ |+| b^\top y| }$.  It is straightforward to show that the relative error $\calE_r (x,y)$ is upper bounded by a constant factor times the distance to optima $\calE_d(x,s)$, for a constant depending only on the data $(A,b,c)$, see Remark \ref{stargaze} in Appendix \ref{phoebe}.)

It can also be observed from \eqref{pro: primal dual reformulated LP} that the linear subspaces associated with the primal and dual problems are orthogonal to each other.
Let us denote by $\linVp$ and $\linVd$ the linear subspaces associated with the affine subspaces $V_p$ and $V_d$.  Then $\linVp$ and $\linVd$ are orthogonal complements. The following fact collects some other useful properties of the symmetric formulation \eqref{pro: primal dual reformulated LP}:
\begin{fact}\label{fact: symmetric LP formulation}
	In the symmetric formulation \eqref{pro: primal dual reformulated LP}, $\linVd$ is the orthogonal complement of  $\linVp$, i.e., 
    $\linVd = \linVp^\bot$.  Furthermore, $P_{\linVp}(c) \in \linVp$ and  $P_{\linVp}(c) = \arg\min_{v\in V_d} \|v\| $, and $q \in \linVd $ and $q = \arg\min_{v \in V_p}\|v\| $.
\end{fact}

Finally, we make the following assumption about \eqref{pro: general primal LP} and its dual problem \eqref{pro: general dual LP}.
\begin{assumption}\label{assump: general LP}
	We assume that the LP problem \eqref{pro: general primal LP} has an optimal solution, and non-optimal feasible solutions exist for the duality-paired problems \eqref{pro: general primal LP} and \eqref{pro: general dual LP}, and equivalently for \eqref{pro: primal dual reformulated LP}.
\end{assumption}

\subsection{Convergence properties of PDHG (without restarts) for LP }\label{subsec: sublinear convergence}

In this subsection, we review the ergodic convergence properties of PDHG (without restarts) for LP. The ergodic convergence is also known as the convergence of the average iterate. A sublinear ergodic convergence bound has been shown for PDHG in \cite{chambolle2016ergodic,applegate2023faster} for general convex-concave saddlepoint problems, whose error is measured in terms of the ``primal-dual gap'' associated with the saddlepoint problem. We briefly review these convergence results with a focus on two issues, namely (i) the role of the primal and dual step-sizes, and (ii) the convergence (to zero) of errors for the LP problem (measured using the distances to the constraints and the duality gap) instead of gap measures of the saddlepoint problem. The material presented in this subsection will be used later in Section \ref{restarts}. Complete proofs of the results in this subsection are deferred to Appendix \ref{sec:proof preliminaries}. 

One step of PDHG for LP \eqref{pro: saddle point LP} is defined in the function \textsc{PDHGstep} in \eqref{eq: one PDHG}. Note that we use $z:=(x,y)\in\mathbb{R}^{m+n}$ to denote the pair of primal and dual solutions $x$ and $y$,  and PDHG generates iterates by $z^+ \leftarrow \textsc{PDHGstep}(z)$.  We will freely use the notation $z$ with sub/superscripts and other modifications, so that it denotes the primal and dual solutions $(x,y)$ with the same sub/superscripts and other modifications, such as $\bar{z}^{k} = (\bar{x}^{k},\bar{y}^{k})$.

The convergence guarantees for PDHG rely on the step-sizes $\tau$ and $\sigma$ in \eqref{eq: one PDHG} being sufficiently small.  In particular, if the following condition is satisfied:
\begin{equation}\label{robsummer}
	M:=	\begin{pmatrix}
		\frac{1}{\tau}I_n & -A^\top \\
		-A & \frac{1}{\sigma}I_m
	\end{pmatrix}  \succeq 0 \ , 
\end{equation}
then PDHG's average iterates will converge to a saddlepoint of the problem \eqref{pro: saddle point LP}  \cite{chambolle2011first,chambolle2016ergodic}, though the guaranteed rate of convergence is sublinear. The requirement \eqref{robsummer} is equivalently written as:
\begin{equation}\label{eq  general step size requirement}
	\tau > 0, \ \sigma >0, \ \text{ and } \tau\sigma \le  \left( \frac{1}{ \sigma_{\max}^+(A) } \right)^2 \ ,
\end{equation}
where $\sigma_{\max}^+(A)$ is the largest positive singular value of $A$.  
The matrix $M$ defined in \eqref{robsummer} turns out to be particularly useful in analyzing the convergence of PDHG through its induced inner product norm defined by $\| z \|_M := \sqrt{z^\top M z}$ (though it is only a semi-norm if $M$ is not positive definite), which will be used extensively in the rest of this paper. 

For notational convenience let us define:
\begin{equation}\label{eq  def lamdab min max}
	\lambda_{\max}:= \sigma_{\max}^+\left(A \right)\text{, }	\lambda_{\min}:= \sigma_{\min}^+\left(A \right), \text{ and }\kappa := \frac{\lambda_{\max}}{\lambda_{\min}}  \ .
\end{equation} Here $\kappa$ is the standard condition number of the matrix $A$. 
To measure the convergence of PDHG for LP, \cite{applegate2023faster} introduces the ``normalized duality gap,'' which we have briefly reviewed in \eqref{fearofheights_intro}, and whose formal definition is as follows:

\begin{definition}[Normalized duality gap, (4a) in \cite{applegate2023faster}]\label{fri}
	For any $z = (x,y)\in \mathbb{R}^n_+\times \mathbb{R}^m$ and $r > 0$, define
	$
	\widetilde{B}(r;z) := \{\hat{z} := (\hat{x},\hat{y}):  \hat{x}\ge 0 \text{ and } \|\hat{z} -z\|_M \le r   \}
	$.
	The normalized duality gap of the saddlepoint problem \eqref{pro: saddle point LP} is then defined as
	\begin{equation}\label{fearofheights}
		\rho(r;z) := \left(\frac{1}{r}\right)\max_{  \hat{z} \in \widetilde{B}(r;z) }  \big[ L(x,\hat{y}) - L(\hat{x},y) \big] \ .
	\end{equation}
\end{definition}
\noindent
Note in Definition \ref{fri} that $\widetilde{B}(r;z)$ is technically not a ball in the usual sense of the term, since the requirement that $\hat x \ge0$ means that $\widetilde{B}(r;z)$ is not necessarily symmetric relative to its center $z$. The normalized duality gap serves as an upper bound for evaluating error tolerances and distances to optimality, see \cite{applegate2023faster}, which also presents an efficient algorithm for approximating $\rho(r;z)$. 

The following lemma shows that the normalized duality gap provides an upper bound on both the distances to feasibility and the magnitude of the duality gap; this lemma is a variation of \cite[Lemma 4]{applegate2023faster} but measures distances instead of error tolerances. 

\begin{lemma}\label{lm: convergence of PHDG without restart}
	For any $r > 0$, $\bar{z} :=(\bar{x},\bar{y})$ such that $\bar{x} \ge 0$, and $\bar{s} := c - A^\top \bar{y}$,  the normalized duality gap $\rho(r; \bar z)$ provides the following bounds:
	\begin{enumerate}[nosep]
		\item Primal near-feasibility: $ \dist(\bar{x},V_p)  \le \frac{1}{\sqrt{\sigma} \lambda_{\min}}\cdot \rho(r;\bar{z})$ and $\dist(\bar{x},\mathbb{R}^n_+)  = 0$, 
		\item Dual near-feasibility: $\dist(\bar{s},V_d) = 0 $ and $\dist(\bar{s},\mathbb{R}^n_+) \le \frac{1}{\sqrt{\tau}} \cdot \rho(r;\bar{z}) $, and 
		\item Duality gap: $\gap(\bar{x},\bar{s})  \le   \max\{ r, \|\bar{z}\|_M\} \rho(r;\bar{z})$.
	\end{enumerate}
\end{lemma}

A proof of Lemma \ref{lm: convergence of PHDG without restart} as well as other results in this section are given in Appendix \ref{guitar}. Let the $k$-th iterate of PDHG be denoted as $z^k:=(x^k,y^k)$ for $k=0,1,\ldots$, and let the average of the first $K$ iterates be denoted as  $\bar z^K = (\bar{x}^K,\bar{y}^K)  := \frac{1}{K}\sum_{i=1}^K (x^i,y^i)$ for $K \ge 1$.
The following lemma presents the sublinear convergence of the average iterates of PDHG for the saddlepoint LP formulation \eqref{pro: saddle point LP} in terms of the normalized duality gap.

\begin{lemma}\label{lm: rho sublinear PDHG}
	Suppose that  $\tau$ and $\sigma$ satisfy \eqref{eq  general step size requirement}. Then for all $K \ge 1$ it holds that
		\begin{equation}
			\rho(\|\bar{z}^K-z^0\|_M;\bar{z}^K) \le \frac{4\dist_M(z^0,\calZ^\star)}{K} \ .
		\end{equation}
\end{lemma}

We remark that Lemma \ref{lm: rho sublinear PDHG} can be viewed as an extension of Property 3 of \cite{applegate2023faster} to the case of different primal and dual stepsizes $\tau$ and $\sigma$. By combining Lemma \ref{lm: convergence of PHDG without restart} and Lemma \ref{lm: rho sublinear PDHG} and changing the norm, we obtain the following theorem regarding the ergodic behavior of PDHG in terms of distances to constraints and the duality gap.
\begin{theorem}\label{thm rho sublinear PDHG to LP natural} Suppose that  $\tau$ and $\sigma$ satisfy \eqref{eq  general step size requirement}, and suppose that PDHG is initiated with $z^0 = (x^0,y^0) := (0,0)$. Then for any $K \ge 1$ and $\bar{x}^K := \frac{1}{K}\sum_{i=1}^K x^i$ and $\bar{s}^K := \frac{1}{K} \sum_{i=1}^K (c - A^\top y^i)$, the following hold:
	\begin{enumerate}[nosep]
		\item Primal near-feasibility: 
		$ \dist(\bar{x}^K,V_p)  \le \frac{4\sqrt{2}}{K}
		{\left(\frac{\dist(0,\calX^\star)}{\sqrt{\sigma\tau}\lambda_{\min}}+\frac{\dist(c,\calS^\star)}{\sigma\lambda_{\min}^2}  \right)}$  and $\dist(\bar{x}^K,\mathbb{R}^n_+) = 0$ , 
		\item Dual near-feasibility: $\dist(\bar{s}^K,V_d)  = 0$ and 
		$\dist(\bar{s}^K,\mathbb{R}^n_+) \le 
		\frac{4\sqrt{2}}{K}
		\left(\frac{\dist(0,\calX^\star)}{\tau}+\frac{\dist(c,\calS^\star)}{\sqrt{\sigma\tau}\lambda_{\min}}  \right) $ , and
		\item Duality gap: 
		$\gap(\bar{x}^K, \bar{s}^K)  \le
		\frac{16}{K}\left(
			 \frac{\dist(0,\calX^\star)^2}{\tau}
			+ \frac{\dist(c,\calS^\star)^2}{\sigma\lambda_{\min}^2} 
			+ \frac{2\dist(0,\calX^\star)\dist(c,\calS^\star)}{\sqrt{\sigma\tau}\lambda_{\min}}
		\right) $ . 
	\end{enumerate} 
\end{theorem} 

Theorem \ref{thm rho sublinear PDHG to LP natural} states that the distances to the constraints and the duality gap of the average iterates of PDHG converge to zero at a sublinear rate.  We note the roles of the primal and dual step-size $\tau$ and $\sigma$:  generally speaking the larger $\tau$ is, the faster $\bar{s}^K$ converges to the dual feasible set, while the larger $\sigma$ is, the faster $\bar{x}^K$ converges to the primal feasible set. However, a balanced choice of $\tau$ and $\sigma$ is required (at least in theory) for faster convergence of the duality gap of $(\bar{x}^K,\bar{s}^K)$. \ignore{Note that while the product of the step-sizes $\tau \sigma$ is the relevant quantity in the conditions for convergence in \eqref{eq general step size requirement}, the ratio between $\tau$ and $\sigma$ can also play a significant role in balancing the convergence of the three different near-optimality conditions. } Items ({\em 1.}) and ({\em 2.}) of Theorem \ref{thm rho sublinear PDHG to LP natural} follow directly using Lemma \ref{lm: rho sublinear PDHG} and Lemma \ref{lm: convergence of PHDG without restart} plus a norm inequality relating distances in the $M$-norm to Euclidean distances (Proposition \ref{fact: norm upperbound}). 

Below we present two lemmas that are related to the performance of PDHG and are used in the proof of Theorem \ref{thm rho sublinear PDHG to LP natural}.  The first is a well-known general nonexpansive property for PDHG (and certain other methods as well), see \cite{ryu2022large,applegate2023faster,chambolle2011first}.

\begin{lemma}[Nonexpansive property of PDHG, see \cite{applegate2023faster}]\label{lm: nonexpansive property} Suppose that  $\tau,\sigma$ satisfy \eqref{eq  general step size requirement}. For any saddlepoint $z^*$ of \eqref{pro: saddle point LP}, and for all $k\ge0$, 
	\begin{equation}\label{eq nonexpansive property}
		\| z^{k+1} -z^\star \|_M \le  \|z^k  - z^\star \|_M \ .
	\end{equation}
	Therefore under the assignment $z:= z^k$ or $z := \bar z^k$ it holds that $
		\left\| z -z^\star\right\|_M \le \left\|z^0  - z^\star\right\|_M$.
\end{lemma}

\begin{lemma}\label{lm: R in the opt gap convnergence} Suppose $z^a$, $z^b$, and $z^c$ satisfy the nonexpansive properties:  $\|z^b - z^\star\|_M \le \| z^{a} - z^\star\|_M$ and $\|z^c - z^\star\|_M \le \| z^{a} - z^\star\|_M$ for every $z^\star \in \calZ^\star$. Then 
	\begin{equation}\label{eq of lm: R in the opt gap convnergence}
		\begin{aligned}
			\max\{ \|z^b - z^c\|_M, \|z^b\|_M\}  \le   & \ 2 \dist_M(z^a,\calZ^\star)  + \|z^{a} \|_M	\ .
		\end{aligned}
	\end{equation}
\end{lemma}
\noindent
\noindent Lemma \ref{lm: nonexpansive property} is a simple extension of Proposition 2 \cite{applegate2023faster} to the setting of different primal and dual step sizes. It can be proved just as in \cite{applegate2023faster} by directly substituting the identical primal and dual step size $\eta$ of \cite{applegate2023faster} with $\tau$ and $\sigma$, respectively. Inequality \eqref{eq nonexpansive property} is known as the nonexpansive property and appears in many operator splitting and other related methods \cite{liang2016convergence,ryu2022large,applegate2023faster}.

\ignore{In \cite{applegate2023faster} it has been proved that directly using Lemma \ref{lm: nonexpansive property} and the classic saddlepoint problem convergence result of PDHG (such as Remark 2 in \cite{chambolle2016ergodic}) proves Lemma \ref{lm: rho sublinear PDHG}.} 
\ignore{Lemma \ref{lm: R in the opt gap convnergence} is frequently used to ensure the $\max\{r,\|\bar{z}\|_M\}$ in the third term of Lemma \ref{lm: convergence of PHDG without restart} does not explode so that the duality gap can be bounded by the normalized duality gap $\rho(r;\bar{z})$. It is used in proving the third item of Theorem \ref{thm rho sublinear PDHG to LP natural} as well as in the and later the main results in Section \ref{restarts}.}

We end this section with a proposition on the relation between the $M$-norm used in PDHG and the Euclidean norm. 
While the $M$-norm arises naturally in the analysis of PDHG, it has the disadvantage that its quadratic form is not separable in $x$ and $y$ and also it implicitly includes the step-sizes $\tau$ and $\sigma$ in its definition. Towards the goal of stating results in terms of the Euclidean norms on $x$ and $y$, we introduce the following $N$-norm on $(x,y) \in \mathbb{R}^{m+n}$, whose quadratic form is separable in $x$ and $y$.  Define:
$$
\left\|(x,y)\right\|_N := \sqrt{\frac{1}{\tau} \|x\|^2 + \frac{1}{\sigma} \|y\|^2} \ \text{ where } \
N:= 	\begin{pmatrix}
	\frac{1}{\tau}I_n & \\
	& \frac{1}{\sigma}I_m
\end{pmatrix} \ . 
$$
In comparison with the $M$-norm, the $N$-norm offers advantages in both computation and analysis. Furthermore, when $\tau$ and $\sigma$ are sufficiently small, the $M$-norm and $N$-norm are equivalent up to well-specified constants related to $\tau$ and $\sigma$, as follows.

\begin{proposition}\label{fact: norm upperbound} Suppose that  $\sigma, \tau$ satisfy \eqref{eq  general step size requirement}. Then for any $z=(x,y) \in \mathbb{R}^{m+n}$, it holds that 

\begin{equation}\label{silly}\sqrt{1 - \sqrt{\tau\sigma}\lambda_{\max}}  \cdot \|z\|_{N}  \le 	\|z\|_M \le 	 \sqrt{2}  \|z\|_{N} \ . \end{equation}
Furthermore, if $x \in \mathbb{R}^n_+$ and $s:= c - A^\top y \in\mathbb{R}^n$ then 
\begin{align}
&		\dist_M(z,\calZ^\star) \le  \frac{\sqrt{2}}{\sqrt{\tau}} \dist(x,\calX^\star) + \frac{\sqrt{2}}{\sqrt{\sigma}\lambda_{\min}} \dist(s,\calS^\star)  \, , \ \text{ and } \label{eq of lm: M norm to seperable norm} \\
&	\dist_M(z,\calZ^\star) \ge  \sqrt{1 - \sqrt{\tau\sigma}\lambda_{\max}} \cdot \max\left\{
\frac{1}{\sqrt{\tau}}\dist(x,\calX^\star), \frac{1}{\sqrt{\sigma}\lambda_{\max}} \dist(s,\calS^\star)
\right\} \label{eq of lm: M norm to seperable norm 2} \ .
\end{align} 
\end{proposition}  

\noindent Appendix \ref{guitar} presents proofs of the (uncited) results in this section.

\section{Computational Guarantees for PDHG with Restarts}\label{restarts}

In this section we formally state and prove the computational guarantees that we previewed and discussed in Section \ref{subsec: main results}.  In Section \ref{subsec: two results} we present our new computational guarantees in two theorems: Theorem \ref{thm overall complexity} is a more formal statement of Main Result \ref{main result} previewed in Section \ref{subsec: main results}, and Theorem \ref{thm special step size complexity} is a formalization of the improved bound previewed in \eqref{eq smart N intro} obtained for a very special choice of stepsizes $\tau$ and $\sigma$ that are primarily of theoretical interest since they are not efficiently computable. Section \ref{subsec: proof of overall complexity} is dedicated to the proofs of our key results, and Section \ref{subsec: proof of Mdistance upper bounded by normalized duality gap} is comprised of statements and proofs of supporting lemmas.

Building on the basic step of PDHG described in \eqref{eq: one PDHG}, Algorithm \ref{alg: PDHG with restarts} formally presents the framework of PDHG with restarts that we will work with. 
\begin{algorithm}[htbp]
	\SetAlgoLined
	{\bf Input:} Initial iterate $z^{0,0}:=(x^{0,0}, y^{0,0})$, $n \gets 0$ \;
	\Repeat{\text{Either $z^{n,0}$ is a saddlepoint or $z^{n,0}$ satisfies some other convergence condition }}{
		\textbf{initialize the inner loop:} inner loop counter $k\gets 0$ \;
		\Repeat{Some (verifiable) restart condition is satisfied by $\bar{z}^{n,k}$ \ }{
			\textbf{conduct one step of PDHG: }$z^{n,k+1} \gets \textsc{PDHGstep}(z^{n,k})$ \;
			\textbf{compute the average iterate in the inner loop. }$\bar{z}^{n,k+1}\gets\frac{1}{k+1} \sum_{i=1}^{k+1} z^{n,i}$
			\label{line:average} \;  \label{line:output-is-average-of-iterates} 
			$k\gets k+1$ \;
		}
		\textbf{restart the outer loop:} $z^{n+1,0}\gets \bar{z}^{n,k}$, $n\gets n+1$ \;
	}
	{\bf Output:} $z^{n,0}$ ($ \ = (x^{n,0},  y^{n,0})$)
	\caption{PDHG with general restart scheme (rPDHG)}\label{alg: PDHG with restarts}
\end{algorithm}
Recall that we use the notation $z^{k+1} \gets \textsc{PDHGstep}(z^k)$ to denote an iteration of PDHG as described in  \eqref{eq: one PDHG}. The double superscript on the variable $z^{n,k}$ indexes the outer iteration counter followed by the inner iteration counter, so that $z^{n,k}$ is the $k$-th inner iteration of the $n$-th outer loop.  In order to implement Algorithm \ref{alg: PDHG with restarts} it is necessary to specify a (verifiable) restart condition on the average iterate $\bar z^{n,k}$ in line $\bf 8$ that is used to determine when to restart PDHG. We will primarily consider Algorithm \ref{alg: PDHG with restarts} using the following restart condition in line $\bf 8$:
\begin{equation}\label{catsdogs}\rho(\|\bar{z}^{n,k} - z^{n,0}\|_M; \bar{z}^{n,k}) \le \beta \cdot \rho(\|z^{n,0} - z^{n-1,0}\|_M; z^{n,0}) \ , \end{equation} for a specific value of $\beta \in (0,1)$ (in fact we will use $\beta = 1/e$ where $e$ is the base of the natural logarithm).  In this way \eqref{catsdogs} is nearly identical to the condition used in \cite{applegate2023faster}. The condition \eqref{catsdogs} essentially states that the normalized duality gap shrinks by the factor $\beta$ between restart values $\bar{z}^{n,k}$ and ${z}^{n,0}$. One of the reasons for using condition \eqref{catsdogs} is that the normalized duality gap can be easily approximated, see \cite[Section 6]{applegate2023faster}.

\subsection{Computational guarantees for rPDHG based on \limitinger~and \LPsharp}\label{subsec: two results}

In this section we present our main computational guarantee for rPDHG. This computational guarantee relies on the two condition measures $\thetax$ and $\mu_p$ for the primal and on their counterparts $\thetas$ and $\mu_d$ for the dual problem. Recall that we formally defined $\thetax$ and $\mu_p$ for the primal problem in Definitions \ref{def:limitinger} and \ref{def: sharpness mu}.  Now that we have established the symmetric representation of the primal and the dual problems in \eqref{pro: primal dual reformulated LP}, we can write the formal descriptions of $\thetas$ and $\mu_d$ for the dual problem.  Similar to \eqref{overcast}, for $s\in V_d\setminus \calF_d$ we define $\theta_d(s) :=\frac{\dist(s, \calF_d)}{\dist(s,\mathbb{R}^n_+)}$, and then define 
\begin{equation}\label{eq:cond_for_dual}
	\thetas :=   \lim_{\eps \to 0} \left( \sup_{s\in V_d, \hspace{.05cm}\dist(s, \calS^\star) \le \eps} \theta_d(s) \right)
		\ , \quad\text{ and }\quad \mu_d  = \inf_{s \in \calF_d\setminus\calS^\star} \frac{\dist(s,V_d\cap H_d^\star   )}{\dist(s,\ \calS^\star)} 
	\ 
\end{equation}
where $H_d^\star$ denotes the optimal objective hyperplane for the dual problem, namely $H_d^\star:=\{\hat{s}\in\mathbb{R}^n_+
:q^\top(c-s) = q^\top (c - s^\star)\}$ for any $s^\star \in \calS^\star$.

We consider running rPDHG (Algorithm \ref{alg: PDHG with restarts}) using the following simple restart condition:
\begin{definition}[$\beta$-restart condition]\label{def beta restart}
	For a given $\beta \in(0,1)$, the iteration $(n,k)$ satisfies the \textit{$\beta$-restart condition} if $n \ge 1$ and condition \eqref{catsdogs} is satisfied, or $n=0$ and $k=1$.
\end{definition}
\noindent In the following theorem we suppose for simplicity and ease of exposition that $c \in \linVp$.    
Note that under Assumption \ref{assump: general LP} it must hold that $\|c\| \ne 0$, since otherwise all feasible solutions for \eqref{pro: general primal LP} would have the same objective value, contradicting the existence of non-optimal feasible solutions.  Similarly for the dual it must hold that $\|q\| \ne 0$.
Also recall the definitions of $\lambda_{\max}$, $\lambda_{\min}$, and $\kappa$ from \eqref{eq  def lamdab min max}. 
\begin{theorem}\label{thm overall complexity}
	Suppose $c \in \linVp$, and that Assumption \ref{assump: general LP} holds, and that Algorithm \ref{alg: PDHG with restarts} (rPDHG) is run starting from $z^{0,0} = (x^{0,0},y^{0,0} )= (0,0)$ using the $\beta$-restart condition with $\beta := 1/e$. Furthermore, let the step-sizes be chosen as follows:
	\begin{equation}\label{eq smart step size 1}
		\tau = \frac{\|q\|}{2\kappa\|c\|   } \ \ \text{ and } \ \ \sigma = \frac{\|c\|}{2\|q\|\lambda_{\max}\lambda_{\min}} \ .
	\end{equation}
Let $T$ be the total number of \textsc{PDHGstep} iterations that are run in order to obtain $n$ for which $(x^{n,0},s^{n,0})$ satisfies $\calE_d(x^{n,0},s^{n,0}) \le \eps$.  Then 

\begin{equation}\label{eq overall complexity}
		T \ \le \ 5 e \cdot  \condN \cdot \ln\Bigg(   
		 \condN \calD \cdot \frac{\calE_d(x^{0,0},s^{0,0} )}{ \eps}
		\Bigg) + 1\ ,
	\end{equation}
	where $\calD$ is defined in \eqref{eq:def_D} and $\condN$ is defined as follows:
	\begin{equation}\label{eq easy L}
		\condN := 8.5 \kappa \left(
		\frac{1}{\mu_p} + \frac{1}{\mu_d} \right)
		\Bigg(  
		\thetax +\thetas   +  \frac{ \dist(0,\calX^\star) }{\dist(0, V_p)} +  \frac{\dist(c,\calS^\star)}{ \dist(0,V_d)} 
		\Bigg) \ .
	\end{equation}
\end{theorem}

In Section \ref{subsec: comments} we discussed key points and comments about Theorem \ref{thm overall complexity}.  Here we present some detailed technical remarks. The optimality tolerance criterion used in the theorem is $\calE_d(x^{n,0},s^{n,0})$ (see \eqref{xmas}), which is the distance to optima of the primal and dual (slack) variable $(x^{n,0}, s^{n,0})$. Note that $\dist(s^{n,0},\calS^\star)$ is equal to $\dist_{AA^\top}(y^{n,0},\calY^\star)$ -- the distance of $y^{n,0}$ to $\calY^\star$ under the $AA^\top$-norm. The prescribed step-sizes in \eqref{eq smart step size 1} are relatively easy to compute as long as estimates of the largest and smallest positive singular values of $A$ are easy to compute. In particular neither the \LPsharp~nor the \limitinger~are needed in the computation of the step-sizes. See Remark \ref{rmk:talk about other parameters} for extensions of the theorem to other step-sizes, different values of $\beta$, and relaxations of the assumption that $c \in \linVp$. The proof of Theorem \ref{thm overall complexity} is presented in Section \ref{subsec: proof of overall complexity}, and the proofs of the supporting technical lemmas are in Section \ref{subsec: proof of Mdistance upper bounded by normalized duality gap} and Appendix \ref{appsec: proof of restarts}.

Let us now show that the expressions \eqref{eq overall complexity} and \eqref{eq easy L}  are equivalent to \eqref{eq overall complexity intro} and \eqref{eq:def_N} up to absolute constant factors, thus validating that Main Result \ref{main result} is the same as Theorem \ref{thm overall complexity}. Under the condition that $c \in \linVp$ in the theorem it follows that $\dist(c,\calS^\star) \le \|c\| + \dist(0,\calS^\star) \le 2\dist(0,\calS^\star)$, and therefore $ \frac{\dist(c,\calS^\star)}{ \dist(0,V_d)} \le  2\frac{\dist(0,\calS^\star)}{ \dist(0,V_d)}  = 2\frac{\dist(0,\calS^\star)}{ \|c\|}$.  
Also it follows from Fact \ref{fact: symmetric LP formulation} that $\dist(0, V_p) = \|q\| = \|b\|_Q$ where $\|b\|_Q$ denotes $\|A^\top (AA^\top)^\dagger b\|$.

The step-sizes \eqref{eq smart step size 1} in Theorem \ref{thm overall complexity} are valid step-sizes that are typically relatively easy to compute.  It has been observed in practice that heuristically tuning the ratio $\tau/\sigma$ can significantly improve the performance of PDHG, see \cite{applegate2021practical,applegate2023faster}. Theorem \ref{thm special step size complexity} below shows that a specially chosen step-size ratio (that is unfortunately not easy to compute in practice) results in an iteration bound that is potentially much smaller than that of Theorem \ref{thm overall complexity}.

\begin{theorem}\label{thm special step size complexity} Suppose $c \in \linVp$, and that Assumption \ref{assump: general LP} holds, and that Algorithm \ref{alg: PDHG with restarts} (rPDHG) is run starting from $z^{0,0} = (x^{0,0},y^{0,0} )= (0,0)$ using the $\beta$-restart condition with $\beta := 1/e$. Let the step-sizes be chosen as follows:
	\begin{equation}\label{eq smart step size 2}
		\tau = \frac{\mu_d\|q\|}{2\kappa\mu_p \|c\|   } \ \  \text{ and } \ \ \sigma = \frac{\mu_p\| c\|}{2\mu_d\|q\|\lambda_{\max}\lambda_{\min}} \ , 
		 \ 
	\end{equation}
and let $T$ be the total number of \textnormal{\textsc{PDHGstep}} iterations that are run in order to obtain $n$ for which $(x^{n,0},s^{n,0})$ satisfies $\calE_d(x^{n,0},s^{n,0}) \le \eps$.  Then 

	\begin{equation}\label{eq overall complexity2}
		T \ \le \ 5 e \cdot  \widehat{\condN} \cdot \ln\Bigg(   
		 \widehat{\condN} \widehat{\calD} \cdot \frac{\calE_d(x^{0,0},s^{0,0} )}{ \eps}
		\Bigg) + 1\ ,
	\end{equation}
where $ \widehat{\calD} $ is defined in \eqref{eq smart D intro} and $ \widehat{\condN} $ is defined as follows:
	\begin{equation}\label{eq smart L}
		 \widehat{\condN}  := 16 \kappa  \left(
		\frac{\thetax }{\mu_p}  +\frac{\thetas}{\mu_d } +
		\frac{\dist(0,\calX^\star) }{\mu_d\cdot \dist(0, V_p)} +
		\frac{ \dist(c,\calS^\star)}{\mu_p \cdot\dist(0,V_d)}    
		\right)  \ .
	\end{equation} 
\end{theorem}
Note that $\widehat{\condN}$ in \eqref{eq smart L} is potentially much smaller than \eqref{eq easy L} because there are fewer cross-terms involving $\mu_p$, $\mu_d$, $\thetax$, and $\thetas$. However, the assignment of the step-sizes \eqref{eq smart step size 2} requires knowledge of the \LPsharp~constants $\mu_p$ and $\mu_d$ (or just their ratio), which are likely to be neither known nor easily computable. Although the step-sizes in \eqref{eq smart step size 2} are not implementable in practice, Theorem \ref{thm special step size complexity} lends some theoretical justification to the observed practical value of adaptively tuning the step-sizes in practical solvers \cite{applegate2023faster,applegate2021practical}. Using identical logic as presented earlier, it is straightforward to show that the value of $\widehat{\condN}$ in \eqref{eq smart L} satisfies \eqref{eq smart N intro} up to an absolute constant factor. See the last paragraph of Section \ref{subsec: comments} for other remarks and observations about Theorem \ref{thm special step size complexity}. The proof of Theorem \ref{thm special step size complexity} is similar to that of Theorem \ref{thm overall complexity}, and is presented in Appendix \ref{appsec: proof of restarts}.

\subsection{Two lemmas, the proof of Theorem \ref{thm overall complexity}, and extensions} \label{subsec: proof of overall complexity}

Similar to the main complexity proofs in \cite{applegate2023faster}, the proof of Theorem \ref{thm overall complexity} involves two steps. The first step is to demonstrate that if the normalized duality gap of the iterates satisfies a certain sharpness condition, then rPDHG achieves linear convergence, with the total number of \textsc{PDHGstep} iterations being largely determined by this sharpness condition. The second step involves analyzing and bounding the sharpness condition using the condition measures $\mu_p$, $\mu_d$, $\thetax$, and $\thetas$ (plus other data-related quantities). Lemma \ref{lm: complexity of PDHG with adaptive restart} below embodies the first step.

\begin{lemma}\label{lm: complexity of PDHG with adaptive restart}
Suppose Algorithm \ref{alg: PDHG with restarts} is run starting from $z^{0,0} = (x^{0,0},y^{0,0} )= (0,0)$ using the $\beta$-restart condition, and let the step-sizes $\tau$ and $\sigma$ satisfy \eqref{eq  general step size requirement}. Suppose also that there exists  a scalar $\condN$ for which it holds that
	\begin{equation}\label{eq restart L condition}
		\dist_M(z^{n,0}, \calZ^\star) \ \le \ \condN  \cdot \rho(\|z^{n,0} - z^{n-1,0} \|_M;z^{n,0})  \ 
	\end{equation}
	for all $n \ge 1$. Let $T$ be the total number of \textnormal{\textsc{PDHGstep}} iterations that are run in order to obtain $n$ for which $z^{n,0}=(x^{n,0},s^{n,0})$ satisfies $\calE_d(x^{n,0},s^{n,0}) \le \eps$.  Then 
	\begin{equation}\label{drewcurse}
	T \ \le \ \frac{5}{\beta \ln(1/\beta)} \cdot   \condN  \cdot \ln\left(   
	\tilde c \cdot  \condN  \cdot \left(\frac{\calE_d(x^{0,0},s^{0,0} )}{\eps} 
	\right)\right) + 1\ ,
	\end{equation}
	where $\tilde c :=  \frac{4\sqrt{2}}{\beta\sqrt{1 - \sqrt{\tau\sigma} \lambda_{\max} } }  \left(
	\sqrt{\tau} + \sqrt{\sigma}\lambda_{\max}
	\right)\cdot\left(
	\frac{1}{\sqrt{\tau}}+ \frac{1}{\sqrt{\sigma}\lambda_{\min}}
	\right) $.
\end{lemma}

The proof of Lemma \ref{lm: complexity of PDHG with adaptive restart} is similar to the proof of \cite[Theorem 2]{applegate2023faster}, and proceeds by establishing an upper bound on the \textsc{PDHGstep} iteration count between restarts. The primary difference is that we are particularly interested in the role of the primal and dual step-sizes. The complete proof of Lemma \ref{lm: complexity of PDHG with adaptive restart} is presented in Appendix \ref{appsec: proof of restarts}.

Lemma \ref{lm: Mdistance upper bounded by normalized duality gap} embodies the second step described above.

\begin{lemma}\label{lm: Mdistance upper bounded by normalized duality gap} Suppose that the initial iterate of rPDHG is $z^{0} := (0,0)$, and let $z^b, \ z^c \in \mathbb{R}^n_+ \times \mathbb{R}^m $ satisfy $z^b \ne z^c$ and the nonexpansive inequalities $\|z^b - z^\star\|_M \le \| z^{0} - z^\star\|_M$ and $\|z^c - z^\star\|_M \le \| z^{0} - z^\star\|_M$ for all $z^\star \in \calZ^\star$. Then it holds that:
	\begin{equation}\label{eq of lm: Mdistance upper bounded by normalized duality gap} 
			\dist_M(z^b,\calZ^\star) \le  \widetilde{\condN} \cdot \rho(\|z^b - z^c\|_M;z^b) \ ,  
	\end{equation}
	in which
	\begin{equation}\label{eq of lm: value tildeN}
		 \widetilde{\condN} :=  \left(
			\begin{array}{rl}
				&	\left(
				\frac{3\sqrt{2} }{\sqrt{\sigma\tau} \lambda_{\min}\mu_p }  
				+ \frac{\sqrt{2}  }{\sigma \lambda_{\min}^2\mu_d } \cdot \frac{\| P_{\linVp} (c)\|}{\|q\|}
				\right) \cdot \left(\thetax + \frac{\| P_{\linVp^\bot} (c)\| }{\| P_{\linVp} (c)\|}\right)  \\
				&  + \left( \frac{2\sqrt{2}}{\sqrt{\sigma\tau} \lambda_{\min} \mu_d } + \frac{\sqrt{2}}{\tau\mu_p } \cdot \frac{\|q\|}{\| P_{\linVp}(c)\| }  \right) \cdot \thetas  \\
				& + \left(  \frac{4}{ \sqrt{\tau}\mu_p \| P_{\linVp}(c)\| } + \frac{4}{\sqrt{\sigma}\mu_d \|q\| \lambda_{\min}  } \right) \cdot \left(\frac{1}{\sqrt{\tau}} \dist(0,\calX^\star)  + \frac{1}{\lambda_{\min}\sqrt{\sigma}} \dist(c,\calS^\star)\right) 
			\end{array} 
			\right) \ . 
	\end{equation}
\end{lemma}

The proof of Lemma \ref{lm: Mdistance upper bounded by normalized duality gap} is presented in Section \ref{subsec: proof of Mdistance upper bounded by normalized duality gap}. Armed with Lemma \ref{lm: complexity of PDHG with adaptive restart} and Lemma \ref{lm: Mdistance upper bounded by normalized duality gap}, we now prove Theorem \ref{thm overall complexity}.

\begin{proof}[Proof of Theorem \ref{thm overall complexity}]
For any $n \ge 1$ it holds that $z^{n,0} = \bar z^{n-1,K}=\tfrac{1}{K}\sum_{i=1}^K z^{n-1,i}$ where $K$ is the total number of inner iterations of rPDHG run in the outer loop iteration $n-1$, whose initial iterate value is $z^{n-1,0}$.  It then follows from the nonexpansive properties of PDHG from Lemma \ref{lm: nonexpansive property} that $\|z^{n,0} - z^\star\|_M \le\|z^{n-1,0} - z^\star\|_M$ for each $z^\star \in \calZ^\star$.  By telescoping inequalities we then have $$\|z^{n,0} - z^\star\|_M \le\|z^{n-1,0} - z^\star\|_M \le \cdots \le \| z^{0,0} - z^\star\|_M $$
for each $z^\star \in \calZ^\star$.  Now notice that $z^0 := z^{0,0}$, $z^b := z^{n,0}$, and $z^c := z^{n-1,0}$ satisfy the hypotheses of Lemma \ref{lm: Mdistance upper bounded by normalized duality gap}, whereby it holds that \begin{equation}\label{eq of lm: Mdistance upper bounded by normalized duality gap2}
		\begin{aligned}
			& \dist_M(z^{n,0},\calZ^\star) \le   \widetilde{\condN} \cdot \rho(\|z^{n,0} - z^{n-1,0}\|_M;z^{n,0})   \ .
		\end{aligned}
	\end{equation}	
From the supposition of Theorem \ref{thm overall complexity} we have $c \in \linVp$ and therefore $\|P_{\linVp}(c)\| = \|c\| $, $\|P_{\linVp^\bot}(c)\| = 0 $, and $\|c\| = \dist(0,V_d)$.  Also by construction we have $\|q\| = \dist(0,V_p)$. Substituting the step-size values \eqref{eq smart step size 1} into \eqref{eq of lm: value tildeN} and using the above norm equalities in \eqref{eq of lm: value tildeN} yields:
	\begin{equation}\label{pence}
		\begin{aligned}
			& \widetilde{\condN} =  6\sqrt{2}\kappa \left(
			\frac{1}{\mu_p} + \frac{\tfrac{1}{3}}{\mu_d} \right)\thetax + 4\sqrt{2}\kappa \left(
			\frac{\tfrac{1}{2}}{\mu_p} + \frac{1}{\mu_d} \right)\thetas + 8\kappa\left(
			\frac{1}{\mu_p} + \frac{1}{\mu_d} \right) \left[ \frac{ \dist(0,\calX^\star) }{\dist(0, V_p)} + \frac{\dist(c,\calS^\star)}{ \dist(0,V_d)} \right]	
			 \ .
		\end{aligned}
	\end{equation}
Now notice that the value of $\condN$ specified in \eqref{eq easy L} is at least as large as the value of $\widetilde{\condN}$ above, whereby it holds that 
	$
	\dist_M(z^{n,0},\calZ^\star) \le  \condN \cdot \rho(\|z^{n,0} - z^{n-1,0}\|_M;z^{n,0})   
	$
for the value of $ \condN $ specified in  \eqref{eq easy L}. Therefore condition \eqref{eq restart L condition} of Lemma \ref{lm: complexity of PDHG with adaptive restart} is satisfied, and it follows from Lemma \ref{lm: complexity of PDHG with adaptive restart} that $T$ satisfies \eqref{drewcurse} with the value of $\tilde c$ specified in the statement of the lemma, namely:
\begin{equation}\label{drewcurse2}
	T \ \le \ \frac{5}{\beta \ln(1/\beta)} \cdot   \condN  \cdot \ln\left(   
	\tilde c \cdot  \condN  \cdot \left(\frac{\calE_d(x^{0,0},s^{0,0} )}{\eps} 
	\right)\right) + 1\ .
	\end{equation}
 Substituting in the step-sizes \eqref{eq smart step size 1} and $\beta = 1/e$ into the value of $\tilde c$ in Lemma \ref{lm: complexity of PDHG with adaptive restart} we find that 
$\tilde c = 8 e \cdot \left(1+\kappa \frac{\|c\|}{\|q\|}\right)\left(1 + \frac{\|q\|}{\|c\|}\right)$, and using $\beta = 1/e$ and the value of $\tilde{c}$ above in \eqref{drewcurse2} we finally arrive at:
\begin{equation}\label{drewcurse3}
	T \ \le \ 5e \cdot   \condN  \cdot \ln\left(   
	8 e  \cdot  \condN  \cdot \left(\frac{\calE_d(x^{0,0},s^{0,0} )}{\eps}\right) \Bigg(1+\kappa \frac{\|c\|}{\|q\|}\Bigg)\Bigg(1 + \frac{\|q\|}{\|c\|}\Bigg)
	\right) + 1\ .
	\end{equation}
	Note that in logarithm term in \eqref{drewcurse3} that because $\kappa \ge 1$ and $\|q\| = \|b\|_Q$ (where $\|b\|_Q$ denotes $\|A^\top (AA^\top)^\dagger b\|$) it follows that
	\begin{equation}\label{drewcurse4_revision}
	8 e \cdot \Bigg(1+\kappa \frac{\|c\|}{\|q\|}\Bigg)\Bigg(1 + \frac{\|q\|}{\|c\|}\Bigg)  	\le   \calD \ 
	\end{equation}
	where $\calD$ is defined in \eqref{eq:def_D}. Substituting \eqref{drewcurse4_revision} into \eqref{drewcurse3} yields \eqref{eq overall complexity},
	which completes the proof of the theorem.
	\end{proof}

We now present several remarks.

\begin{remark}[Extensions: other parameter settings, and allowing $c \notin \linVp$.]\label{rmk:talk about other parameters} Theorem \ref{thm overall complexity} uses the specific values of the primal and dual step-sizes in \eqref{eq smart step size 1}, uses $\beta = 1/e$, and assumes that $c \in \linVp$.  The results could have been presented without these specific conditions, albeit at the expense of more complicated expressions and more challenging exposition. Indeed, the step-sizes $\tau$ and $\sigma$ could have been left unassigned -- just as they are in Lemma \ref{lm: complexity of PDHG with adaptive restart} and Lemma \ref{lm: Mdistance upper bounded by normalized duality gap}. Using \eqref{eq of lm: Mdistance upper bounded by normalized duality gap2} directly with Lemma \ref{lm: complexity of PDHG with adaptive restart} and some algebraic calculations yields the proof for other step-size assignments. 

If $c \notin \linVp$, then the value of $\condN$ in \eqref{eq easy L} changes based on the changed value of $\widetilde{\condN}$ in \eqref{eq of lm: value tildeN}. Observing \eqref{eq of lm: value tildeN}, we see that $\|P_{\linVp^\bot} (c)\|$ appears only in the last term of the first of the three lines; if $c \notin \linVp$ then $\|P_{\linVp^\bot} (c)\| \neq 0$, and then the iteration bound in Theorem \ref{thm overall complexity} will change by replacing $\thetax$ with $\thetax + \frac{\|P_{\linVp^\bot} (c)\|}{\|P_{\linVp}(c)\|}$ throughout, and the step-sizes in \eqref{eq smart step size 1} will also change. If the step-sizes follow some other rules, then in addition to the change of $\condN$ based on $\widetilde{\condN}$ in \eqref{eq of lm: value tildeN}, the value of $\tilde{c}$ in Lemma \ref{lm: complexity of PDHG with adaptive restart} will also change. 

The choice of the value of the scalar $\beta$ affects only the expression $\frac{5}{\beta\ln(1/\beta)}$ outside the logarithmic term in \eqref{drewcurse} of Lemma \ref{lm: complexity of PDHG with adaptive restart}.  The specified value $\beta = 1/e$ was chosen to minimize the scalar $\frac{5}{\beta \ln(1/\beta)}$, but the value of $\beta$ could have been left unassigned, again at the expense of expositional simplicity.\end{remark}
	
	\ignore{Unpacking the proofs above, we see that $\| P_{\linVp^\bot} (c)\| $ only affects the proof in the bound \eqref{eq of lm: Mdistance upper bounded by normalized duality gap2}.  Suppose that $\| P_{\linVp^\bot} (c)\| \ne 0$ but is not too large, say, $\|P_{\linVp^\bot}(c)\| \le \theta_p^\star(\calX^\star) \|P_{\linVp}(c)\|$.  Then the right-hand side of \eqref{eq of lm: Mdistance upper bounded by normalized duality gap2}, and hence the value of $\condN$ in Lemma \ref{lm: complexity of PDHG with adaptive restart}, will only increase by a constant factor of at most $2$.}

\begin{remark}[Relation to the bound \eqref{intro:alpha_complexity} from \cite{applegate2023faster}]\label{remark: comparison with alpha complexity}
	The proof of Theorem \ref{thm overall complexity} is based in part on Lemma \ref{lm: complexity of PDHG with adaptive restart}. We note that the idea of Lemma \ref{lm: complexity of PDHG with adaptive restart} stems from the iteration bound \eqref{intro:alpha_complexity} from \cite{applegate2023faster}.  Indeed the condition \eqref{eq restart L condition} in Lemma \ref{lm: complexity of PDHG with adaptive restart} is quite similar to the sharpness condition \eqref{intro:sharpness} from \cite{applegate2023faster}; the difference is only in using $\condN$ instead of $\frac{1}{\alpha}$ and using the $M$-norm distance $\dist_M(z^{n,0},\calZ^\star)$ instead of the Euclidean distance $\dist(z^{n,0},\calZ^\star)$.  
\end{remark}

\begin{remark}[Relating our results to sharpness of the normalized duality gap functional in \cite{applegate2023faster}]\label{greengreen}  It follows from Lemma \ref{lm: Mdistance upper bounded by normalized duality gap} and the proof of Theorem \ref{thm overall complexity} that the normalized duality gap functional $\rho(r;z)$ obeys an ``$\alpha$-sharpness'' condition $\rho(\|z^{n,0} - z^{n-1,0} \|_M;z^{n,0}) \ge \alpha \cdot \dist_M(z^{n,0}, \calZ^\star)$ along the algorithm iterates for 
$$\alpha := \frac{1}{\condN} = \frac{1}{8.5 \kappa \left(
		\frac{1}{\mu_p} + \frac{1}{\mu_d} \right)
		\left(  
		\thetax +\thetas   +  \frac{ \dist(0,\calX^\star) }{\dist(0, V_p)} +  \frac{\dist(c,\calS^\star)}{ \dist(0,V_d)} 
		\right) }\ , $$
where the expression for $\condN$ above is from \eqref{eq easy L}. Note that this is just like the $\alpha$-sharpness condition \eqref{intro:sharpness} except it is expressed with the $M$-norm distance $\dist_M(z^{n,0},\calZ^\star)$ instead of the Euclidean distance.  Furthermore, it follows from Proposition \ref{fact: norm upperbound} that $\dist_M(z^{n,0}, \calZ^\star) \ge \gamma \cdot \dist(z^{n,0}, \calZ^\star) $ for a certain constant  $\gamma$ whose expression is defined using $\tau,\sigma,\lambda_{\max}$, see the formulas in Proposition \ref{fact: norm upperbound} for details.
\end{remark}

\begin{remark}[Extensions to other first-order methods]\label{remark: extension to other first-order methods}
	The proof of Theorem \ref{thm overall complexity} is based on the two ``steps'' of Lemma \ref{lm: complexity of PDHG with adaptive restart} and Lemma \ref{lm: Mdistance upper bounded by normalized duality gap}, and it is quite possible that it can be extended to certain other first-order methods. Indeed, \cite{applegate2023faster} proves that some iteration bounds analogous to \eqref{intro:alpha_complexity} also apply to other first-order methods, such as the alternating direction method of multipliers (ADMM) and extragradient method (EGM). For these methods, if similar results to Lemma \ref{lm: Mdistance upper bounded by normalized duality gap} can be extended (which we believe likely is the case), then computational guarantees similar to Theorem \ref{thm overall complexity} may be obtained for these other methods.  
\end{remark}

Section \ref{lm: Mdistance upper bounded by normalized duality gap} below is focused on the proof of Lemma \ref{lm: Mdistance upper bounded by normalized duality gap}.

\subsection{Proof of Lemma \ref{lm: Mdistance upper bounded by normalized duality gap}}\label{subsec: proof of Mdistance upper bounded by normalized duality gap}

This subsection is devoted to proving Lemma \ref{lm: Mdistance upper bounded by normalized duality gap}. As part of this task, we will need the following two very specific lemmas.

\begin{lemma}\label{lm: straight line mono}
	For any $\tilde z  = (\tilde x,\tilde y)$, let $z^\star \in \arg\min_{z \in \calZ^\star} \|z - \tilde z\|_M$, and define for all $t \ge 0$ 
	\begin{equation}\label{eq of lm: straight line mono 1}
		\tilde {z}_t := \tilde z + t\cdot(\tilde z -  z^\star) \  \
	\end{equation} 
	(whereby $\tilde z_0 = \tilde z$). Suppose that there exist nonnegative scalars $C_1$, $C_2$, $C_3$ such that the following inequality \eqref{eq of lm: straight line mono 2} holds for $t=0$:
	\begin{equation}\tag{$\calI_t$}\label{eq of lm: straight line mono 2}
			\dist_M(\tilde{z}_t,\calZ^\star) \le  C_1 \cdot \dist(\tilde{x}_t,V_p) + C_2 \cdot \dist(\tilde{s}_t, \mathbb{R}^n_+) + C_3 \cdot \max\{0, \gap(\tilde {x}_t, \tilde {s}_t) \} \ . 
	\end{equation}
	Then inequality \eqref{eq of lm: straight line mono 2} holds for all $t \ge 0$.
\end{lemma}
\begin{proof} Let $z^\star \in \arg\min_{z \in \calZ^\star}\|z - \tilde{z}\|_M$ be given. Then because $\calZ^\star$ is a convex set, it follows that $z^\star \in \arg\min_{z \in \calZ^\star}\|z - \tilde{z}_t\|_M$ for all $t \ge 0$.  Therefore $\dist_M(\tilde{z}_t,\calZ^\star) = \| z^\star - \tilde{z}_t\|_M = (1+t)\cdot \| \tilde z - z^\star\|_M  = (1+t)\cdot \dist_M(\tilde z,\calZ^\star)$.
	
	Regarding the terms on the right-hand side of \eqref{eq of lm: straight line mono 2}, for $t \ge 0$ we have
	$\dist( \tilde{x}_t,V_p)  = (1+ t)\cdot \dist(\tilde x,V_p)$ and $ \max\{0, \gap( \tilde{x}_t,  \tilde{s}_t) \}= (1+t) \cdot \max\{0, \gap(\tilde x, \tilde s) \}$. It also holds that:
	$$
	\dist( \tilde{s}_t, \mathbb{R}^n_+)  = \|( \tilde{s}_t)^-\| = \| (\tilde s + t\cdot(\tilde s -  s^\star))^-\| \ge \| (1+t)(\tilde s)^-\| = (1+ t)\cdot \dist(\tilde s, \mathbb{R}^n_+)   \ ,
	$$
	where the inequality above follows since $s^\star \ge 0$ and hence $(\tilde s + t\cdot(\tilde s -  s^\star))^- \ge (1+t)(\tilde s)^-$.
	Combining the above equalities and inequalities proves \eqref{eq of lm: straight line mono 2} for all $t \ge 0$.
\end{proof}

\begin{lemma}\label{lm: converegnce under M distance}
	Suppose that $\tau$ and $\sigma$ satisfy \eqref{eq  general step size requirement}. For any $x^0 \in \mathbb{R}^n_+$,  $y^0\in\mathbb{R}^m$ and $s^0:= c - A^\top y^0 \in\mathbb{R}^n$, then $z^0:= (x^0,y^0)$ satisfies:    
	\begin{equation}\label{eq  M-norm distance converge}
		\dist_M(z^0,\calZ^\star) \le  C_1 \cdot \dist(x^0,V_p) + C_2 \cdot \dist(s^0, \mathbb{R}^n_+) + C_3 \cdot \max\{0,\gap(x^0,s^0)\} \ ,
	\end{equation}
	where 
	\begin{equation}\label{eq of lm: converegnce under M distance}
		\begin{array}{rl}
			&C_1  :=  \Big( \thetax \| P_{\linVp}(c)\|  +  \| P_{\linVp^\bot}(c)\| \Big)  \cdot 
			\left(
			\frac{3\sqrt{2}}{\sqrt{\tau} \mu_p \| P_{\linVp}(c)\| } + \frac{\sqrt{2} }{\sqrt{\sigma} \mu_d \|q\| \lambda_{\min}}
			\right) \vspace{0.1cm}  \\
			&C_2  := \thetas \|q\| \cdot
			\left(
			\frac{2\sqrt{2}}{\sqrt{\sigma} \mu_d \|q\| \lambda_{\min}} + \frac{\sqrt{2}}{\sqrt{\tau} \mu_p \| P_{\linVp}(c)\| }
			\right) \vspace{0.1cm}  \\
			&C_3  :=  \left(  \frac{\sqrt{2}}{ \sqrt{\tau}\mu_p \| P_{\linVp}(c)\| } + \frac{\sqrt{2}}{\sqrt{\sigma}\mu_d \|q\| \lambda_{\min}  } \right) \ .
		\end{array}
	\end{equation}
\end{lemma}
\begin{proof} We presume that $z^0 \notin \calZ^\star$, for otherwise \eqref{eq  M-norm distance converge} follows trivially.
	Let $\tilde{x}^0$ be the projection of $x^0$ onto $V_p$, namely $\tilde x^0 = P_{V_p}(x^0)$. Then $\tilde{x}^0 - x^0$ is orthogonal to $\linVp$.  Also let $\hat{x}$ and $\hat{s}$ be the projections of $\tilde{x}^0$ and $s^0$ onto $\calF_p$ and $\calF_d$, respectively, namely $
	\hat{x}:= \arg\min_{x\in\calF_p} \|x - \tilde{x}^0\| $ and $\hat{s}:= \arg\min_{s\in\calF_d} \|s-s^0\| $.
	Since $\hat{x}$ and $\hat{s}$ are feasible for $\calF_p$ and $\calF_d$, respectively, the duality gap $\gap(\hat{x},\hat{s})$ is nonnegative. Furthermore, we have
	\begin{equation}\label{eq  lm: converegnce under M distance 1-pre0}
		\begin{aligned}
			\gap(\hat{x},\hat{s}) & \  = c^\top \hat{x} - q^\top (c - \hat{s})  = \gap(x^0,s^0) + c^\top (\hat{x} - x^0) + q^\top (\hat{s} - s^0) \\
			& \ \le \gap(x^0,s^0) +c^\top (\hat{x} - x^0) + \|q\| \cdot \dist(s^0,\calF_d) \ ,
		\end{aligned}
	\end{equation}
	and also 
	\begin{equation}\label{eq  lm: converegnce under M distance 1-pre1}
		\begin{small}
		\begin{aligned}
			& c^\top (\hat{x} - x^0)    = \left(P_{\linVp}(c) + P_{\linVp^\bot}(c)\right)^\top \left(\left(\hat{x} - \tilde{x}^0\right) + \left(\tilde{x}^0 - x^0\right)\right) =   P_{\linVp}(c)^\top \left(\hat{x} - \tilde{x}^0\right)  
			+ P_{\linVp^\bot}(c)^\top\left(\tilde{x}^0 - x^0\right) \\
			& \le \| P_{\linVp}(c)\| \cdot \left\| \hat{x} - \tilde{x}^0 \right\|
			+ \| P_{\linVp^\bot}(c)\| \cdot \left\| \tilde{x}^0 - x^0 \right\| = \| P_{\linVp}(c)\| \cdot \dist\left(\tilde{x}^0,\calF_p\right)
			+ \| P_{\linVp^\bot}(c)\| \cdot \dist(x^0,V_p) \ ,
		\end{aligned}
		\end{small}
	\end{equation}
	where the second equality above is due to $\hat{x} - \tilde{x}^0 \in \linVp$ (because both $\hat{x}$, $\tilde{x}^0 \in V_p$) and $\tilde{x}^0 - x^0 \in \linVp^\bot$. Substituting \eqref{eq  lm: converegnce under M distance 1-pre1} into \eqref{eq  lm: converegnce under M distance 1-pre0} then yields:
	\begin{equation}\label{eq  lm: converegnce under M distance 1}
		\begin{aligned}
			\gap(\hat{x},\hat{s})  \le \gap(x^0,s^0) + \| P_{\linVp}(c)\| \cdot \dist\left(\tilde{x}^0,\calF_p\right)
			+ \| P_{\linVp^\bot}(c)\| \cdot \dist(x^0,V_p)  + \|q\| \cdot \dist(s^0,\calF_d) \ .
		\end{aligned}
	\end{equation}

	Now we aim to replace the distance term involving $\tilde{x}^0$ in the right-hand side of \eqref{eq  lm: converegnce under M distance 1} with a term involving $x^0$. From the definition of the \er~$\theta_p(\cdot)$ we have:
	\begin{equation}\label{eq  lm: converegnce under M distance 1-0}
		\begin{aligned}
			\dist(\tilde{x}^0,\calF_p)& \ = \theta_p(\tilde{x}^0) \cdot \dist(\tilde{x}^0, \mathbb{R}^n_+) 
			\ \le  \theta_p(\tilde{x}^0) \cdot \|\tilde{x}^0 - x^0\|   =  \theta_p(\tilde{x}^0) \cdot \dist(x^0, V_p) \ ,
		\end{aligned}
	\end{equation}
	where the inequality uses ${x}^0 \in \mathbb{R}^n_+ $. Note that $\dist(x^0,\calF_p)\le 	\dist(\tilde{x}^0,\calF_p) + \|x^0 - \tilde{x}^0\|  = 	\dist(\tilde{x}^0,\calF_p) + \dist(x^0,V_p) $, 	so using \eqref{eq  lm: converegnce under M distance 1-0} we obtain 
	\begin{equation}\label{eq  lm: converegnce under M distance 1-1}
		\begin{aligned}
			\dist(x^0,\calF_p) \le (  \theta_p(\tilde{x}^0) +1)\cdot \dist(x^0, V_p) \ .
		\end{aligned}
	\end{equation}
	Similarly, since $s^0 \in V_d$, using the \er~$\theta_d(\cdot) $ we have:
	\begin{equation}\label{eq  lm: converegnce under M distance 1-2}
		\dist(s^0,\calF_d)\le \theta_d(s^0) \cdot \dist(s^0,\mathbb{R}^n_+) \ .
	\end{equation}
	Substituting \eqref{eq  lm: converegnce under M distance 1-0} and \eqref{eq  lm: converegnce under M distance 1-2} into \eqref{eq  lm: converegnce under M distance 1} yields:
	\begin{equation}\label{eq  lm: converegnce under M distance 2}
		\gap(\hat{x},\hat{s}) \le  \gap(x^0,s^0) + \left(\| P_{\linVp}(c)\|   \theta_p(\tilde{x}^0) + \| P_{\linVp^\bot}(c)\| \right) \cdot \dist(x^0,V_p) + \|q\|\cdot \theta_d(s^0) \cdot \dist(s^0,\mathbb{R}^n_+) \ .
	\end{equation}
	
	Let us now use \eqref{eq  lm: converegnce under M distance 2} to bound the distances to optima. 
	Note that the duality gap $\gap(\hat{x},\hat{s})$ is an upper bound for both $c^\top \hat{x} - f^\star$ and $f^\star - q^\top(c- \hat{s})$, where $f^\star$ denotes the optimal objective value. Then because $\hat{x} \in V_p$ and $\hat{s} \in V_d$ we have: 
	\begin{equation*} 
				\dist(\hat{x}, V_p\cap \{x : c^\top x = f^\star\}) \le \frac{\gap(\hat{x},\hat{s})}{\| P_{\linVp}(c)\| } \text{ and }
			 \dist(\hat{s}, V_d\cap \{s : q^\top (c-s) = f^\star\}) \le \frac{\gap(\hat{x},\hat{s})}{\|P_{\linVd}(q)\|} \ , 
	\end{equation*}
and note that $ \|P_{\linVd}(q)\| = \|q\| $ because $q \in \linVd$. Because $\gap(\hat{x},\hat{s}) \ge |c^\top \hat{x} - f^\star|$ and $\gap(\hat{x},\hat{s}) \ge |q^\top(c - \hat{s}) - f^\star|$, we have:
	\begin{equation}\label{eq  lm: converegnce under M distance 3}
		\begin{array}{ll}
			&\dist(\hat{x},\calX^\star) \le \frac{\dist(\hat{x}, V_p\cap \{x : c^\top x = f^\star\}) }{\mu_p} \le \frac{1}{\mu_p} \cdot \frac{\gap(\hat{x},\hat{s})}{\| P_{\linVp}(c)\| }  \ , \vspace{0.1cm}  \\
			&\dist(\hat{s},\calS^\star) \le \frac{\dist(\hat{s}, V_d\cap \{s : q^\top (c-s) = f^\star\}) }{\mu_d} \le \frac{1}{\mu_d} \cdot \frac{\gap(\hat{x},\hat{s})}{\|q\|}  \ . 
		\end{array}
	\end{equation}
	Now since
	$\dist(x^0,\calX^\star) \le \|x^0 - \hat{x}\| + \dist(\hat{x},\calX^\star) = \dist(x^0,\calF_p)+ \dist(\hat{x},\calX^\star)$, using \eqref{eq  lm: converegnce under M distance 1-1} and \eqref{eq  lm: converegnce under M distance 3} implies that:
	\begin{equation}\label{eq  lm: converegnce under M distance 4-0}
		\dist(x^0,\calX^\star) \le ( \theta_p(\tilde{x}^0) + 1)\cdot \dist(x^0,V_p)  + \frac{1}{\mu_p} \cdot \frac{\gap(\hat{x},\hat{s})}{\| P_{\linVp}(c)\| }  \ . 
	\end{equation}
	Combining \eqref{eq  lm: converegnce under M distance 2} and \eqref{eq  lm: converegnce under M distance 4-0} we obtain:
	\begin{equation}\label{eq  lm: converegnce under M distance 4}
		\begin{aligned}
			& \dist(x^0,\calX^\star)  \ \le    ( \theta_p(\tilde{x}^0) + 1)\cdot \dist(x^0,V_p)
			\\
			& \ \ \  \  \ +\frac{\gap(x^0,s^0) + \left(\| P_{\linVp}(c)\|   \theta_p(\tilde{x}^0) + \| P_{\linVp^\bot}(c)\| \right) \cdot \dist(x^0,V_p) + \|q\|\cdot \theta_d(s^0) \cdot \dist(s^0,\mathbb{R}^n_+)}{\mu_p \| P_{\linVp}(c)\| }  \\
			& \ = \frac{\gap(x^0,s^0) }{\mu_p \| P_{\linVp}(c)\| } + \left(
			\frac{ \theta_p(\tilde{x}^0)}{\mu_p } +
			\frac{\| P_{\linVp^\bot}(c)\| }{\mu_p \| P_{\linVp}(c)\| } 
			+ \theta_p(\tilde{x}^0) + 1
			\right) \cdot \dist(x^0,V_p) + \frac{\|q\| \theta_d(s^0) }{\mu_p \| P_{\linVp}(c)\| } \cdot \dist(s^0,\mathbb{R}^n_+) \ .
		\end{aligned}
	\end{equation}
	Note that because $ \mu_p \le 1$ and $\theta_p(\tilde{x}^0) \ge 1$, it follows that \eqref{eq  lm: converegnce under M distance 4} can be relaxed to:
	\begin{equation}\label{eq  lm: converegnce under M distance 4-2}
		\begin{aligned}
			\dist(x^0,\calX^\star) & \le \frac{\gap(x^0,s^0) }{\mu_p \| P_{\linVp}(c)\| } + 
			\left(\frac{3\theta_p(\tilde{x}^0)}{\mu_p} +\frac{\| P_{\linVp^\bot}(c)\| }{\mu_p \| P_{\linVp}(c)\| } 
			\right)
			\cdot \dist(x^0,V_p) + \frac{\|q\| \theta_d(s^0) }{\mu_p \| P_{\linVp}(c)\| } \cdot \dist(s^0,\mathbb{R}^n_+) \ .
		\end{aligned}
	\end{equation}
	Using almost identical logic applied to $s^0$ instead of $x^0$, we obtain:
	\begin{equation}\label{eq  lm: converegnce under M distance 5}
		\dist(s^0,\calS^\star) 
		\le \frac{\gap(x^0,s^0) }{\mu_d \|q\|} +
		\frac{2\theta_d(s^0)}{\mu_d} \cdot \dist(s^0,\mathbb{R}^n_+) + \frac{\| P_{\linVp}(c)\| \theta_p(\tilde{x}^0)  + \| P_{\linVp^\bot}(c)\| }{\mu_d \|q\|} \cdot \dist(x^0,V_p) \ .
	\end{equation}
	
	Combining \eqref{eq  lm: converegnce under M distance 4-2} with \eqref{eq  lm: converegnce under M distance 5} and using the right-most inequality of \eqref{eq of lm: M norm to seperable norm}, it follows that
	\begin{equation}\label{ngeq  M-norm distance converge}
		\dist_M(z^0,\calZ^\star) \le  \bar C_1(z^0)  \cdot \dist(x^0,V_p) + \bar C_2(z^0)  \cdot \dist(s^0, \mathbb{R}^n_+) + \bar C_3(z^0)  \cdot \max\{0,\gap(x^0,s^0)\} \ ,
	\end{equation}
	where 
	\begin{equation}\label{ngeq of lm: converegnce under M distance}
		\begin{array}{rl}
			&\bar C_1(z^0) :=\Big( \theta_p(P_{V_p}(x^0))\| P_{\linVp}(c)\|  +  \| P_{\linVp^\bot}(c)\| \Big)  \cdot 
			\left(
			\frac{3\sqrt{2}}{\sqrt{\tau} \mu_p \| P_{\linVp}(c)\| } + \frac{\sqrt{2} }{\sqrt{\sigma} \mu_d \|q\| \lambda_{\min}}
			\right) \vspace{0.1cm} \\
			& \bar C_2(z^0) := \theta_d(s^0) \|q\| \cdot
			\left(
			\frac{2\sqrt{2}}{\sqrt{\sigma} \mu_d \|q\| \lambda_{\min}} + \frac{\sqrt{2}}{\sqrt{\tau} \mu_p \| P_{\linVp}(c)\| }
			\right) \vspace{0.1cm} \\
			&\bar C_3 (z^0)  :=  C_3 \ ,
		\end{array}
	\end{equation}and notice in the definition $\bar C_1$ we have written $\theta_p(P_{V_p}(x^0))$ since in fact $\tilde x^0 :=P_{V_p}(x^0)$. Now notice that \eqref{ngeq  M-norm distance converge} is nearly identical to \eqref{eq  M-norm distance converge}, except that the constants $\bar C_1(z^0)$ and $\bar C_2(z^0)$ use  $\theta_p(P_{V_p}(x^0))$ instead of $\thetax$, and use $\theta_d(s^0)$ instead of $\thetas$.  
	
	To finish the proof, let $z^{\star} \in \arg\min_{z \in \calZ^\star}\| z - z^0\|_M$ be fixed, and define $z^\lambda := (1-\lambda) z^0 + \lambda z^\star$ for all $\lambda \in [0,1)$. Then \eqref{ngeq  M-norm distance converge} holds for $z^\lambda$, namely:
	\begin{equation}\label{laketime}
		\dist_M(z^\lambda,\calZ^\star) \le  \bar C_1(z^\lambda)  \cdot \dist(x^\lambda,V_p) + \bar C_2(z^\lambda)  \cdot \dist(s^\lambda, \mathbb{R}^n_+) + \bar C_3(z^\lambda)  \cdot \max\{0,\gap(x^\lambda,s^\lambda)\} \ ,
	\end{equation}
	since $z^\lambda$ satisfies the same hypotheses as $z^0$. And since $z^{\star} \in \arg\min_{z \in \calZ^\star}\| z - z^\lambda\|_M$ we can invoke Lemma \ref{lm: straight line mono}.  It follows from Lemma \ref{lm: straight line mono} that for all $t \ge 0$ with $z^\lambda_t:= z^\lambda + t (z^\lambda - z^\star)$ that
	\begin{equation}\label{laketime2}
		\dist_M(z^\lambda_t,\calZ^\star) \le  \bar C_1(z^\lambda)  \cdot \dist(x^\lambda_t,V_p) + \bar C_2(z^\lambda)  \cdot \dist(s^\lambda_t, \mathbb{R}^n_+) + \bar C_3(z^\lambda)  \cdot \max\{0,\gap(x^\lambda_t,s^\lambda_t)\} \ .
	\end{equation}
	Setting $t = \lambda/(1-\lambda)$ yields $z^\lambda_t = z^0$, whereby:
	\begin{equation}\label{laketime3}
		\dist_M(z^0,\calZ^\star) \le  \bar C_1(z^\lambda)  \cdot \dist(x^0,V_p) + \bar C_2(z^\lambda)  \cdot \dist(s^0, \mathbb{R}^n_+) + \bar C_3(z^\lambda)  \cdot \max\{0,\gap(x^0,s^0)\} \ .
	\end{equation} Now let $\lambda \rightarrow 1$, whereby $ z^\lambda \rightarrow z^\star$, and so $\operatorname{\lim\sup}_{\lambda \rightarrow 1} \theta_p(P_{V_p}(x^\lambda)) \le \thetax$ and therefore $\operatorname{\lim\sup}_{\lambda \rightarrow 1} \bar C_1(z^\lambda) \le C_1$. Similarly $\operatorname{\lim\sup}_{\lambda \rightarrow 1} \theta_d(s^\lambda) \le \thetas$ and therefore $\operatorname{\lim\sup}_{\lambda \rightarrow 1} \bar C_2(z^\lambda) \le C_2$.  Thus, we can conclude \eqref{eq M-norm distance converge} from \eqref{laketime3}.
\end{proof}

Finally, we prove Lemma \ref{lm: Mdistance upper bounded by normalized duality gap}.

\begin{proof}[Proof of Lemma \ref{lm: Mdistance upper bounded by normalized duality gap}] The proof is a combination of Lemmas \ref{lm: convergence of PHDG without restart}, \ref{lm: R in the opt gap convnergence}, and \ref{lm: converegnce under M distance}. Setting $s^b = c - A^\top y^b$ it follows from Lemma \ref{lm: convergence of PHDG without restart} that
	\begin{equation}\label{eq lm: Mdistance upper bounded by normalized duality gap 1-3}
		\begin{array}{rl}
			&\dist(x^b,V_p)  \le \frac{1}{\sqrt{\sigma} \lambda_{\min}}\cdot \rho(\|z^b- z^c\|_M;z^b) \ , \vspace{0.1cm} \\ 
			&\dist(s^b,\mathbb{R}^n_+)  \le \frac{1}{\sqrt{\tau}} \cdot \rho(\|z^b - z^c\|_M;z^b) \ , \vspace{0.1cm}  \\
			&\gap(x^b,s^b)  \le   \max\{ \|z^b- z^c\|_M, \|z^b\|_M\} \rho(\|z^b - z^c\|_M;z^b) \ .
		\end{array} 
	\end{equation}
	Also, from Lemma \ref{lm: converegnce under M distance} it follows that $\dist_M(z^b,\calZ^\star)$ can be bounded using the terms in the left-hand side of \eqref{eq lm: Mdistance upper bounded by normalized duality gap 1-3}:
	\begin{equation}\label{eq lm: Mdistance upper bounded by normalized duality gap 4}
		\dist_M(z^b,\calZ^\star) \le  C_1 \cdot \dist(x^b,V_p) + C_2 \cdot \dist(s^b, \mathbb{R}^n_+) + C_3 \cdot \gap(x^b,s^b) \ ,
	\end{equation}
	where $C_1$, $C_2$ and $C_3$ are the scalars defined in \eqref{eq of lm: converegnce under M distance}.
	Substituting \eqref{eq lm: Mdistance upper bounded by normalized duality gap 1-3} into \eqref{eq lm: Mdistance upper bounded by normalized duality gap 4} yields:
	\begin{equation}\label{eq lm: Mdistance upper bounded by normalized duality gap 5}
		\dist_M(z^b,\calZ^\star) \le  \left(
		\frac{C_1}{\sqrt{\sigma}\lambda_{\min}} + \frac{C_2}{\sqrt{\tau}} + C_3 \cdot \max\{ \|z^b - z^c\|_M, \|z^b\|_M\} 
		\right) \rho(\|z^b- z^c\|_M;z^b)  \ .
	\end{equation}
	From Lemma \ref{lm: R in the opt gap convnergence} it holds that	
	$$
	\begin{aligned}
		\max\{ \|z^b - z^c\|_M, \|z^b\|_M\} & \ \le    2 \dist_M(z^{0},\calZ^\star)  + \|z^{0} \|_M	 = 2 \dist_M(0,\calZ^\star)  \\ 
		& \ \le   \frac{2\sqrt{2}}{\sqrt{\tau}} \dist(0,\calX^\star) + \frac{2\sqrt{2}}{\sqrt{\sigma}\lambda_{\min}} \dist(c,\calS^\star) \ ,
	\end{aligned}
	$$
	where the second inequality uses \eqref{eq of lm: M norm to seperable norm}. Substituting this inequality into \eqref{eq lm: Mdistance upper bounded by normalized duality gap 5} yields
	\begin{equation}\label{eq lm: Mdistance upper bounded by normalized duality gap 6}
		\begin{aligned}
			& 	\dist_M(z^b,\calZ^\star)  \le  \left(
			\frac{C_1}{\sqrt{\sigma}\lambda_{\min}} + \frac{C_2}{\sqrt{\tau}} 
			+
			\frac{2  \sqrt{2} C_3}{\sqrt{\tau}} \dist(0,\calX^\star) + \frac{2 \sqrt{2}C_3}{\lambda_{\min}\sqrt{\sigma}} \dist(c,\calS^\star)  			
			\right) \rho(\|z^b- z^c\|_M;z^b) \ .
		\end{aligned}
	\end{equation}
	The proof is completed by substituting the values of $C_1, C_2, C_3$ defined in \eqref{eq of lm: converegnce under M distance} into \eqref{eq lm: Mdistance upper bounded by normalized duality gap 6}, which then yields \eqref{eq of lm: Mdistance upper bounded by normalized duality gap}.
\end{proof}

\section{Properties of the Limiting Error Ratio (\limitinger)}\label{sec error ratio}

In this section we present some relevant properties of the \limitingeratio~(\limitinger). Without loss of generality, we focus primarily on $\thetax$ and similar arguments could also be made for $\thetas$.
Theorem \ref{thm local geometry to feasibility error ratio at optima} in Section \ref{subsec: local geometry to feasibility error ratio at optima} characterizes an upper bound on $\thetax$ that is connected to the notion of a ``nicely interior'' point in a convex set -- which itself is critical to the complexity of separation-oracle methods \cite{fv4}.  Proposition \ref{cusco} in Section \ref{subsec computable theta upperbound} presents a convex optimization problem (actually a conic optimization problem with one second-order cone constraint) whose solution provides an upper bound on $\thetax$, thus showing that computing a bound on $\thetax$ is computationally tractable. Finally, Theorem \ref{thm: general theta with infeasibility} in Section \ref{subsec: theta and distance to infeasibility} shows that the error ratio $\theta_p(x)$ is upper-bounded by a simple quantity involving the data-perturbation condition number $\distinfeas(\cdot)$ of Renegar \cite{renegar1994some}. Proofs of these results are presented in Appendix \ref{subsec: proof of sec error ratio}. 

Our setup once again is the LP problem \eqref{pro: general primal LP} in which the feasible set is $\calF_p$, the intersection of $V_p$ and $\mathbb{R}^n_+$.  We will assume in this subsection that $\calX^\star$ is nonempty, and recall the definition of the \limitingeratio~(\limitinger) $\thetax$ in \eqref{kansas}. Let $\calF_{++}$ denote the strictly feasible solutions of \eqref{pro: general primal LP}, namely $\calF_{++} := V_p\cap \mathbb{R}^n_{++}$.  We do not necessarily assume that $\calF_{++} \ne \emptyset$.

\subsection{An upper bound based on ``nicely interior'' feasible solutions}\label{subsec: local geometry to feasibility error ratio at optima}

The following theorem presents an upper bound on the \limitingeratio~$\thetax$  using the existence of a ``nicely interior'' point in $\calF_{++}$. (In the theorem we use the convention that the infimum over an empty set is $+\infty$.)

\begin{theorem}\label{thm local geometry to feasibility error ratio at optima}
	For the LP problem \eqref{pro: general primal LP}, suppose that the optimal solution set $\calX^\star$ is nonempty and bounded.  Then 
	\begin{equation}\label{eq local geometry to feasibility error ratio at optima}
		\thetax \le \sup_{x^\star\in \calX^\star} \inf_{x_{\mathrm{int}} \in \calF_{++}}  \frac{\|x^\star - x_{\mathrm{int}}\|}{\min_i (x_{\mathrm{int}})_i} \ .
	\end{equation}
\end{theorem}

This theorem states that if every optimal solution $x^\star$ has a nicely interior point near to it -- in the sense that there exists $x_{\mathrm{int}} \in \calF_{++}$ that is simultaneously close to $x^\star$ and far from the boundary of the nonnegative orthant $\mathbb{R}^n_+$, then the \limitinger~value $\thetax$ will not be excessively large. In the case when $\calX^\star$ is a singleton, then \eqref{eq local geometry to feasibility error ratio at optima} simplifies to finding a single nicely interior point that balances the distance from the optimal solution (in the numerator above) with the distance to the boundary of $\mathbb{R}^n_+$ (in the denominator above). (Note that the concept of a nicely interior point is quite similar to that of a ``reliable solution'' in \cite{EpeFre00}, see also \cite{fv1} for connections to Renegar's data-perturbation condition number $\distinfeas(\cdot)$ \cite{renegar1994some}.) 

The same argument also holds for $\thetas$. We prove Theorem \ref{thm local geometry to feasibility error ratio at optima} in Appendix \ref{subsec: proof of sec error ratio} for a generic form of LP, which encompasses both the primal and dual problems.

\subsection{A computable upper bound for the \limitingeratio}\label{subsec computable theta upperbound}

Here we show how, in principle, we can use the upper bound in Theorem \ref{thm local geometry to feasibility error ratio at optima} to construct a computationally tractable convex optimization problem that computes an upper bound on the \limitinger~$\thetax$.  

We first suppose that we can compute, without too much extra computational effort, a ball that contains the optimal solution set $\calX^\star$.  That is, we suppose we can compute a point $x_a \in \calX^\star$ and a radius value $R_a$ such that $\calX^\star \subset B(x_a, R_a)$, so that every optimal solution is within a distance $R_a$ from the optimal solution $x_a$.  If $\calX^\star$ is a singleton, then $R_a = 0$ trivially.  If $\calX^\star$ is not a singleton, then one choice of $x_a$ is the analytic center of $\calX^\star$ (see Sonnevend \cite{son1}, also \cite{nesterov1994interior}), from which one can then easily construct a bounding ellipsoid $\calE^{\mathrm{out}}$ that contains $\calX^\star$ and then examine the eigenstructure $\calE^{\mathrm{out}}$ to compute a suitable value of $R_a$.

\begin{proposition}\label{cusco} Suppose $x_a \in \calX^\star$ and there exists $R_a$ for which $\calX^\star \subset \{x: \|x-x_a \| \le R_a\}$, then it holds that $ \thetax \le G^\star$ for $G^\star$ defined as follows:
	\begin{equation}\label{rrr}  
			G^\star := \ \inf_{r >0, \ x\in \mathbb{R}^n}  \displaystyle\frac{R_a+\|x-x_a\|}{r} \quad \operatorname{s.t.} \ x \in V_p , \ x \ge r \cdot e \ . 
	\end{equation} Furthermore, since $V_p  = \{ \hat{x} \in \mathbb{R}^n : A\hat{x} = b\}$, then 
	\begin{equation}\label{rrrr}  
			G^\star := \ \min_{v\in \mathbb{R}^n, \ \alpha \in \mathbb{R}}  R_a \alpha+\|v- \alpha x_a \| \quad \operatorname{s.t.} \ Av = \alpha  b , \ v \ge e , \ \alpha \ge 0 \ . 
	\end{equation}
\end{proposition}

Proposition \ref{texas primal} in Section \ref{subsec: comments} is just a restatement of \eqref{rrr}.
The formulation of the upper bound in \eqref{rrrr} is a convex optimization problem of essentially the same size as that of the original LP problem \eqref{pro: general primal LP}, and its only non-linear component is the norm term $\|v- x_a \alpha\| $ in the objective function.  This can easily be handled by a single second-order cone constraint, or can be upper bounded by an $\ell_1$ or $\ell_\infty$ norm which then can be converted to a pure LP problem.  

We prove Proposition \ref{cusco} in Appendix \ref{subsec: proof of sec error ratio} for a generic form of LP, which encompasses both the primal and dual problems. 
As for computing the upper bound for $\thetas$, given $s_a$ and $R_a$ such that $s_a \in \calS^\star$ and $\calS^\star \subset B(s_a, R_a)$, the upper bound $G^\star$ in \eqref{rrr} could be computed by solving the optimization problem:
	\begin{equation}\label{rrrrr}  
			\ \min_{v\in \mathbb{R}^n, \ y\in\mathbb{R}^m , \ \alpha \in \mathbb{R}}  R_a \alpha+\|v- \alpha s_a \| \quad \operatorname{s.t.} \ A^\top y + v = \alpha c , \ v \ge e , \ \alpha \ge 0 \ . 
	\end{equation}

\subsection{Relationship between the \limitinger~and the distance to infeasibility}\label{subsec: theta and distance to infeasibility}

We have established the relationship between $\thetax$ and the geometric properties of the feasible sets. 
In this subsection, we demonstrate that this relationship also extends to the distance of the data from infeasibility. The concept of distance to infeasibility was initially utilized to assess the complexity of LP  \cite{renegar1994some}.
Previous research, such as \cite{luo1994perturbation,hu2000perturbation}, has also studied the connections between global upper bounds of error bounds and the existence of perturbations. Here we show it also holds for the \er~$\theta_p(x)$ and the \limitinger~$\thetax$. 
We primarily focus on $\thetax$ for the primal problem for clarity, but similar arguments also hold for $\thetas$. In the appendix we will prove them for a generic form of LP, which encompasses both the primal and dual problems.                

Note that $V_p$ is given by $\{\hat{x} \in \mathbb{R}^n: A \hat{x} = b\}$ for an $m\times n$ real matrix $A$ and a vector $b$ in $\mathbb{R}^m$. Let $\soln(A, b)$ denote the feasible set corresponding to $(A, b)$, namely $\soln(A, b) := \{\hat{x}\in \mathbb{R}^n : A \hat{x} = b, x \ge 0\}$.  We suppose that $\soln(A, b) \ne \emptyset$, in which case the ``distance to infeasibility'' of the data $(A,b)$ is defined as follows:
$$ \distinfeas(A,b) := \inf\left\{\|\Delta A\| + \| \Delta b\| : \soln(A + \Delta A, b + \Delta b)  = \emptyset  \right\} \ ,  
$$ see \cite{renegar1994some}.
Now we have the following general theorem about the relationship between $\theta_p(x)$ and the distance to infeasibility.

\begin{theorem}\label{thm: general theta with infeasibility}
	Suppose that $\calF_p$ is nonempty for the LP \eqref{pro: general primal LP}. Then for every $x\in V_p \setminus \calF_p$, it holds that
	\begin{equation}\label{eq thm general theta with infeasibility}
		\theta_p(x)  \le \frac{\|A\|(1 + \|x\|)}{\distinfeas(A,b)} \ .
	\end{equation}
\end{theorem}

This theorem shows that the larger the distance to infeasibility is (namely, the larger the least data perturbation to infeasibility), the smaller the error ratio $\theta_p(x)$ must be. The inequality \eqref{eq thm general theta with infeasibility} looks similar in spirit to Theorem 1.1 part (1) of Renegar \cite{renegar1994some}, even though the setup and context are structurally different from that considered here.  
From Theorem \ref{thm: general theta with infeasibility} we also have the following relationship between $\thetax$ and the distance to infeasibility.

\begin{corollary}\label{cor: theta with infeasibility}
	Suppose that the LP \eqref{pro: general primal LP} has an optimal solution.  If $\distinfeas(A,b) > 0$, then it holds that
	\begin{equation}\label{eq thm theta with infeasibility}
		\thetax \le  \frac{\|A\|(1 + \max_{x\in \calX^\star}\|x\|)}{\distinfeas(A,b)} \ .
	\end{equation}
\end{corollary}

Both Theorem \ref{thm: general theta with infeasibility} and Corollary  \ref{cor: theta with infeasibility} imply that the farther the feasible set $\calF_p$ is from infeasibility~(namely, the larger $\distinfeas(A,b)$ is), then the smaller $\theta_p(x)$ and $\thetax$ must be. It should be noted that these inequalities do not hold oppositely, because $\theta_p(x)$ and $\thetax$ are only determined by the geometry of the feasible set, while the $\distinfeas(A,b)$ is affected by the data as well. For example, simultaneously rescaling a row of $A$ and the corresponding entry of $b$ by a small factor could decrease the value of $\distinfeas(A,b)$ while keeping $\|A\|$ roughly unchanged. This would significantly increase the right-hand sides of \eqref{eq thm general theta with infeasibility} and \eqref{eq thm theta with infeasibility}, but the left-hand sides are unchanged because it does not affect the geometry of the feasible set.

The main idea of the proof of Theorem \ref{thm: general theta with infeasibility} is to construct, for each $x\in V_p \setminus \calF_p$, a suitable perturbation $(\Delta A,\Delta b)$ of $(A,b)$ for which $\distinfeas(A + \Delta A, b + \Delta b)  = 0$ and $\theta_p(x) \le \frac{\|A\| ( 1 + \|x\|)}{\|\Delta A \| + \|\Delta b\|}$.

\section{LP Sharpness and stability under perturbation}\label{sec sharpness}
In this section we present a characterization of the LP sharpness in terms of the least relative perturbation of the objective function vector that yields a different optimal solution set that is nonempty and not a subset of the original solution set.  We also present a characterization of the LP sharpness via polyhedral geometry, with implications for computing methods.
We still primarily focus on $\mu_p$ for the primal problem \eqref{pro: general primal LP} while similar results also hold for $\mu_d$ because of the symmetric reformulation \eqref{pro: primal dual reformulated LP}.

\subsection{Characterization of LP sharpness and the stability of optimal solutions under perturbation}

Let $\opt(c,\calF_p)$ denote the set of optimal solutions of the  LP problem \eqref{pro: general primal LP}, namely
$$
\opt(c,\calF_p) := \arg\min_{x\in\calF_p} \ c^\top x \ .
$$
Let $C_{\calX^\star}$ denote the recession cone of the optimal solution set $\calX^\star$, and let $C_{\calX^\star}^*$ denote its dual cone, namely
$C_{\calX^\star}^* = \{w : w^\top d \ge 0 \ \text{for all}\ d \in C_{\calX^\star} \}$.
Then for any $\Delta c \in \linVp^\bot$, $\opt(c,\calF_p) = \opt(c + \Delta c,\calF_p)$.
Additionally, if $\Delta c \notin C_{\calX^\star}^*$, then $\opt(c + \Delta c,\calF_p) = \emptyset$ (indeed, there exists $d \in C_{\calX^\star}$ with $\Delta c^\top d < 0$; moreover $c^\top d=0$ for all $d\in C_{\calX^\star}$, and therefore $(c+\Delta c)^\top(x^\star+t d)\to -\infty$ as $t \to \infty$ and the resulting LP instance has unbounded objective value).

Intuition suggests that the LP sharpness value $\mu_p$ should be related to objective function perturbations that alter the set of optimal solutions.  Indeed, we have the following theorem that characterizes this relationship completely, namely $\mu_p$ is the smallest relative perturbation $\Delta c$ of $c$ that yields a different optimal solution set that is nonempty and not a subset of the original optimal solution set.  More precisely, we have:
\begin{theorem}\label{thm: sharpness and perturbation} Consider the general LP problem \eqref{pro: general primal LP} under Assumption \ref{assump: general LP}, and let $\mu_p$ be the \LPsharp~of \eqref{pro: general primal LP}. Then
	\begin{equation}\label{eq of thm: sharpness and perturbation}
		\mu_p = \inf_{\Delta c}\left\{
		\frac{\|P_{\linVp}(\Delta c)\|}{\|P_{\linVp}(c)\|}:
		\opt(c + \Delta c,\calF_p) \ne \emptyset \ \text{ and }  \opt(c + \Delta c,\calF_p) \not\subset \opt(c,\calF_p)
		\right\} \ .
	\end{equation}
\end{theorem}
Note that under Assumption \ref{assump: general LP} we must have $\|P_{\linVp}(c)\| > 0$ as otherwise all feasible solutions would be optimal (which violates Assumption \ref{assump: general LP}).  The proof of Theorem \ref{thm: sharpness and perturbation} in its generic form (which encompasses both the primal and dual problems) is in Appendix \ref{app: proof of sharpness section}.

\subsection{Polyhedral geometry characterization of \LPsharp}\label{poker2}

In this subsection we present a polyhedral characterization of the \LPsharp, which leads to an explicit formula for the \LPsharp.
Before we go into details, we first convey our general results as follows, using the 
the primal problem \eqref{pro: general primal LP} as an example. If the optimal solution set is a singleton, namely $\calX^\star = \{x^\star\}$, then we will show that the LP sharpness $\mu_p$ is the smallest sharpness along all of the edges of $\calF_p$ emanating from $x^\star$.  One implication of this result is that if the LP instance is primal and dual nondegenerate, then the dual nondegeneracy implies that $\calX^\star$ is a singleton, and the primal nondegeneracy implies that there are exactly $n-m$ edges emanating from $\calX^\star$, and hence it will be very easy to compute the LP sharpness.  The more general result that we will show is that LP sharpness is the smallest sharpness along all edges of $\calF_p$ that intersect $\calX^\star$ but are not subsets of $\calX^\star$.  In the absence of nondegeneracy there can be exponentially many such edges, and so computing the LP sharpness for either a primal or dual degenerate instance is not a tractable problem in general.  

We now develop these results more formally. For any $x \in \calF_p \setminus \calX^\star$,  we define the sharpness of the point $x$ to be:
$$G(x):=\frac{\dist(x,V_p\cap H_p^\star)}{\dist(x, \, \calX^\star)} = \frac{\frac{c^\top (x - x^*) }{ \|P_{\linVp}(c)\|}}{\|x - x^\star \|} $$
where $x^\star := \arg\min_{v^\star \in \calX^\star} \|v^\star - x\|$ is the projection of $x$ onto $\calX^\star$, and $H_p^\star := \{\hat{x} \in \mathbb{R}^n:c^\top \hat{x}  = c^\top x^\star\}$ is the optimal objective hyperplane.  With this notation the LP sharpness \eqref{topsyturvy} is $\mu_p = \inf_{x \in \calF_p \setminus \calX^\star} G(x)$.  
Next let us recall some notation about convex polyhedra, see \cite{grunbaum}.  An edge of a polyhedron is a $1$-dimensional face of the polyhedron.  And since $\calF_p$ is a polyhedron it follows that $\calF_p$ will have a finite number of edges. Furthermore, every edge will either be (i) a line segment joining two different vertices $v^1 \ne v^2$ of $\calF_p$ which we denote by $\e=[v^1,v^2]$, or (ii) a half-line of points $v + \theta r $ for all $\theta \ge 0$, where $v$ is a vertex of $\calF_p$ and $r$ is an extreme ray of $\calF_p$, and which we denote by $\f = [v;r]$. We will be concerned with the subset of edges of $\calF_p$ which have one endpoint in $\calX^\star$ but are not subsets of $\calX^\star$, which we call edges emanating away from $\calX^\star$ and which we denote as $\calM$, and whose formal definition is:
$$\begin{array}{rll} \calM := \calM_1 \cup \calM_2  \ \ \ \ \text{      where     } & \calM_1 := \{\e=[v^1, v^2] : \e \text{ is an edge of } \calF_p, \ v^1 \in \calX^\star, v^2 \notin \calX^\star \} \\ \text{     and     } & \calM_2 := \{ \f=[v,r] : \f \text{ is an edge of } \calF_p, \ v \in \calX^\star, c^\top r =1 \} \ . \end{array} $$
The theorem below shows that the \LPsharp~can be characterized using the following two functions:
$$
	\begin{aligned}
	& R_1(\e):= G(\tilde x) = \frac{\dist(\tilde{x},V_p\cap H_p^\star)}{\dist(\tilde{x},\, \calX^\star)} \text{ for all edges }\e = [x^\star,\tilde{x}] \in \calM_1 \, \text{, and} \\
	& R_2(\f;\bar{\eps}):=G(x^\star + \bar{\eps}\cdot r) =\frac{\dist(x^\star + \bar{\eps}\cdot r,V_p\cap H_p^\star)}{\dist(x^\star + \bar{\eps}\cdot r,\, \calX^\star)} \text{ for all edges } \f = [x^\star;r] \in \calM_2 \text{ and all }\bar{\eps}>0 \ .
	\end{aligned}
$$
Because $\calM$ is a finite set, we can write $\calM = \{\e^i:i=1,2,\dots, m_1\} \cup \{\f^j:j=1,2,\dots,m_2\}$ for some integers $m_1, m_2$.
\begin{theorem}\label{thm: sharpness from adjacent edges} For any given $\bar{\eps} > 0$, the \LPsharp~is characterized as follows:
	\begin{equation}\label{eq thm: sharpness from adjacent edges 2}
		\mu_p= \min \left\{ R_1(\e^1), R_1(\e^2),\dots,R_1(\e^{m_1}), R_2(\f^1;\bar{\eps}),R_2(\f^2;\bar{\eps}),\dots,R_2(\f^{m_2};\bar{\eps}) \right\} \  .
	\end{equation} 
\end{theorem}

Theorem~\ref{thm: sharpness from adjacent edges} provides a purely polyhedral characterization of $\mu_p$ via edges of $\calF_p$ that emanate from $\calX^\star$. In the special case of a unique optimal basis, Theorem~\ref{thm: sharpness from adjacent edges} leads to a closed-form expression for $\mu_p$ in terms of the optimal basis and optimal dual slack variable values of the simplex tableau at the optimal basis, see Lemma 5.2 of \cite{xiong2024accessible} for details. For general LP instances, depending on the number of optimal bases, the computation of $\mu_p$ may be straightforward or exponentially expensive in the worst case.

The proof of Theorem \ref{thm: sharpness from adjacent edges} is in Appendix \ref{app: proof of sharpness section}. 
 It is noted that the $R_i{(\e^i)}$ and $R_2(\f;\bar{\eps})$ in the right-hand side of \eqref{eq thm: sharpness from adjacent edges 2} are both purely geometric quantities. Therefore, the characterization of $\mu_p$ in Theorem \ref{thm: sharpness from adjacent edges} and its implications for computing $\mu_p$  also hold for $\mu_d$ in a symmetric way on the dual problem.

\section{Numerical Experiments}\label{sec experiments}

Here we present results of numerical experiments designed to test how consistent our theoretical bounds are with computational practice, as well as to demonstrate the value of various heuristics on practical computation, based on our theoretical results.  All computation was conducted on the MIT Engaging Cluster, and each experiment used a 2.4 GHz 14 Core CPU and 32G RAM, with CentOS version 7. All experiments were implemented in Julia 1.8.5. 

\subsection{Simple validation experiments}\label{subsec:five simple experiments}

We conducted five simple experiments to test the extent to which the iteration bounds in Theorem \ref{thm overall complexity} are ``valid,'' by which we mean that the bounds are directionally consistent with computational practice on a specifically chosen family of test problems.  \medskip

\noindent {\bf Experiment 1: Sensitivity to the Hoffman constant of the KKT system.} The bounds in Theorem \ref{thm overall complexity} are based essentially on three condition measures: LP sharpness, \limitinger, and (relative) distance to optima.  This is in contrast to the analysis of \cite{applegate2023faster} whose bounds are mostly based on the Hoffman constant of the KKT system of the LP instance.  To test the sensitivity of Algorithm \ref{alg: PDHG with restarts} to the Hoffman constant of the KKT system, we created the following family of LP instances in standard form \eqref{pro: general primal LP} with $m=1$, $n=3$, and data $(A^1, b^1, c^1)$ parameterized by $\gamma \in (0,1]$ as follows: $A^{1}_{\gamma} := \left[\frac{\sin(\gamma)}{\sqrt{2}}   , \cos(\gamma), \frac{\sin(\gamma)}{\sqrt{2}}  \right]$, $b^{1} := 1 $, 	$c^{1}_{\gamma} := \left[\frac{\cos(\gamma)}{\sqrt{2}}  , -\sin(\gamma), \frac{ \cos(\gamma)}{\sqrt{2}}  \right]^\top $.
This family of problems was designed to have the following properties: $\|c\| =1$, $Ac = 0$, $\|q\|=1$, with uniform values of LP sharpness values and \limitinger~for both the primal and dual problems, but with increasing values of the Hoffman constant of the KKT system as $\gamma \searrow 0$.  

The first three columns of the first row of Figure \ref{fig good classes} show the LP sharpness values (computed using the methodology in Section \ref{poker2}), \limitinger~values (computed using the upper bound methodology in Section \ref{subsec computable theta upperbound}), and relative distances to optima for this simple family of problems, which are all constant over the range of $\gamma \in (0,1]$. In the fourth column we report the actual iterations of Algorithm \ref{alg: PDHG with restarts} (to obtain a solution whose Euclidean distance to the optimum is at most $10^{-10}$), and the iteration bound of Theorem \ref{thm overall complexity} as well as the iteration bound \eqref{intro:complexity} based on  \cite{applegate2023faster}. (Since these two bounds are based on linear convergence rates, we report the constant outside
of the logarithmic term for simplicity, and we computed the Hoffman constant for the KKT system using the algorithm and code from \cite{pena2018algorithm}.) Notice that the bound \eqref{intro:complexity} from  \cite{applegate2023faster} grows exponentially in $\ln(1/\gamma)$ while the actual number of iterations and the bound of Theorem \ref{thm overall complexity} are constant over $\gamma \in (0,1]$.  This simple example validates the absence of the Hoffman constant from the bound in Theorem \ref{thm overall complexity}, and shows for this simple family that the actual number of iterations of Algorithm \ref{alg: PDHG with restarts} is constant as suggested by Theorem \ref{thm overall complexity}. \medskip

\begin{figure}[htbp]
	\centering
	\begin{subfigure}[b]{\textwidth}
		\centering
		\includegraphics[width=1\linewidth]{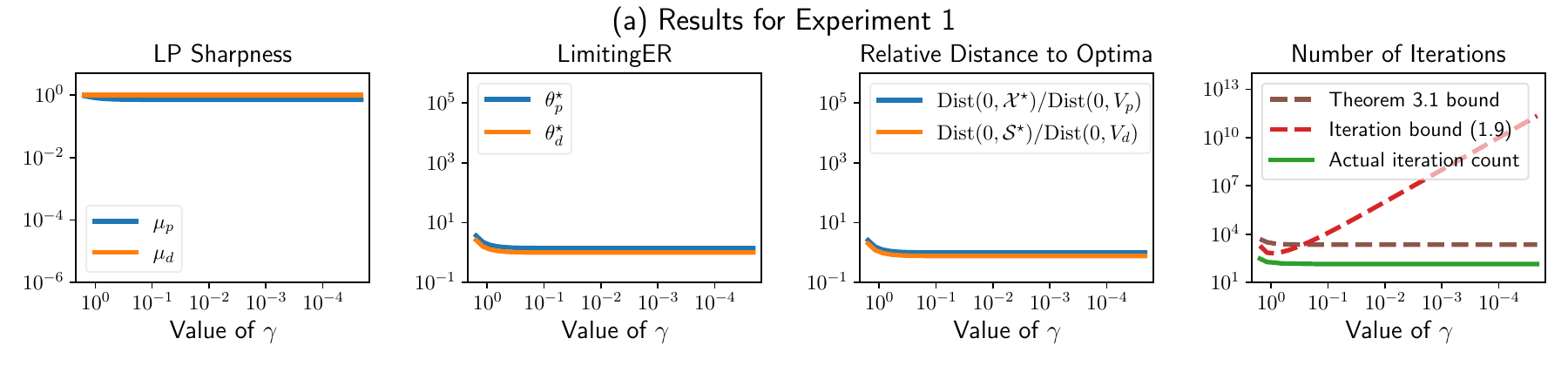}
	\end{subfigure}
	\begin{subfigure}[b]{\textwidth}	
		\centering
		\includegraphics[width=1\linewidth]{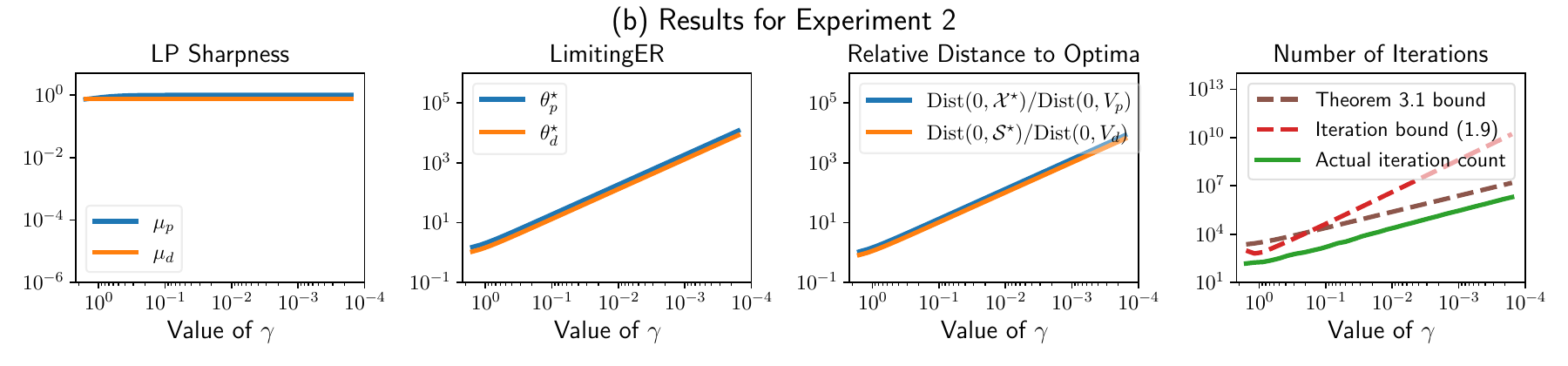}
	\end{subfigure}
	\begin{subfigure}[b]{\textwidth}
		\centering
		\includegraphics[width=1\linewidth]{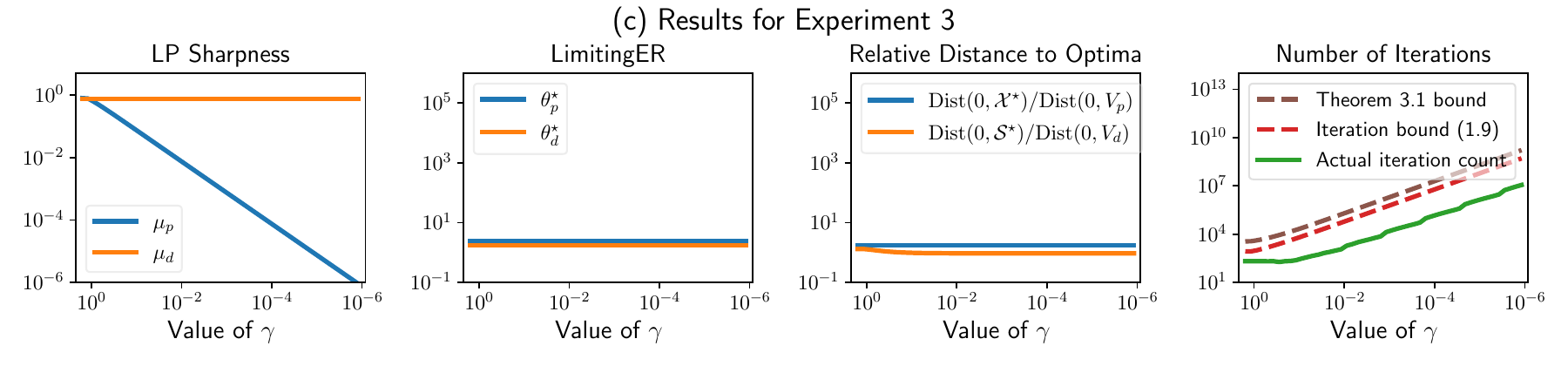}
	\end{subfigure}
	\begin{subfigure}[b]{\textwidth}
		\centering
		\includegraphics[width=1\linewidth]{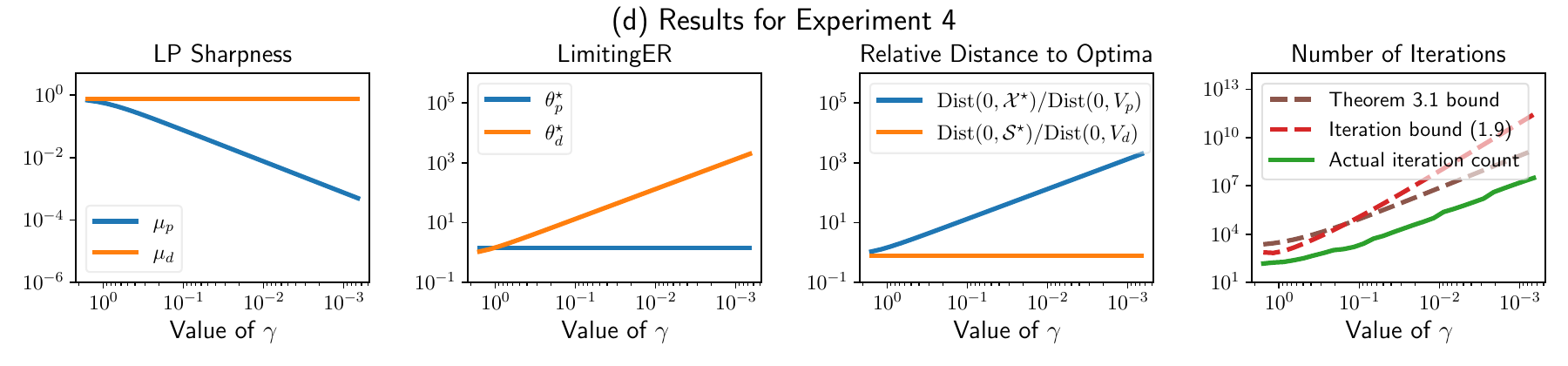} 
	\end{subfigure}
	\caption{Values of LP sharpness, \limitinger, relative distance to optima, theoretical iteration upper bound of Theorem \ref{thm overall complexity}, actual iteration count, and the iteration bound \eqref{intro:complexity} based on \cite{applegate2023faster} for the four simple validation experiments. } 
	\label{fig good classes}
\end{figure}

\noindent {\bf Experiment 2: Sensitivity to \limitinger.} This simple experiment is designed to test the sensitivity of Algorithm \ref{alg: PDHG with restarts} to the \limitinger.  Similar in approach to Experiment 1, we created the family: $A^{2}_{\gamma} := \left[\frac{\cos(\gamma)}{\sqrt{2}}  , \sin(\gamma), \frac{ \cos(\gamma)}{\sqrt{2}} \right]$, $b^2 = 1$, $c^{2}_{\gamma} := \left[\frac{\sin(\gamma)}{\sqrt{2}}  , -\cos(\gamma), \frac{\sin(\gamma)}{\sqrt{2}} \right]^\top$, where again $\|c\| =1$, $Ac = 0$, $\|q\|=1$, with constant values of LP sharpness for $\gamma \in (0,1]$, but now the \limitinger~value increases as $\gamma \searrow 0$.  The second row of Figure \ref{fig good classes} shows our results.  In this family of instances the LP sharpness values are constant even as the \limitinger~grows.  Notice that the relative distance to optima also grows similarly to the \limitinger; this must occur since the relative distances to optima are lower-bounded by the \limitinger~, see \cite{ZX2023-z}.  The fourth column shows that for the smaller values of $\gamma$, the bound in Theorem \ref{thm overall complexity} follows a similar pattern -- including the slope in the log-log plot -- as the actual iterations of Algorithm \ref{alg: PDHG with restarts}.\medskip

\noindent {\bf Experiment 3: Sensitivity to LP Sharpness.} This simple experiment is designed to test the sensitivity of rPDHG to LP sharpness.  We created the family: $A^{3}_{\gamma} := \left[\frac{1}{\sqrt{3}}, \frac{1}{\sqrt{3}}, \frac{1}{\sqrt{3}}\right]$, $b^3=1$, $c^{3}_{\gamma} :=   \cos(\gamma) \cdot \left[\frac{-1}{\sqrt{6}}, \frac{-1}{\sqrt{6}}, \frac{2}{\sqrt{6}}\right]^\top + \sin(\gamma) \cdot \left[\frac{-1}{\sqrt{2}}, \frac{1}{\sqrt{2}}, 0\right]^\top$.  Similar to the previous experiments we have $\|c\| =1$, $Ac = 0$, $\|q\|=1$, with constant values of the \limitinger~and the relative distances to optima for $\gamma \in (0,1]$, but now the primal LP sharpness $\mu_p$ value decreases as $\gamma \searrow 0$.  The third row of Figure \ref{fig good classes} shows our results. Similar in spirit to Experiment 2, the fourth column shows that for the smaller values of $\gamma$ that the bound in Theorem \ref{thm overall complexity} follows a similar pattern -- including the slope of the log-log plot -- as the actual iterations of Algorithm \ref{alg: PDHG with restarts}.\medskip

\noindent {\bf Experiment 4: Sensitivity to simultaneous changes in LP sharpness and \limitinger.} We created the family: $A^{4}_{\gamma} := \left[\sin(\gamma), \frac{\cos(\gamma)}{\sqrt{2}}, -\frac{\cos(\gamma)}{\sqrt{2}}\right]$, $b^4 =1$, $c^{4}_{\gamma} :=   \left[0, \frac{1}{\sqrt{2}}, \frac{1}{\sqrt{2}}\right]^\top$, in which the primal LP sharpness $\mu_p$ decreases and the dual \limitinger~$\thetas$ increases as $\gamma \searrow 0$, see the fourth row of Figure \ref{fig good classes} for the computational values.  Examining the fourth column of this row, we see the multiplicative effect of these two condition measures both on the theoretical bounds of Theorem \ref{thm overall complexity} as well as a doubling of the slope of the log-log plot of actual iteration counts, which aligns well with the theoretical results.\medskip

\noindent{\bf Experiment 5: Effect of the step-size rule based on LP sharpness.} Theorem \ref{thm special step size complexity} presented a step-size rule \eqref{eq smart step size 1} based on knowledge of the LP sharpness measures $\mu_p$ and $\mu_d$ that leads to a structurally superior complexity bound for Algorithm \ref{alg: PDHG with restarts}.  (However, this rule is impractical since the LP sharpness measures are neither known nor easily computable in practice.) In this experiment we test the utility of this rule using the simple family of LP instances ($A^{4}_{\gamma}$, $b^4 $, $c^{4}_{\gamma}$) described in Experiment 4, where for this simple family we know the LP sharpness measures.  Figure \ref{fig step size class 4} shows the theoretical upper bounds and the actual iteration numbers of Algorithm \ref{alg: PDHG with restarts} with standard step-sizes (Theorem \ref{thm overall complexity}) and the step-sizes of Theorem \ref{thm special step size complexity} for this family of LP instances.  The figure shows that this step-size rule reduces the actual number of iterations in line with the theory.
\medskip

\begin{figure}[htbp]
	\centering
	\includegraphics[width=0.55\linewidth]{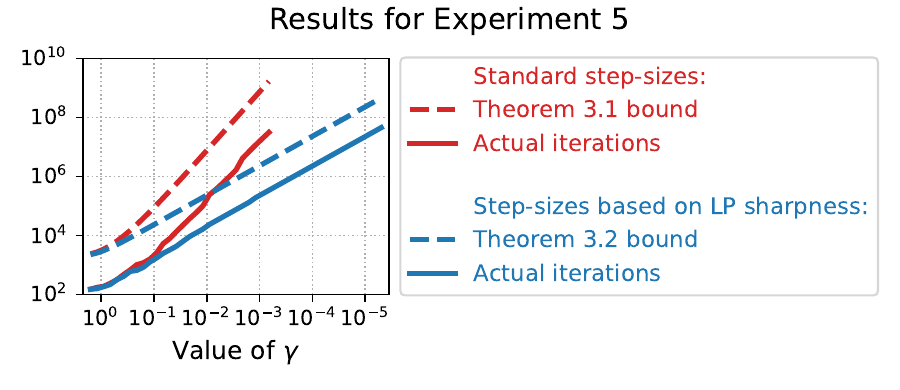}\vspace{-10pt}
	\caption{Theoretical upper bounds and the actual iteration numbers of Algorithm \ref{alg: PDHG with restarts} with standard step-sizes (Theorem \ref{thm overall complexity}), and theoretical upper bounds and the actual iteration numbers of Algorithm \ref{alg: PDHG with restarts} with the LP sharpness-based step-sizes of Theorem \ref{thm special step size complexity}, for the family of LP instances described by ($A^{4}_{\gamma}$, $b^4 $, $c^{4}_{\gamma}$) . }	\label{fig step size class 4}
\end{figure}

\noindent
\textbf{Experiment of Figure \ref{fig intro}.}
Finally, we conducted an additional experiment on the family of LP instances given in \eqref{pro sample 2dim LP}, and the results are shown in Figure \ref{fig intro} in Section \ref{sec:intro}. In this experiment, the iteration bound from Theorem \ref{thm overall complexity}, the Hoffman constant, and the iteration bound \eqref{intro:complexity} were all computed using the same approach as in the five experiments described above.

\subsection{Computational evaluation of two theory-based heuristics on the MIPLIB 2017 dataset}
In this subsection we introduce two heuristics that are inspired by our theoretical guarantees, and are designed to improve the practical performance of Algorithm \ref{alg: PDHG with restarts}. The first heuristic involves the choice of step-sizes $\tau$ and $\sigma$ for Algorithm \ref{alg: PDHG with restarts}. It was observed in \cite{applegate2023faster} that even while keeping the product of the primal and dual step-sizes $\tau$ and $\sigma$ constant, heuristically modifying the ratio $\tau/\sigma$ had the potential to improve the computational performance of rPDHG. Theoretical justification for that observation can be seen in the computational bounds for Algorithm \ref{alg: PDHG with restarts} in Theorem \ref{thm overall complexity} using different step-sizes $\tau$ for the primal and $\sigma$ for the dual in \eqref{eq smart step size 1}, and Theorem \ref{thm special step size complexity} shows -- at least in concept -- how the complexity bound can be structurally improved by appropriately varying the ratio $\tau/\sigma$ while keeping the product constant, namely $\tau \sigma = 1/(4\lambda_{\max}^2)$.  In the spirit of ``learning from experience,'' our first heuristic is essentially an adaptation of the methodology in  \cite{applegate2023faster} to learn a reasonably good step-size ratio, and works as follows.  We consider five possible choices of step-sizes, namely $(\tau,\sigma) = (40^\ell/2\lambda_{\max},40^{-\ell}/2\lambda_{\max})$ for $\ell = -1, -\tfrac{1}{2}, 0, \tfrac{1}{2},1$.  For each of these step-size pairs we run Algorithm \ref{alg: PDHG with restarts} for $5,000$ iterations from the same initial point $(x^{0,0},y^{0,0}) = (0,0)$, and then choose which of the five step-sizes to use based on the smallest relative error $\calE_r (x,y) := \frac{\|Ax^+ - b\|}{1+ \|b\|} + \frac{\|(c-A^\top y)^-\|}{1+\|c\|} + \frac{|c^\top x^+ - b^\top y|}{1 +|c^\top x^+ |+| b^\top y| }$. The heuristic essentially spends $20,000$ iterations exploring/testing for a better step-size ratio.  (We note that the relative error $\calE_r (x,y) $ is upper-bounded by the distance to optima $\calE_d(x,s)$ \eqref{xmas}, see Remark \ref{stargaze}.)

The second heuristic is also motivated by the computational bound in Theorem \ref{thm overall complexity} where we observe in \eqref{eq overall complexity} that the bound grows at least linearly in the condition number $\kappa$ of the matrix $A$ (recall the definition of $\kappa$ in \eqref{eq  def lamdab min max}). The heuristic is to compute and apply a (full-rank) row-pre-conditioner $D\in\mathbb{R}^{m\times m}$ to the equality constraints $Ax=b$ to yield the equivalent system $D Ax = D b$ for which the condition number $\kappa' := \kappa(DA):= \frac{\lambda_{\max}^+(DA)}{\lambda_{\min}^+(DA)} $ is reduced. Notice that for any such $D$, the preconditioned LP instance
\begin{equation}\label{pro: precodnitioned primal LP}
 		\min_{x\in\mathbb{R}^n}  \ c^\top x \quad	\text{s.t.}  \ D Ax = D b, \  x \ge 0
\end{equation}
and its dual problem have the identical duality-paired symmetric format \eqref{pro: primal dual reformulated LP} as the original LP instance; and so the \LPsharp, the \limitinger, and the relative distance to optima are unchanged by the preconditioner.  Indeed the only quantity in the iteration bound \eqref{eq overall complexity} that is changed is the matrix condition number $\kappa(DA)$.  In our second heuristic we work with the ``complete'' pre-conditioner $D := (AA^\top)^{-1/2}$, for which $\kappa' = \kappa(DA) =1$, which requires one (potentially expensive) matrix factorization. Other first-order methods for LP, such as \cite{lin2021admm,o2016conic}, also compute and use a single matrix factorization throughout all iterations. (When the problem is very large and computing even one matrix factorization is not tractable, \cite{applegate2021practical} proposed to use a diagonal preconditioner $D$, but there were no theoretical guarantees.)

We tested the usefulness of the two heuristics using the LP relaxations of the MIPLIB 2017 dataset~\cite{gleixner2021miplib}, which is a collection of mixed-integer programs from real applications. We took the LP relaxations of the problems in the dataset and converted them to standard form so that Algorithm \ref{alg: PDHG with restarts} can be directly applied.  We ran Algorithm \ref{alg: PDHG with restarts} to compare the following choice of heuristic strategies for step-sizes and preconditioners:
\begin{itemize}[nosep]
	\item \textbf{Simple Step-size}:  this is the simple step-size rule originally used in the proofs in \cite{applegate2023faster}, namely $\tau = \sigma = 1/(2\lambda_{\max})$,	
	\item \textbf{Learned Step-size}:  use an extra $20,000$ iterations to heuristically learn the best of five step-sizes as described above,
	\item \textbf{Preconditioner}: apply the preconditioner $D = (AA^\top)^{-1/2}$ as described above, and  
	\item \textbf{Learned Step-size+Preconditioner}: apply both of the above heuristics.
	\end{itemize}\medskip

Figure \ref{fig LP real instances} illustrates the individual effects of the two heuristics on three representative problems, namely \texttt{nu120-pr9}, \texttt{n2seq36f} and \texttt{n3705}. The horizontal axis is the number of iterations and the vertical axis is the relative error $\calE_r (x,y) := \frac{\|Ax^+ - b\|}{1+ \|b\|} + \frac{\|(c-A^\top y)^-\|}{1+\|c\|} + \frac{|c^\top x^+ - b^\top y|}{1 +|c^\top x^+ |+| b^\top y| }$ computed using the original data of the LP instance for consistency.  (The rather chaotic pattern of the early iterations of the Learned Step-size heuristic is due to the fact that the first $25,000$ iterations are used to test five different step-sizes.) For most of the LP instances in the MIPLIB 2017 we observed that the Learned Step-size heuristic enables much faster linear convergence, though \texttt{n3705} is an exception to this observation.  We also observed that the preconditioner improves convergence significantly across all problems.  

\begin{figure}[htbp]
	\centering
	\includegraphics[width=\linewidth]{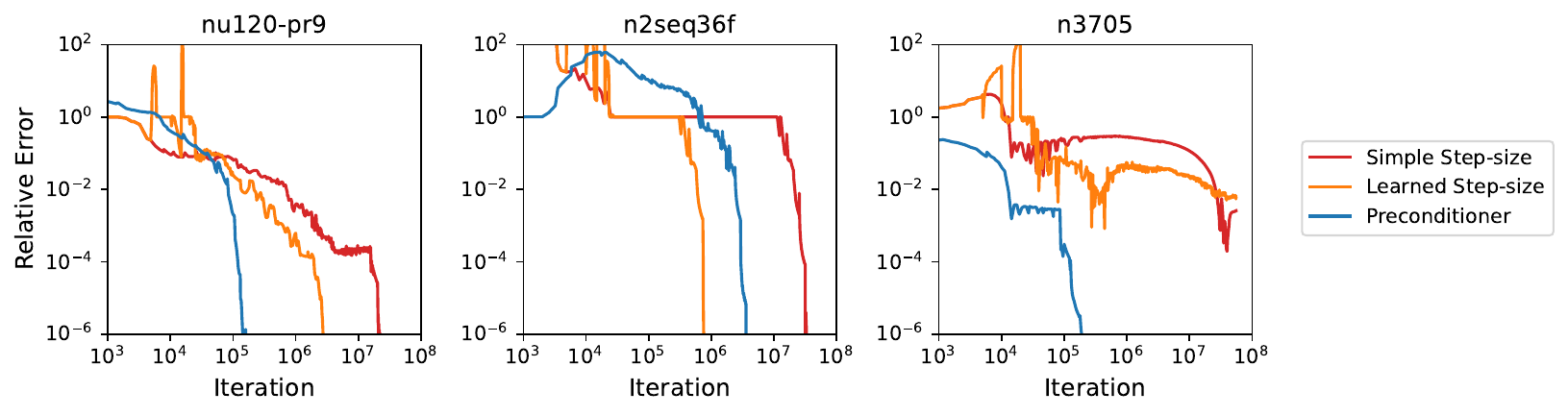}
	\caption{Performance of Algorithm \ref{alg: PDHG with restarts} using two heuristic strategies, on problems \texttt{nu120-pr9}, \texttt{n2seq36f} and \texttt{n3705}.}	\label{fig LP real instances}
\end{figure}

Last of all, we tested all four combinations of heuristics on a large subset of the MIPLIB 2017 dataset, namely all LP relaxation problems in which $mn \le 10^9$, of which there are 574 such problems in total.  For this evaluation we consider an LP instance to be ``solved'' if Algorithm \ref{alg: PDHG with restarts} computes a solution $(x,y)$ for which $\calE_r (x,y) \le 10^{-4}$. Figure \ref{fig proportion} shows the fraction of solved problems (of the 574 instances) on the horizontal axis, and the maximum iterations (leftmost plot) and the maximum runtime (rightmost plot).  Notice that these two techniques both help the rPDHG solve more problems in a shorter time. Among the two heuristics, the preconditioner plays a prominent role in reducing the number of iterations, and also in reducing runtimes. Moreover, applying both heuristics is also valuable.  Finally, it bears mentioning that if the problem is so large that the cost of working a matrix factorization is prohibitive, it is still possible to apply diagonal preconditioners to potentially improve the value of $\kappa$, see \cite{applegate2021practical}.

\begin{figure}[htbp]
	\centering
	\includegraphics[width=0.9\linewidth]{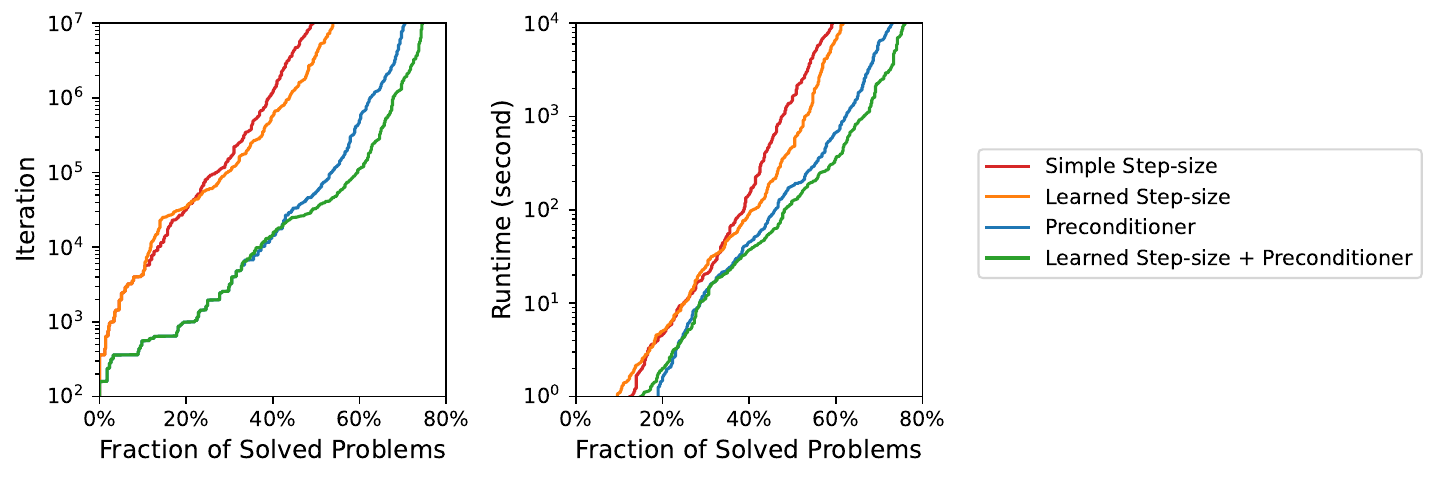}
	\caption{Performance of Algorithm \ref{alg: PDHG with restarts} combined with heuristic strategies, on the 574 LP relaxation instances from the MIPLIB 2017 dataset.}	\label{fig proportion}
\end{figure}

The above experiments show that these two heuristics -- which are motivated by our theoretical results -- have clear potential to improve the practical performance of Algorithm \ref{alg: PDHG with restarts}, which also highlights the value of the theoretical understanding in the development of practical improvements in solution methods.

\appendix

\section*{Appendix}

\section{From the Distance to Optima to the Relative Error}\label{phoebe}

Here we show that the relative error is upper-bounded by the distance to optima up to a scalar factor.  
For the pair $(x,y)$ and $s := c - A^\top y$, let $\tilde{\calE}_r (x,s)$ denote the corresponding relative error of $(x,y)$, i.e., $\tilde{\calE}_r (x,s):= \calE_r (x,y)$. Then it follows that $\tilde{\calE}_r (x,s) = \frac{\|Ax^+ - b\|}{1+ \|b\|} + \frac{\|s^-\|}{1+\|c\|} + \frac{|c^\top x^+ - q^\top (c-s)|}{1 +|c^\top x^+ |+| q^\top (c-s)| }$, in which $q := A^\top (AA^\top)^\dag b$. 

\begin{remark}[Relative error bounded by distance to optima]\label{stargaze} There exists a scalar constant $\bar c$ depending only on the data $(A,b,c)$, such that for any solution pair $(x,y)$ and $s:=c-A^\top y$, the relative error of $(x,s)$ is upper bounded by the distance to optima by a factor of $\bar c$, namely $\tilde{\calE}_r(x,s) \le \bar c \cdot \calE_d(x,s)$. One such value of $\bar c$ is $\bar c = c_0 := \frac{2\|A\|}{1+\|b\|}+ 2\|c\| + \|q\| + 1$.
\end{remark}

\begin{proof}
The error $\tilde{\calE}_r(x,s)$ is comprised of three parts, namely $\frac{\|Ax^+ - b\|}{1+ \|b\|}$, $ \frac{\|s^-\|}{1+\|c\|} $ and $ \frac{|c^\top x^+ - q^\top (c-s)|}{1 +|c^\top x^+ |+| q^\top (c-s)|}$.  For any $x^\star \in \calX^\star$ and $s^\star \in \calS^\star$, because $\|x - x^\star\| \ge \|x - x^+\|$,  we have
$$
 \frac{\|Ax^+ - b\|}{1+ \|b\|} \le 	\frac{\|Ax - Ax^+\| + \|Ax - Ax^\star\|}{1+ \|b\|} \le \frac{\|A\| \cdot(\|x - x^\star\| + \|x - x^+\|)}{1+ \|b\|} \le \frac{2\|A\| \cdot \|x - x^\star\|}{1+ \|b\|} \ 
$$
and
$$
\begin{aligned}
	&	\frac{|c^\top x^+ - q^\top (c-s)|}{1 +|c^\top x^+ |+| q^\top (c-s)|}  \le |c^\top x^+ - q^\top (c-s)| = 
	|c^\top x^+ - c^\top x^\star - q^\top (c-s) + q^\top (c-s^\star)|
	\\
	\le \ &  \|c\| \cdot (\|x - x^\star\| + \|x - x^+\|) + \|q\|\cdot \|s - s^\star\|\le 2\|c\| \cdot \|x - x^\star\|  + \|q\|\cdot \|s - s^\star\| \ .
\end{aligned}
$$
Similarly, because $\|s - s^+\| \le \|s - s^\star\|$, we have $ \frac{\|s^-\|}{1+\|c\|} = \frac{\|s - s^+\|}{1+\|c\|} \le \frac{\|s - s^\star \|}{1+\|c\|} \le \|s - s^\star \|$. Now let $x^\star := \arg\min_{\hat{x}\in\calX^\star}\|x-\hat{x}\|$ and $s^\star :=\arg\min_{\hat{s}\in\calS^\star}\|s-\hat{s}\|$, then combining the above three inequalities implies that $\tilde{\calE}_r(x,s) \le \bar c \cdot \calE_d(x,s)$. \end{proof}

\section{Proofs for Section \ref{subsec: sublinear convergence}}\label{sec:proof preliminaries}\label{guitar}

\subsection{Proof of Lemma \ref{lm: convergence of PHDG without restart}}\label{sec: proof of convergence of PHDG without restart}

\begin{proof}[Proof of Lemma \ref{lm: convergence of PHDG without restart}]
	We first examine primal near-feasibility. It holds trivially from the supposition that $ \bar x \ge 0$ that $\dist(\bar{x},\mathbb{R}^n_+)  = 0$.  Let us now show the upper bound on $ \dist(\bar{x},V_p)$.  We assume that $ \dist(\bar{x},V_p) > 0$ as otherwise the upper bound holds trivially.  Let $\hat{x} = \arg\min_{x\in V_p}\|x - \bar{x}\|$ and hence $\dist(\bar{x},V_p) = \|\hat{x} - \bar{x}\|$. Note from the standard optimality conditions that $\hat{x} - \bar{x} \in \operatorname{Im}(A^\top)$ and hence there exists $w$ such that $\hat{x} -\bar{x} = A^\top w$ and also $w\in \operatorname{Im}(A)$.  It further holds that $	 A^\top w \ne 0$, since $ \dist(\bar{x},V_p) > 0$.
	
	From the definition of $\rho(r;\cdot)$ we have:
	\begin{equation}\label{eq  lm: convergence of PHDG without restart 1}
		L(\bar{x},y) - L(x,\bar{y}) \le r \rho(r;\bar{z}) \ \ \text{for any $z \in \widetilde{B}(r;\bar{z})$ . }
	\end{equation}
	Define $y:= \bar{y} + \sqrt{\sigma}r \cdot w/ \|w\|$ and set $z := (\bar x, y)$, whereby $ z \in \widetilde{B}(r;\bar{z})
	$ and hence from \eqref{eq  lm: convergence of PHDG without restart 1} we have 
	$$ r \rho(r;\bar{z}) \ge L(\bar{x},y) - L(\bar x,\bar{y}) = (b-A \bar x)^\top (y - \bar y) = (\hat x - \bar x)^\top A^\top (y - \bar y)  = w^\top A A^\top  w  \sqrt{\sigma}r / \|w\| \ . $$ It then follows that 
	$$ \dist(\bar{x},V_p) = \|\hat{x} - \bar{x}\| = \|A^\top w\| \le  \frac{\rho(r;\bar{z})}{\sqrt{\sigma}} \cdot \frac{\|w\|}{\|A^\top w\|} = \frac{\rho(r;\bar{z})}{\sqrt{\sigma}} \cdot \frac{\|w\|}{\|w\|_{AA^\top}} \le  \frac{\rho(r;\bar{z}) }{\sqrt{\sigma} \lambda_{\min}} \ , 
	$$ where the last inequality above follows since $\lambda_{\min} = \min_{v\in\operatorname{Im}(A)}\frac{\|v\|_{AA^\top}} {\|v\|}$. This proves item {\em 1}.
	
	Let us now examine dual near-infeasibility. Notice that by definition it holds that $\bar{s} \in V_d$.  Define $x:= \bar{x} + \sqrt{\tau} r  \cdot (\bar{s})^- / \|(\bar{s})^-\|$ and set $z := (x, \bar y)$, whereby $ z \in \widetilde{B}(r;\bar{z})
	$ and hence from \eqref{eq  lm: convergence of PHDG without restart 1} we have 
	$$ r \rho(r;\bar{z}) \ge L(\bar x, \bar y) - L(x,\bar{y}) =  (c - A^\top \bar{y})^\top (\bar{x} - x) = -\bar{s}^\top  (\bar{s})^- \sqrt{\tau} r / \|(\bar{s})^-\| = \sqrt{\tau} r \|(\bar{s})^-\| \ ,  $$
	and hence $\dist(\bar{s},\mathbb{R}^n_+) = \|(\bar{s})^-\| \le \frac{1}{\sqrt{\tau}} \cdot \rho(r;\bar{z}) $.  This proves item {\em 2}.
	
	Lastly, we examine the duality gap $\gap(\bar{x},\bar{s}) = c^\top \bar{x} - b^\top \bar{y}$, and we consider two cases, namely $\bar z = 0$ and $\bar{z} \ne 0$.  If $\bar{z} = 0$, then $\gap(\bar{x},\bar{s}) = c^\top \bar{x} - b^\top \bar{y}=0$, which satisfies the duality gap bound trivially.  If $\bar{z} \neq 0$, then define $z := \bar{z} - \min\{\frac{r}{\|\bar{z}\|_M},1\}\bar{z}$, which satisfies $\|z - \bar{z}\|_M \le r$. 
Substituting the $z := \bar{z} - \min\{\frac{r}{\|\bar{z}\|_M},1\}\bar{z}$ in \eqref{eq lm: convergence of PHDG without restart 1} yields:
	\begin{equation}\label{eq  lm: convergence of PHDG without restart 7}
		r \rho(r;\bar{z}) \ge L(\bar{x},y) - L(x,\bar{y}) = \min\left\{\frac{r}{\|\bar{z}\|_M},1\right\}(c^\top \bar{x} - b^\top \bar{y}) \ ,
	\end{equation}
	which simplifies to
	\begin{equation}\label{eq  lm: convergence of PHDG without restart 8}
		c^\top \bar{x} - b^\top \bar{y}  \le   \max\{ r, \|\bar{z}\|_M\} \rho(r;\bar{z}) \ .
	\end{equation}
	This proves the desired bound in item {\em 3}.
\end{proof}

\subsection{Proof of Lemma \ref{lm: rho sublinear PDHG}}\label{sec: proof of rho sublinear PDHG}
We first recall the convergence result for PDHG in Remark 2 of \cite{chambolle2016ergodic}. 

\begin{lemma} [Sublinear convergence of PDHG (Remark 2 of \cite{chambolle2016ergodic})]\label{lm: original sublinear PDHG} Suppose that  $\tau$ and $\sigma$ satisfy \eqref{eq  general step size requirement}. For all $K \ge 1$,  and for all $x\ge 0 $ and $y$ and $z=(x,y)$, it holds that\begin{equation}\label{beringer}
	L(\bar{x}^K, y) - L(x,\bar{y}^K) \le \frac{\left\|z - z^0\right\|_M^2}{2K} \ .
\end{equation}
\end{lemma}

The actual result in Remark 2 of \cite{chambolle2016ergodic} is slightly different than above, but the logic of the proof leads to \eqref{beringer} in our set-up for PDHG for LP.

\begin{proof}[Proof of Lemma \ref{lm: rho sublinear PDHG}]
From the triangle inequality it holds that $$
\left\|z - z^0\right\|_M^2  \le \left(
\|z - \bar{z}^K\|_M + \| \bar{z}^K - z^0\|_M
\right)^2 \ ,
$$
which then implies via Lemma \ref{lm: original sublinear PDHG} that every $z \in \widetilde{B}(\|\bar{z}^K - z^0\|_M;\bar{z}^K)$ satisfies
$$L(\bar{x}^K, y) - L(x,\bar{y}^K) \le \frac{1}{2K}\left\|z - z^0\right\|_M^2 \le \frac{ \left(
	\|z - \bar{z}^K\|_M + \| \bar{z}^K - z^0\|_M
	\right)^2}{2K} \le \frac{2}{K}\left\|\bar{z}^K - z^0\right\|_M^2  \ . $$
Therefore 
\begin{equation}\label{toolatenow}\rho(\|\bar{z}^K-z^0\|_M;\bar{z}^K)\le \frac{2 \|\bar{z}^K  - z^0\|_M}{K}  \ . \end{equation}
For any $z^\star\in \calZ^\star$, it follows from Lemma \ref{lm: nonexpansive property} that $ \|\bar{z}^K  - z^0\|_M \le  \|\bar{z}^K  - z^\star\|_M +  \|z^\star  - z^0\|_M \le  2 \|z^\star  - z^0\|_M$, which combines with \eqref{toolatenow} to prove the lemma.
\end{proof}

\subsection{Proof of Lemma \ref{lm: R in the opt gap convnergence}}

\begin{proof}[Proof of Lemma \ref{lm: R in the opt gap convnergence}]
	Let  $z^\star := \arg\min_{z\in\calZ^\star} \| z - z^{a}\|_M$, we have
	\begin{equation}\label{eq lm: R in the opt gap convnergence 1}
		\begin{array}{ll}
			 \max\{ \|z^b - z^c\|_M, \|z^b\|_M\}   \\
			\le  \ \    \max\{\|z^b-z^\star\|_M + \|z^c-z^\star\|_M, \|z^b - z^\star\|_M + \|z^\star\|_M\}  \\
			\le    \ \     
			\max\{\|z^{a}-z^\star\|_M + \|z^{a}-z^\star\|_M, \|z^{a} - z^\star\|_M + \|z^\star\|_M\}   \\
			\le     \ \    
			\max\{\|z^{a}-z^\star\|_M + \|z^{a}-z^\star\|_M, \|z^{a} - z^\star\|_M +\|z^{a} - z^\star\|_M + \|z^{a} \|_M\}   \\
			=     \ \    2\|z^{a} - z^\star\|_M + \|z^{a} \|_M  \ ,
		\end{array}
	\end{equation}
where the second inequality uses the nonexpansive property (Lemma \ref{lm: nonexpansive property}).  This completes the proof. \end{proof}

\subsection{Proof of Proposition \ref{fact: norm upperbound}}\label{tension}

We first prove the following elementary inequality: 

\begin{proposition}\label{jeep} For any $y$ and $s= c - A^\top y$, it holds that $
		\dist(y,\calY^\star) \le \dist(s,\calS^\star)\cdot \frac{1}{\lambda_{\min}}$.
\end{proposition}

\begin{proof} First observe that:
$$
\dist(s,\calS^\star) = \dist(c - A^\top y, \calS^\star) = \dist(c - A^\top y, c - A^\top (\calY^\star)) = \dist(A^\top y, A^\top (\calY^\star)) = \dist_{AA^\top} (y,\calY^\star) \ . 
$$
Let $AA^\top = PD^2P^\top$ denote the thin eigendecomposition of $AA^\top$, so that $P^\top P=I$ and $D$ is the diagonal matrix of positive singular values of $A$, whereby $D_{ii} \ge \min_j D_{jj} = \lambda_{\min}$ for each $i$.  Now let $y^\star$ solve the shortest distance problem from $y$ to $\calY^\star$ in the norm $\| \cdot\|_{AA^\top}$, hence $y^* \in \calY^*$ and  $\dist_{AA^\top} (y,\calY^\star) = \| y - y^\star\|_{AA^\top}$, and let us write $y-y^\star = u + v $ where $u \in \operatorname{Im}(A)$ and $v \in \operatorname{Null}(A^\top)$. Then setting $\tilde y = y^\star + v$ and noting that $\tilde y \in \calY^\star$, we have:
\begin{equation}\label{skylight} \dist_{AA^\top}(y,\calY^\star) \le \|y - \tilde y\|_{AA^\top} = \|u\|_{AA^\top} \ . \end{equation}
Next notice that since $u \in \operatorname{Im}(A) = \operatorname{Im}(AA^\top)$, there exists $\pi$ for which $u=AA^\top \pi$, and define $\lambda = D^2P^\top \pi$. It then follows that $u = P \lambda$, $\lambda = P^\top u$, and $\|u\| = \|\lambda\|$.   We therefore have:
\begin{equation}\label{firepit}\begin{aligned}
 \dist_{AA^\top} (y,\calY^\star)^2 & = (u+v)^\top AA^\top (u+v)   \\ & = u^\top AA^\top u = \lambda^\top P^\top P D^2P^\top P\lambda = \lambda^\top D^2 \lambda \ge \lambda_{\min}^2 \|\lambda\|^2 \ ,
 \end{aligned}
\end{equation} 
and hence $\dist(s,\calS^\star) = \dist_{AA^\top} (y,\calY^\star) \ge \lambda_{\min} \|\lambda\| = \lambda_{\min} \|u\| \ge \lambda_{\min}\dist(y,\calY^\star) $, 
where the second inequality uses \eqref{skylight}.  Rearranging completes the proof.
\end{proof}

\begin{proof}[Proof of Proposition \ref{fact: norm upperbound}]
We first prove \eqref{silly}. For a given $\alpha > 0$, the statement ``$\|z\|_M \ge \alpha\|z\|_N $ for any $z$'' is equivalent to 
$$
\|z\|_M^2 - \alpha^2\|z\|_N^2 = z^\top Q_\alpha z\ge 0 \ , \ \text{ where }Q_\alpha := \left(\begin{matrix}
	\frac{1-\alpha^2}{\tau}I_n & -A^\top \\
	-A & \frac{1 - \alpha^2}{\sigma}I_m
\end{matrix}\right) \ , 
$$
and hence $\|z\|_M \ge \alpha\|z\|_N $ for any $z$ if and only if $Q_\alpha \succeq 0$.  A Schur complement argument then establishes that $Q_\alpha \succeq 0$ if and only if $(1-\alpha^2)^2/\sigma\tau  \ge \lambda_{\max}^2$, which rearranges to $\alpha \le  \sqrt{1 - \sqrt{\tau\sigma}\lambda_{\max}}$ (where the right-hand side is well defined due to \eqref{eq  general step size requirement}).  This establishes the first inequality in \eqref{silly}.  For the second inequality, note that the statement ``$\|z\|_M \le \sqrt{2}\|z\|_N $ for any $z$'' holds if and only $1/(\tau \sigma) \ge \lambda_{\max}^2$, which is satisfied due to \eqref{eq  general step size requirement}), completing the proof of \eqref{silly}.

Let us now prove \eqref{eq of lm: M norm to seperable norm}. We have
	\begin{equation}\label{eq lm: M norm to seperable norm 1}
		\begin{aligned}
			\dist_M(z,\calZ^\star) & \ = \min_{\tilde z\in\calX^\star \times \calY^\star} \|z-\tilde z\|_M \le \sqrt{2} \cdot \min_{\tilde z\in\calX^\star \times \calY^\star} \|z-\tilde z\|_{N} \\
			&  \ \le \frac{\sqrt{2} }{ \sqrt{\tau}} \dist(x,\calX^\star) + \frac{\sqrt{2} }{ \sqrt{\sigma}} \dist(y,\calY^\star)\le \frac{\sqrt{2} }{ \sqrt{\tau}} \dist(x,\calX^\star) + \frac{\sqrt{2} }{ \sqrt{\sigma}\lambda_{\min} } \dist(s,\calS^\star) \ ,
		\end{aligned}
	\end{equation}
where the first inequality utilities \eqref{silly} and the third inequality uses Proposition \ref{jeep}. 
	
We next prove \eqref{eq of lm: M norm to seperable norm 2}. Again using \eqref{silly}, we have $\dist_M(z,\calZ^\star) \ge  \sqrt{1 - \sqrt{\tau\sigma}\lambda_{\max}} \cdot	\dist_N(z,\calZ^\star)$,
and it also holds that  $\dist_N(z,\calZ^\star)  \ge \max \cdot \left\{
		\frac{1}{\sqrt{\tau}}\dist(x,\calX^\star), \frac{1}{\sqrt{\sigma}} \dist(y,\calY^\star)
		\right\}$.
Furthermore, $\dist(s,\calS^\star) = \dist(A^\top y, A^\top (\calY^\star)) \le \|A\| \cdot \dist(y,\calY^\star)$ and $\|A\| = \lambda_{\max}$, which combined with the above two inequalities yields the proof of \eqref{eq of lm: M norm to seperable norm 2}.
\end{proof}

\subsection{Proof of Theorem \ref{thm rho sublinear PDHG to LP natural}}

\begin{proof}[Proof of Theorem \ref{thm rho sublinear PDHG to LP natural}]
	Before proving Theorem \ref{thm rho sublinear PDHG to LP natural}, we first prove the following claim:
	\begin{enumerate}[nosep]
		\item Primal near-feasibiltiy: $ \dist(\bar{x}^K,V_p)  \le \frac{1}{\sqrt{\sigma} \lambda_{\min}}\cdot  \frac{4\dist_M(0,\calZ^\star)}{K}$ and $\dist(\bar{x}^K,\mathbb{R}^n_+) = 0$ , 
		\item Dual near-feasibility: $\dist(\bar{s}^K,V_d)  = 0$ and $\dist(\bar{s}^K,\mathbb{R}^n_+) \le \frac{1}{\sqrt{\tau}} \cdot\frac{4\dist_M(0,\calZ^\star)}{K}  $ , and
		\item Duality gap: $\gap(\bar{x}^K, \bar{s}^K)  \le \frac{8\dist_M(0,\calZ^\star)^2}{K}  $ . 
	\end{enumerate} 
	The upper bounds for the primal near-feasibility and dual near-feasibility follow directly from Lemmas \ref{lm: convergence of PHDG without restart} and \ref{lm: rho sublinear PDHG}.  To prove the bound on the duality gap, let  $r = \|\bar{z}^K - z^0\|_M$, and let $z^\star \in \calZ^\star$.  Then it follows from the duality gap bound in Lemma \ref{lm: convergence of PHDG without restart} that
	$$
	\begin{array}{rl}
		\gap(\bar{x}^K,\bar{s}^K) & \le  \max\{ \|\bar{z}^K - z^0\|_M, \|\bar{z}^K\|_M\} \cdot  \rho(\|\bar{z}^K - z^0\|_M;\bar{z}^K)  \le   \|\bar{z}^K\|_M  \cdot \frac{4\|z^0 - z^\star\|_M}{K} \ , 
	\end{array}
	$$ 
	where the second inequality above uses Lemma \ref{lm: rho sublinear PDHG} and $z^0 = (0,0)$.  Now we apply Lemma \ref{lm: R in the opt gap convnergence} with $z^a = z^0 = (0,0)$ and $z^b =  z^c = \bar{z}^K$, which yields $\|\bar{z}^K\|_M \le 2 \dist_M(0,\calZ^\star)$, and we obtain $\gap(\bar{x}^K,\bar{s}^K) \le \frac{8 \dist_M(0,\calZ^\star)^2}{K}$, which completes the proof of the claim. 

	Proposition \ref{fact: norm upperbound} states that $\dist_M(0,\calZ^\star)$ is upper bounded by $\frac{\sqrt{2}}{\sqrt{\tau}} \dist(0,\calX^\star) + \frac{\sqrt{2}}{\sqrt{\sigma}\lambda_{\min}} \dist(c,\calS^\star)$. Substituting this upper bound into the claim proves Theorem \ref{thm rho sublinear PDHG to LP natural}.
\end{proof}

\section{Proofs for Section \ref{restarts}}\label{appsec: proof of restarts}

\subsection{Proof of Lemma \ref{lm: complexity of PDHG with adaptive restart}}

\begin{proof}[Proof of Lemma \ref{lm: complexity of PDHG with adaptive restart}]
	
	We first bound the number of inner iterations $k$ of rPDHG between restarts in the outer loop. When $n=0$ we have $k=1$.  For $n \ge 1$ we show that $k \le  5 \condN /\beta  $.  To see this, note that it follows from Lemma \ref{lm: rho sublinear PDHG} that
	\begin{equation}\label{eq thm: complexity of PDHG with adaptive restart 1}
		\rho(\|\bar{z}^{n,k}-z^{n,0}\|_M;\bar{z}^{n,k}) \le \frac{4\dist_M(z^{n,0} ,\calZ^\star)}{k}\ .
	\end{equation}
We may presume that $\rho(\|z^{n,0}-z^{n-1,0}\|_M; z^{n,0}) \neq 0$, for otherwise it follows from Lemma \ref{lm: convergence of PHDG without restart} that $z^{n,0} \in \calZ^\star$, and Algorithm \ref{alg: PDHG with restarts} would have terminated already in line $\bf 10$.  Let us rewrite \eqref{eq thm: complexity of PDHG with adaptive restart 1} as:
	\begin{equation}\label{eq thm: complexity of PDHG with adaptive restart 2}
		\frac{	\rho(\|\bar{z}^{n,k}-z^{n,0}\|_M;\bar{z}^{n,k})}{\rho(\|z^{n,0}-z^{n-1,0}\|_M; z^{n,0})} \le \frac{4}{k}  \cdot 	\frac{\dist_M(z^{n,0} ,\calZ^\star)}{\rho(\|z^{n,0}-z^{n-1,0}\|_M; z^{n,0})} \ .
	\end{equation}
	It then follows from \eqref{eq thm: complexity of PDHG with adaptive restart 2} and \eqref{eq restart L condition} that $k =\lceil 4 \condN/\beta \rceil $ suffices to ensure that condition \eqref{catsdogs} is satisfied. Since $\condN \ge 1$ and $\beta \in (0,1)$, such a $k$ is no larger than $5\condN/\beta$.
	
	Next we prove an upper bound on the number of outer iterations. When $n = 0$ rPDHG restarts when $k = 1$, and it follows from Lemma \ref{lm: rho sublinear PDHG} and inequality \eqref{eq of lm: M norm to seperable norm} that the initial normalized duality gap is upper bounded as follows:
	\begin{equation}\label{eq thm number of restarts 1}
		\rho(\|\bar{z}^{0,1}-z^{0,0}\|_M;\bar{z}^{0,1}) \le 4 \dist_M(z^{0,0} ,\calZ^\star) \le 4 \left(\frac{\sqrt{2}}{\sqrt{\tau}}+\frac{\sqrt{2}}{\sqrt{\sigma}\lambda_{\min}} \right) \calE_d(x^{0,0},s^{0,0})\ .
	\end{equation}
Now note from \eqref{eq of lm: M norm to seperable norm 2} that 
	$$
	\begin{aligned}
		\dist_M(z^{n,0}, \calZ^\star)  
		\ge \ &  \gamma \cdot \max\left\{\frac{\dist(x^{n,0},\calX^\star)}{\sqrt{\tau}} , \frac{\dist(s^{n,0},\calS^\star)}{\sqrt{\sigma}\lambda_{\max}}  \right\} \\
		\ge \ & \gamma \cdot \min\left\{\frac{1}{\sqrt{\tau}} , \frac{1}{\sqrt{\sigma}\lambda_{\max}} \right\}  \cdot \calE_d(x^{n,0},s^{n,0}) \ ,
	\end{aligned}
	$$ 
	where  $\gamma:= \sqrt{1-\sqrt{\sigma\tau}\lambda_{\max}}$.
	Substituting this inequality back into \eqref{eq restart L condition} yields:
		\begin{equation}\label{eq thm number of restarts 2}
		\calE_d(x^{n,0},s^{n,0}) \le \frac{\condN}{\gamma} \cdot \max\{\sqrt{\tau},\sqrt{\sigma}\lambda_{\max}\} \cdot \rho(\|z^{n,0} - z^{n-1,0}\|_M;z^{n,0}) \ .
	\end{equation}
	According to the restart condition, we have $\rho(\|z^{n,0} - z^{n-1,0}\|_M;z^{n,0}) \le \beta \cdot \rho(\|z^{n-1,0} - z^{n-2,0}\|_M;z^{n-1,0})$ for each $n\ge 2$.  And noting that $z^{1,0} = \bar{z}^{0,1}$, it follows that:
	\begin{equation}\label{eq thm number of restarts 3}
		\begin{aligned}
			\rho(\|z^{n,0} - z^{n-1,0} \|_M;z^{n,0}) & \ \le \beta^{n-1} \cdot 	\rho(\|z^{1,0}-z^{0,0}\|_M;z^{1,0})  \\
			& \ =   \beta^{n-1} \cdot 	\rho(\|\bar{z}^{0,1}-z^{0,0}\|_M;\bar{z}^{0,1})   \le 4\beta^{n-1} \left(\frac{\sqrt{2}}{\sqrt{\tau}}+\frac{\sqrt{2}}{\sqrt{\sigma}\lambda_{\min}} \right) \calE_d(x^{0,0},s^{0,0}) \ ,
		\end{aligned}
	\end{equation}
	where the second inequality uses \eqref{eq thm number of restarts 1}.
	Combining \eqref{eq thm number of restarts 2} and \eqref{eq thm number of restarts 3} yields:
	\begin{equation}\label{eq thm number of restarts 4}
		\begin{aligned}
			\calE_d(x^{n,0},s^{n,0})  \le & \  \frac{\condN}{\gamma} \cdot \max\{\sqrt{\tau},\sqrt{\sigma}\lambda_{\max}\} \cdot 4\beta^{n-1} \cdot  \left(\frac{\sqrt{2}}{\sqrt{\tau}}+\frac{\sqrt{2}}{\sqrt{\sigma}\lambda_{\min}} \right) \calE_d(x^{0,0},s^{0,0}) \\
			\le &\  4\beta^{n-1} \cdot  \condN \cdot \frac{\sqrt{2}}{\gamma} \cdot 
			\left(
			\sqrt{\tau} + \sqrt{\sigma}\lambda_{\max}
			\right)\cdot\left(
			\frac{1}{\sqrt{\tau}}+ \frac{1}{\sqrt{\sigma}\lambda_{\min}}
			\right)
			\cdot \calE_d(x^{0,0},s^{0,0}) \ .
		\end{aligned}
	\end{equation}
Note that $\gamma = \sqrt{1-\sqrt{\sigma\tau}\lambda_{\max}}$, whereby \eqref{eq thm number of restarts 4} implies that for any $\eps > 0$, $\calE_d(x^{n,0},s^{n,0})  \le \eps$ for all 
	\begin{equation}
			n \ge     \left\lceil\frac{\ln\left(   
				4 \condN \cdot \frac{\sqrt{2}}{\sqrt{1-\sqrt{\sigma\tau}\lambda_{\max}}} \cdot \left(
				\sqrt{\tau} + \sqrt{\sigma}\lambda_{\max}
				\right)\cdot\left(
				\frac{1}{\sqrt{\tau}}+ \frac{1}{\sqrt{\sigma}\lambda_{\min}}
				\right) \cdot \calE_d(x^{0,0},s^{0,0} )\cdot \eps^{-1} 
				\right)}{\ln(\beta^{-1})} \right\rceil+ 1 \ ,
	\end{equation}
which must be true when
	\begin{equation}\label{sunny}
	n \ge \frac{\ln\left(   
		4 \condN \cdot \frac{\sqrt{2}}{\sqrt{1-\sqrt{\sigma\tau}\lambda_{\max}}} \cdot \left(
		\sqrt{\tau} + \sqrt{\sigma}\lambda_{\max}
		\right)\cdot\left(
		\frac{1}{\sqrt{\tau}}+ \frac{1}{\sqrt{\sigma}\lambda_{\min}}
		\right) \cdot \calE_d(x^{0,0},s^{0,0} )\cdot \eps^{-1} \cdot \beta^{-1}
		\right)}{\ln(\beta^{-1})} + 1 \ .
\end{equation}
The upper bound for the total number of \textsc{PDHGstep} iterations now follows from \eqref{sunny} and noting that the inner loop at $n=0$ uses $k=1$ iteration whereas for all $n \ge 1$ we have bounded the number of \textsc{PDHGstep} iterations by $k \le  5\condN/\beta $. \end{proof}

\subsection{Proof of Theorem \ref{thm special step size complexity}}\label{appsec: proof of special step size complexity}

\begin{proof}[Proof of Theorem \ref{thm special step size complexity}]
Substituting in the step-sizes \eqref{eq smart step size 2} and using the norm equalities $\|P_{\linVp}(c)\| = \|c\| $, $\|P_{\linVp^\bot}(c)\| = 0 $, $\|q\| = \dist(0,V_p)$, and $\|c\| = \dist(0,V_d)$ into the value of $\widetilde{\condN}$ in \eqref{eq of lm: Mdistance upper bounded by normalized duality gap2} yields: 
	\begin{equation}\label{pence2}
		\begin{array}{ll}
		 \widetilde{\condN}  =    
			  2\sqrt{2}\kappa  \left(
		\frac{4\thetax + 4 \sqrt{2}\cdot  \frac{\dist(c,\calS^\star)}{ \dist(0,V_d)}  }{\mu_p}  + \frac{3\thetas+ 4 \sqrt{2}  \cdot \frac{ \dist(0,\calX^\star) }{\dist(0, V_p)}}{\mu_d }
		\right) 	
			 \; .
		\end{array}
	\end{equation}
	Due to  \eqref{eq of lm: Mdistance upper bounded by normalized duality gap2}, the value of $\condN$ specified in \eqref{eq smart L} is at least as large as $\widetilde{\condN}$ in \eqref{pence2}, whereby it holds that 
	$
	\dist_M(z^{n,0},\calZ^\star) \le  \condN \cdot \rho(\|z^{n,0} - z^{n-1,0}\|_M;z^{n,0})   
	$
for the value of $ \condN $ specified in  \eqref{eq smart L}. Therefore condition \eqref{eq restart L condition} of Lemma \ref{lm: complexity of PDHG with adaptive restart} is satisfied, and it follows from Lemma \ref{lm: complexity of PDHG with adaptive restart} that $T$ satisfies \eqref{drewcurse} with the value of $\tilde c$ specified in the statement of the lemma, namely:
\begin{equation}\label{drewcurse4}
	T \ \le \ \frac{5}{\beta \ln(1/\beta)} \cdot   \condN  \cdot \ln\left(   
	\tilde c \cdot  \condN  \cdot \left(\frac{\calE_d(x^{0,0},s^{0,0} )}{\eps} 
	\right)\right) + 1\ .
	\end{equation}
 Substituting in the step-sizes \eqref{eq smart step size 2} and $\beta = 1/e$ into the value of $\tilde c$ in Lemma \ref{lm: complexity of PDHG with adaptive restart} we find that:
$\tilde c = 8 e \cdot \left(
1+\kappa \frac{\mu_p\|c\|}{\mu_d\|q\|}
\right)
\left(
1 + \frac{\mu_d\|q\|}{\mu_p\|c\|}
\right)$, and using $\beta = 1/e$ we finally arrive at:
\begin{equation}\label{drewcurse5 }
	T \ \le \ 5e \cdot   \condN  \cdot \ln\left(   
	8 e  \cdot  \condN  \cdot \left(\frac{\calE_d(x^{0,0},s^{0,0} )}{\eps}\right) \cdot \left(
1+\kappa \frac{\mu_p\|c\|}{\mu_d\|q\|}
\right)
\left(
1 + \frac{\mu_d\|q\|}{\mu_p\|c\|}
\right)
	\right) + 1\ .
	\end{equation}
Finally, because  $\kappa \ge 1$ and $\|q\| = \|b\|_Q$ (where $\|b\|_Q$ denotes $\|A^\top (AA^\top)^\dagger b\|$), it holds that $8 e  \cdot    \left(
1+\kappa \frac{\mu_p\|c\|}{\mu_d\|q\|}
\right)
\left(
1 + \frac{\mu_d\|q\|}{\mu_p\|c\|}
\right) \le \widehat{\calD}$ for $\widehat{\calD}$ defined in \eqref{eq smart D intro}.
Substituting this inequality into \eqref{drewcurse5 } yields \eqref{eq overall complexity2}, which completes the proof of the theorem.
\end{proof}

\section{Proofs for Section \ref{sec error ratio}}\label{subsec: proof of sec error ratio}

	First of all, we introduce the following generic LP format: 
	\begin{equation}\label{pro: genecal class of standard form LP}
			\calU^\star :=	\arg\min_{u\in\mathbb{R}^n}  \ g^\top u \quad  \text{s.t.} \  u \in \calF := V \cap \mathbb{R}^n_+ \ , 
	\end{equation}
	which generalizes the duality-paired LPs in \eqref{pro: primal dual reformulated LP} as specific instances. 
	Let $\calU^\star$, $u$, $g$, $\calF$ and $V$ be $\calX^\star$, $x$, $c$, $\calF_p$ and $V_p$, respectively, then \eqref{pro: genecal class of standard form LP} is the primal problem of \eqref{pro: primal dual reformulated LP}.
	Let $\calU^\star$, $u$, $g$, $\calF$ and $V$ be $\calS^\star$, $s$, $-q$, $\calF_d$ and $V_d$, respectively, then \eqref{pro: genecal class of standard form LP} is the dual problem  of \eqref{pro: primal dual reformulated LP}. 
	We let $\calF_{++}$ denote the strictly feasible solutions of \eqref{pro: genecal class of standard form LP}, let $f^\star$ denote the optimal objective value, let $\theta(u)$ denote the error ratio of $\calF$ at $u$,
	and let $\thetau$ denote the \limitinger~of \eqref{pro: genecal class of standard form LP}.

\subsection{Proof of Theorem \ref{thm local geometry to feasibility error ratio at optima}}
In this subsection, we prove Theorem \ref{thm local geometry to feasibility error ratio at optima} by proving its generalization to the generic LP \eqref{pro: genecal class of standard form LP}:
\begin{theorem}\label{thm generalized local geometry to feasibility error ratio at optima}
	For the generic LP \eqref{pro: genecal class of standard form LP}, suppose that the optimal solution set $\calU^\star$ is nonempty and bounded.  Then 
	\begin{equation}\label{eq generalized local geometry to feasibility error ratio at optima}
		\thetau \le \sup_{u^\star\in \calU^\star} \inf_{u_{\mathrm{int}} \in \calF_{++}}  \frac{\|u^\star - u_{\mathrm{int}}\|}{\min_i (u_{\mathrm{int}})_i} \ .
	\end{equation}
\end{theorem}

Before proving Theorem \ref{thm generalized local geometry to feasibility error ratio at optima}, we first introduce a more general result about $\theta(u)$.  Suppose $u\in V \setminus \calF$ and $u_{\mathrm{int}} \in \calF_{++}$, then the line segment from $u$ to $u_{\mathrm{int}}$ will contain a unique point that lies on the boundary of $\calF$, and let us denote this point by $\calF(u;u_{\mathrm{int}}) $.  More formally we have
\begin{equation}\label{ineq: local geometry to feasibility error ratio V_p 2}
\calF(u;u_{\mathrm{int}})  := \arg\min_{\tilde{u}}\left\{\|u - \tilde{u}\| : \tilde{u}\in \calF \ , \ \tilde{u}:= \lambda u_{\mathrm{int}} + (1- \lambda)u \ \text{for some } \lambda \in \mathbb{R}
\right\} \ .
\end{equation} The following lemma will be used in our proof of Theorem \ref{thm generalized local geometry to feasibility error ratio at optima}.

\begin{lemma}\label{lm: local geometry to feasibility error ratio}
For the general LP presentation \eqref{pro: genecal class of standard form LP}, suppose that Assumption \ref{assump: general LP} holds. Then for $u\in V \setminus \calF$ and $u_{\mathrm{int}} \in \calF_{++}$, it holds  that 
	\begin{equation}\label{ineq: geo thm basic ineq 7}
		\theta(u) \le \frac{ \| \calF(u;u_{\mathrm{int}}) - u_{\mathrm{int}}\|}{\min_i (u_{\mathrm{int}})_i} \le \frac{ \| u - u_{\mathrm{int}}\|}{\min_i (u_{\mathrm{int}})_i},
	\end{equation}
where $\calF(u;u_{\mathrm{int}}) $ is given by \eqref{ineq: local geometry to feasibility error ratio V_p 2}.

\end{lemma}

\begin{proof}
	Suppose that $u_{\mathrm{int}}$ is any given strictly feasible point in $\calF_{++}$. Let $r :=\min_i (u_{\mathrm{int}})_i$, and so $r>0$. In the line segment connecting $u_{\mathrm{int}}$ and $u$, let $v :=\calF(u;u_{\mathrm{int}})$ defined in \eqref{ineq: local geometry to feasibility error ratio V_p 2}. Then because  $r > 0$ and $u \in V \setminus \calF$, there exists $\lambda \in (0,1)$ for which $v = \lambda u_{\mathrm{int}} + (1-\lambda)u$. Also we have $v \in \partial \mathbb{R}^n_+$ and there exists $i\in[n]$ such that $v_i = 0$, whereby $0 = v_i = \lambda (u_{\mathrm{int}})_i + (1 - \lambda) u_i$ and $u_i <0$ and so:
	\begin{equation}
		\frac{(u_{\mathrm{int}})_i}{|u_i|} = \frac{1-\lambda}{\lambda} = \frac{\|v - u_{\mathrm{int}}\|}{\|v - u\|} \ .
	\end{equation}
And since $(u_{\mathrm{int}})_i \ge  r$ it follows that 
	\begin{equation}\label{ineq: geo thm basic ineq 3}
		\frac{\|v - u\|}{ |u_i |}  =  \frac{\|v - u_{\mathrm{int}}\|}{(u_{\mathrm{int}})_i} \le \frac{\|v - u_{\mathrm{int}}\|}{r} \ .
	\end{equation}
	On the left-most term of \eqref{ineq: geo thm basic ineq 3} it follows from $u_i<0$ that 
	\begin{equation}\label{ineq: geo thm basic ineq 4}
		 \frac{\|v - u\|}{ |u_i |} \ge \frac{\|v - u\|}{\| (u)^- \|} = \frac{\|v - u\|}{ \dist(u,\mathbb{R}^n_+)}\ .
	\end{equation}
	Combining \eqref{ineq: geo thm basic ineq 3} and \eqref{ineq: geo thm basic ineq 4} yields
		\begin{equation}\label{ineq: local geometry to feasibility error ratio V_p 1}
		\frac{\| u - \calF(u;u_{\mathrm{int}}) \|}{\dist(u,\mathbb{R}^n_+)}  = \frac{\| u - v\|}{\dist(u,\mathbb{R}^n_+)}  \le  \frac{\|v - u_{\mathrm{int}}\|}{(u_{\mathrm{int}})_i} \le \frac{\|v - u_{\mathrm{int}}\|}{r} = \frac{\|\calF(u;u_{\mathrm{int}}) - u_{\mathrm{int}}   \|}{\min_i (u_{\mathrm{int}})_i} \ .
	\end{equation}
	Noting that the numerator of the right-most term of \eqref{ineq: local geometry to feasibility error ratio V_p 1}  is bounded by
\begin{equation}\label{ineq: geo thm basic ineq 6}
	\begin{aligned}
		\|\calF(u;u_{\mathrm{int}}) - u_{\mathrm{int}}\| &\  \le \|\calF(u;u_{\mathrm{int}})  - u_{\mathrm{int}}\|  + \|\calF(u;u_{\mathrm{int}}) - u\| =  \| u - u_{\mathrm{int}}\|  \ , \\
	\end{aligned}
\end{equation}
and the left-most term \eqref{ineq: local geometry to feasibility error ratio V_p 1} satisfies $ \frac{\| u - \calF(u;u_{\mathrm{int}}) \|}{\dist(u,\mathbb{R}^n_+)} \ge \theta(u)$, 
we therefore have  $\theta(u) \le \frac{ \| \calF(u;u_{\mathrm{int}}) - u_{\mathrm{int}}\|}{\min_i (u_{\mathrm{int}})_i}$.
Last of all, since $\|u - u_{\mathrm{int}}\| \ge \|\calF(u;u_{\mathrm{int}})-u_{\mathrm{int}}\|$, $\theta(u)$ is also further upper bounded by $\frac{ \| u - u_{\mathrm{int}}\|}{\min_i (u_{\mathrm{int}})_i}$, which completes the proof. \end{proof}

Lemma \ref{lm: local geometry to feasibility error ratio} shows that the error ratio $\theta(u)$ is upper-bounded by the ratio of the distance from $u$ to $u_{\mathrm{int}}$ to the distance of $u_{\mathrm{int}}$ to the boundary of the nonnegative orthant.  We now use Lemma \ref{lm: local geometry to feasibility error ratio} to prove Theorem \ref{thm generalized local geometry to feasibility error ratio at optima}.

\begin{proof}[Proof of Theorem \ref{thm generalized local geometry to feasibility error ratio at optima}]
	If $\calF_{++} = \emptyset$, then the right-hand side of \eqref{eq generalized local geometry to feasibility error ratio at optima} is equal to $+\infty$ so \eqref{eq generalized local geometry to feasibility error ratio at optima} is trivially true. 
We therefore consider the case when $\calF\neq \emptyset$. For any optimal solution $u^\star \in \calU^\star$ and a given associated strictly feasible solution $u_{\mathrm{int}}\in \calF_{++}$, let the   $\{u^k\}_{k=1}^{\infty}$ be a sequence in $V$ that converges to $u^\star$, whereby from Lemma \ref{lm: local geometry to feasibility error ratio} it holds that 
	$$
	\theta(u^k) \le \frac{\| u^k- u_{\mathrm{int}}\|}{\min_i(u_{\mathrm{int}})_i} \le \frac{\| u^\star - u_{\mathrm{int}}\|  + \| u^\star - u^k\| }{\min_i(u_{\mathrm{int}})_i} \ .
	$$
	Taking the limit as $k \rightarrow \infty$ on both sides, and noting that $\lim_{k\to \infty} \| u^\star - u^k \| = 0$, it thus follows that  $	 \operatorname{\lim\sup}_{k\to\infty} \theta(u^k) \le \frac{\| u^\star - u_{\mathrm{int}}\| }{\min_i(u_{\mathrm{int}})_i}$. And since $u_{\mathrm{int}}$ is any strictly feasible point, taking the infimum over all such $u_{\mathrm{int}}$ yields
	\begin{equation}\label{eq lm local geometry to feasibility error ratio at optima}
		\operatorname{\lim\sup}_{k\to\infty} \theta(u^k) \le \inf_{u_{\mathrm{int}}\in\calF_{++} }\frac{\| u^\star - u_{\mathrm{int}}\| }{\min_i(u_{\mathrm{int}})_i}  \ .
	\end{equation}
We now seek to prove:
	\begin{equation}\label{aclu}
	\thetau := \lim_{\eps \to 0} \sup_{u\in V, \, \dist(u,\calU^\star) \le \eps} \theta(u) \le \sup_{u^\star \in \calU^\star }  \inf_{u_{\mathrm{int}}\in\calF_{++} } \frac{\| u^\star - u_{\mathrm{int}}\| }{\min_i(u_{\mathrm{int}})_i} \ .
	\end{equation}
	If this were false, there would exist $\delta > 0$ and a sequence $\{\bar u^k\}_{k=1}^{\infty}$ in $V$ such that $\dist(\bar{u}^k, \calU^\star) \le 1/k$ and $\theta(\bar{u}^k) \ge \delta+ \sup_{u^\star \in \calU^\star }\inf_{u_{\mathrm{int}}\in\calF_{++} } \frac{\| u^\star - u_{\mathrm{int}}\| }{\min_i(u_{\mathrm{int}})_i} $. Note that the points in the sequence $\{\bar u^k\}_{k=1}^{\infty}$ all lie in the compact set $\{u: \dist(u,\calU^\star)\le 1\}$ as $\calU^\star$ is convex, closed and bounded. Therefore there exists a subsequence of $\{\bar u^k\}_{k=1}^{\infty}$ that converges to a limit point in $\calU^\star$. This violates \eqref{eq lm local geometry to feasibility error ratio at optima}, and so provides a contradiction, whereby \eqref{aclu} is true, thus completing the proof.
\end{proof}

Therefore, Theorem \ref{thm local geometry to feasibility error ratio at optima} follows as a special case of Theorem \ref{thm generalized local geometry to feasibility error ratio at optima}, when  \eqref{pro: genecal class of standard form LP} is taken to be \eqref{pro: general primal LP}.

\subsection{Proof of Proposition \ref{cusco}}
In this subsection, we prove Proposition \ref{cusco} by proving its generalization to \eqref{pro: genecal class of standard form LP}:
\begin{proposition}\label{generalized cusco} Suppose $u_a \in \calU^\star$ and there exists $R_a$ for which $\calU^\star \subset \{u: \|u-u_a \| \le R_a\}$, then it holds that $ \thetau \le G^\star$ for $G^\star$ defined as follows:
	\begin{equation}\label{generalized rrr}  
			G^\star := \ \inf_{r >0, \ u\in \mathbb{R}^n}  \displaystyle\frac{R_a+\|u-u_a\|}{r} \quad \operatorname{s.t.} \ u \in V , \ u \ge r \cdot e \ . 
	\end{equation} Furthermore, let $V  = \{ \hat{u} \in \mathbb{R}^n : A\hat{u} = b\}$, then 
	\begin{equation}\label{generalized rrrr}  
			G^\star := \ \min_{v\in \mathbb{R}^n, \ \alpha \in \mathbb{R}}  R_a \alpha+\|v- \alpha u_a \| \quad \operatorname{s.t.} \ Av = \alpha  b , \ v \ge e , \ \alpha \ge 0 \ . 
	\end{equation}
\end{proposition}

\begin{proof}
	From Theorem \ref{thm generalized local geometry to feasibility error ratio at optima} we have:
	\begin{equation}\begin{aligned}
	\thetau \ \ \le \ \ & \sup_{u^\star\in \calU^\star} \inf_{u_{\mathrm{int}} \in \calF_{++}}  \displaystyle\frac{\|u^\star - u_{\mathrm{int}}\|}{\min_i (u_{\mathrm{int}})_i} \ \ \le\  \ \sup_{u \in B(u_a, R_a)} \inf_{u_{\mathrm{int}} \in \calF_{++}}  \displaystyle\frac{\|u - u_{\mathrm{int}}\|}{\min_i (u_{\mathrm{int}})_i} \\ 
	 \ \ \le \ \ & \sup_{u \in B(u_a, R_a)} \inf_{u_{\mathrm{int}} \in \calF_{++}} \displaystyle\frac{\|u- u_a\| + \| u_a - u_{\mathrm{int}}\|}{\min_i (u_{\mathrm{int}})_i} \ \ \le \ \  \inf_{u_{\mathrm{int}} \in \calF_{++}}  \frac{R_a+ \| u_a - u_{\mathrm{int}}\|}{\min_i (u_{\mathrm{int}})_i} \\ 
	 \ \ = \ \ &  \inf_{r >0, \ u\in \mathbb{R}^n}  \displaystyle\frac{R_a+\|u-u_a\|}{r} \quad \operatorname{s.t.} \ u \in V , \ u \ge r \cdot e \ ,  \end{aligned}
	\end{equation} and notice that the final right-hand side is precisely $G^\star$, which proves \eqref{generalized rrr}.  Next notice that, if $V = \{\hat{u} \in \mathbb{R}^n:A\hat{u} =  b\}$, \eqref{generalized rrr} and \eqref{generalized rrrr} are equivalent via the elementary projective transformations $u = v/\alpha$ and $(v, \alpha)  = (u/r,  1/r)$ if we add the additional constraint $\alpha>0$ to \eqref{generalized rrrr}. However, since we are only interested in the optimal objective value of \eqref{generalized rrr} and \eqref{generalized rrrr}, solving the \eqref{generalized rrrr} yields the same optimal objective value as  \eqref{generalized rrr}. 
	\end{proof}

\subsection{Proofs of Theorem \ref{thm: general theta with infeasibility} and Corollary \ref{cor: theta with infeasibility}}

In this section, we prove Theorem \ref{thm: general theta with infeasibility} and Corollary \ref{cor: theta with infeasibility} by proving their generalizations to \eqref{pro: genecal class of standard form LP}. Suppose that $V$ in the generalized LP \eqref{pro: genecal class of standard form LP} is represented by $V=\{\hat{u}: Ku=h\}$ for $K \in \mathbb{R}^{m\times n}$ and $h\in\mathbb{R}^m$, and then the generalizations are as follows:
\begin{theorem}\label{thm: generalized general theta with infeasibility}
	Suppose that $\calF$ is nonempty for \eqref{pro: genecal class of standard form LP}. Then for every $u\in V \setminus \calF$, it holds that
	\begin{equation}\label{eq generalized thm general theta with infeasibility}
		\theta(u)  \le \frac{\|K\|(1 + \|u\|)}{\distinfeas(K,h)} \ .
	\end{equation}
\end{theorem}
\begin{corollary}\label{cor: generalized theta with infeasibility}
	Suppose that \eqref{pro: genecal class of standard form LP} has an optimal solution.  If $\distinfeas(K,h) > 0$, then it holds that
	\begin{equation}\label{eq generalized thm theta with infeasibility}
		\thetau \le  \frac{\|K\|(1 + \max_{u\in \calU^\star}\|u\|)}{\distinfeas(K,h)} \ .
	\end{equation}
\end{corollary}

We first prove Corollary \ref{cor: generalized theta with infeasibility} using Theorem \ref{thm: generalized general theta with infeasibility}:
\begin{proof}[Proof of Corollary \ref{cor: theta with infeasibility}]
	By the definition of $\thetau$ and Theorem \ref{thm: generalized general theta with infeasibility},  we have
	\begin{equation}
		\begin{aligned}
					\thetau & \ =   \lim_{\eps \to 0} \left( \sup_{u\in V, \hspace{.05cm}\dist(u, \calU^\star) \le \eps} \theta(u) \right)   \le  \lim_{\eps \to 0} \left( \sup_{u\in V, \hspace{.05cm}\dist(u, \calU^\star) \le \eps} 
					\left(
					\frac{\|K\|(1 + \|u\|)}{\distinfeas(K,h)} 
					\right)
					\right) 
			\ , 
		\end{aligned}
	\end{equation}
	which is exactly \eqref{eq generalized thm theta with infeasibility}.
\end{proof}

Then we prove Theorem \ref{thm: generalized general theta with infeasibility} through the approach of constructing, for each $u\in V \setminus \calF$, a suitable perturbation $(\Delta K,\Delta h)$ of $(K,h)$ for which $\distinfeas(K + \Delta K, h + \Delta h)  = 0$ and $\theta(u) \le \frac{\|K\| ( 1 + \|u\|)}{\|\Delta K \| + \|\Delta h\|}$.  Before presenting the formal proof of the theorem, we establish several key properties of points $u\in V \setminus \calF$.

Let $\bar{u} \in V\setminus \calF$ be fixed and given, and let $\hat{u}$ be the projection of $\bar{u}$ onto $\calF$, denoted as $\hat{u}:= P_{\calF}(\bar{u})$. Then $\hat{u}$ solves the following convex quadratic program:
\begin{equation}\label{pro projection problem}
	\min_{u\in\mathbb{R}^n}\  \tfrac{1}{2} \|u - \bar{u}\|^2\  , \quad \operatorname{ s.t. } \ Ku = h, \ u\ge 0 \ , 
\end{equation}
whereby there exist multipliers $\hat y$ and $\hat s$ that together with $\hat u$ satisfy the KKT optimality conditions:
\begin{equation}\label{eq kkt}
	K \hat{u} = h, \ \hat{u}\ge 0, \ \hat{u} - \bar{u} = K^\top \hat{y} + \hat{s}, \ \hat{s} \ge 0, \ \hat{u}^\top \hat{s} =  0 \ .
\end{equation}
Note that since $\bar{u}\in V\setminus \calF$, then $K\bar{u} = K\hat{u} = h$. In addition, the following proposition holds for $(\hat{u},\hat{s},\hat{y})$:
\begin{proposition}\label{pr H seperates}
	 For any $u \in \mathbb{R}^n_+$ it holds that $
		-\|\hat{u} - \bar{u}\|^2 = \hat{s}^\top \bar{u} < \hat{s}^\top \hat{u} = 0 \le \hat{s}^\top u$.
\end{proposition}
\begin{proof} This first equality follows from \eqref{eq kkt} since $\|\hat{u} - \bar u \|^2 =  (K^\top \hat{y} + \hat{s})^\top (\hat{u} - \bar u)$,  $K(\hat{u} - \bar u) = 0$, and $\hat{u}^\top \hat{s} = 0$, whereby $\|\hat{u} - \bar{u}\|^2 = -\hat{s}^\top \bar{u}$.  The first inequality follows trivially since $\hat{s}^\top \bar{u} = -\|\hat{u}-\bar{u}\|^2< 0$ and $\hat{u}^\top \hat{s} = 0$. And the last inequality follows since $u \ge 0$ and $\hat s \ge 0$. \end{proof}
Let $H$ be the hyperplane defined as $H:=\{u: \hat{s}^\top u = 0\}$. Proposition \ref{pr H seperates} implies that $H$ separates $\bar{u}$ and $\mathbb{R}^n_+$, and $\hat{u} \in H$. We denote the projection of $\bar{u}$ onto $H$ as $\check{u}$, namely $\check{u} = P_{H}(\bar{u})$ which has closed form $\check{u} = \bar{u} - \frac{\hat{s}^\top \bar{u}}{\|\hat{s}\|^2} \cdot \hat{s}$.  For simplicity of exposition we use $a$ to denote $\check{u} - \bar{u}$ and use $b$ to denote $\hat{u} - \bar{u}$, namely
\begin{equation}\label{eq def ab}
	a  := \check{u} - \bar{u} = - \frac{\hat{s}^\top \bar{u}}{\|\hat{s}\|^2}\cdot \hat{s} \ , \quad b:= \hat{u} - \bar{u}= K^\top \hat{y} + \hat{s} \ .
\end{equation} 
\begin{proposition}\label{pr ab}
	For $a$ and $b$ defined in \eqref{eq def ab} it holds that 
	\begin{enumerate} [nosep]
		\item $a = \frac{\|b\|^2}{\|\hat{s}\|^2}\cdot \hat{s}$ , \label{claim 1}
		\item $\|b\| \ge  \|a\| > 0$ , and  \label{claim 2}
		\item $a - \frac{\|a\|^2}{\|b\|^2} \cdot b = K^\top w,$ where $ w:= -\frac{\|b\|^2}{\|\hat{s}\|^2}\cdot  \hat{y}$ .\label{claim 3}
	\end{enumerate}
\end{proposition}
\begin{proof}
	To prove item \ref{claim 1}, note from Proposition \ref{pr H seperates} that
	$-\|b\|^2 = -\|\hat{u} - \bar{u}\|^2 = \hat{s}^\top \bar{u}$, whereby $a = - \frac{\hat{s}^\top \bar{u}}{\|\hat{s}\|^2}\cdot \hat{s} = \frac{\|b\|^2}{\|\hat{s}\|^2}\cdot \hat{s}$.
	
	To prove item \ref{claim 2}, notice that since $Kb = (K\hat{u}-K \bar u) = h-h = 0$, it follows that $\|\hat{s}\|^2 = \| b - K^\top \hat{y}\|^2= \|b\|^2 + \hat{y}^\top K K^\top \hat{y} \ge \|b\|^2$. And since we have from item \ref{claim 1} that $\|a\| = \frac{\|b\|^2}{\|\hat{s}\|}$, this implies $\|a\| / \|b\| = \|b\| / \|\hat s\| \le 1$ which proves the first inequality in item \ref{claim 2}.  To prove that $\|a\| >0$, note that we cannot have $\|b\|=0$ since $\bar{u} \in V\setminus \calF$ and $\hat u \in \calF$, and we cannot have $\hat s =0$, for otherwise Proposition \ref{pr H seperates} would also imply that $b =  \hat u - \bar u = 0$.  Therefore from item \ref{claim 1} we have $\|a\| >0$.

	To prove item \ref{claim 3},  we first show that $ \|a\|^2 = b^\top a$. From item \ref{claim 1} and the definition of $b$, we have $
		a^\top b  
		=
		\frac{\|b\|^2}{\|\hat{s}\|^2}\cdot \hat{s}^\top (\hat{s} + K^\top \hat{y} )$.
	Since 	$	K(\hat{s} + K^\top \hat{y}) = K \hat{u} - K \bar{u} = h-h = 0$, it follows that 
	\begin{equation}\label{eq a equa 3}
		\hat{s}^\top (\hat{s} + K^\top \hat{y} ) = \| \hat{s}\|^2 +  \hat{s}^\top  K^\top \hat{y}  + \hat{y}^\top K(\hat{s} + K^\top \hat{y}) = \| \hat{s} + K^\top \hat{y}\|^2 =  \|b\|^2 \ ,
	\end{equation}
	 and therefore $a^\top b = \frac{\|b\|^4}{\|\hat{s}\|^2} = \|a\|^2$ (from item \ref{claim 1}). This proves $ \|a\|^2 = b^\top a$ and then we have
	\begin{equation}\label{eq ab in subspace}
		\begin{aligned}
			a - \frac{\|a\|^2}{\|b\|^2} \cdot b & =a - \frac{b^\top a}{\|b\|^2} \cdot b  = \frac{\|b\|^2}{\|\hat{s}\|^2} \cdot \hat{s} - \frac{\|b\|^2 \cdot b^\top \hat{s}}{\|\hat{s}\|^2 \|b\|^2} \cdot   b =  \frac{\|b\|^2}{\|\hat{s}\|^2} \cdot \hat{s} - \frac{\|b\|^2 \cdot (\hat{s} + K^\top \hat{y})^\top \hat{s}}{\|\hat{s}\|^2 \|b\|^2} \cdot   (\hat{s} + K^\top \hat{y}) \\
			& =  \frac{\|b\|^2}{\|\hat{s}\|^2} \cdot \hat{s} \cdot \left(
			1 - \frac{(\hat{s} + K^\top \hat{y})^\top \hat{s}}{\|b\|^2} 
			\right) - \frac{ (\hat{s} + K^\top \hat{y})^\top \hat{s}}{\|\hat{s}\|^2 } \cdot    K^\top \hat{y} \ .
		\end{aligned}
	\end{equation}
	Here, the second equality follows from item \ref{claim 1}, and the third equality uses $b = \hat{s} + K^\top \hat{y}$.
	Substituting \eqref{eq a equa 3}  into \eqref{eq ab in subspace} yields $a - \frac{\|a\|^2}{\|b\|^2} \cdot b = -\frac{\|b\|^2}{\|\hat{s}\|^2}\cdot K^\top \hat{y} = K^\top w$, thus proving item \ref{claim 3}.
\end{proof}

With the above propositions established, we now prove Theorem \ref{thm: generalized general theta with infeasibility}.

\begin{proof}[Proof of Theorem \ref{thm: generalized general theta with infeasibility}] Let $\bar{u} \in V \setminus \calF$ be given. We will use all of the notation developed earlier in this subsection, including $\hat{u} = P_{\calF}(\bar u)$, the KKT multipliers $(\hat y, \hat s)$, $H:=\{u: \hat{s}^\top u = 0 \}$, and $\check u = P_{H}(\bar{u}) = \bar{u} - \frac{\hat{s}^\top \bar{u}}{\|\hat{s}\|^2} \cdot \hat{s}$. Additionally, let $a$ and $b$ be as given in \eqref{eq def ab} and $ w =  -\frac{\|b\|^2}{\|\hat{s}\|^2}\cdot  \hat{y}$ as specified in  Proposition \ref{pr ab}.

We first examine the case when $\hat y \ne 0$, and hence $ w \ne 0$. Let us consider the following perturbations of $K$ and $h$: 
	\begin{equation}\label{def perturbation}
		\Delta K := \frac{\|a\|^2}{\|w\|^2 \|b\|^2}\cdot w (b-a)^\top,\quad \Delta h := \Delta K \bar{u} -
		\varepsilon w \ ,  
	\end{equation}
where $\varepsilon >0$ is a small positive scalar. Then: 
	\begin{equation}\label{eq farkas 1}
		\begin{aligned}
		(K + \Delta K)^\top w & = K^\top w + \Delta K^\top w = a - \frac{\|a\|^2}{\|b\|^2} \cdot b  + \frac{\|a\|^2}{\|w\|^2 \|b\|^2}\cdot  (b-a) \cdot w^\top w  \\
		& = \left(
		1 - \frac{\|a\|^2}{\|b\|^2}
		\right)a  =  \left(
		1 - \frac{\|a\|^2}{\|b\|^2}
		\right) \cdot \frac{\|b\|^2 }{\|\hat{s}\|} \cdot \hat{s} \ge 0 \ ,
	\end{aligned}
	\end{equation}
	where the second and the fourth equalities are due to Proposition \ref{pr ab}, and the final inequality is also due to Proposition \ref{pr ab}.  Furthermore, we have	
	\begin{equation}\label{eq farkas 2}
		\begin{aligned}
	w^\top (h + \Delta h) = w^\top (K + \Delta K)\bar{u} - \varepsilon \|w\|^2 =  \left(
	1 - \frac{\|a\|^2}{\|b\|^2}
	\right) \cdot \frac{\|b\|^2 }{\|\hat{s}\|} \cdot \hat{s}^\top  \bar{u} - \varepsilon \|w\|^2 < 0 \ ,
	\end{aligned}
	\end{equation}
where the strict inequality follows since $\hat{s}^\top \bar{u} <0 $ from Proposition \ref{pr H seperates}, $\|a\| \le \|b\|$ from Proposition \ref{pr ab}, and $\|w\| >0$ by supposition for this case. Examining \eqref{eq farkas 1} and \eqref{eq farkas 2} yields
	$(K + \Delta K)^\top w  \ge 0$ and $	w^\top (h + \Delta h)  < 0$, which implies via Farkas' lemma that $\soln(K+\Delta K,h + \Delta h) = \emptyset$, and hence $\distinfeas(K,h) \le \|\Delta K\| + \|\Delta h\|$.  
	
Let us now bound the size of $\|\Delta K\|$ and $\|\Delta h\|$. From \eqref{def perturbation} we have 
	\begin{equation}\label{size of delta K}
		\|\Delta K \| \le \frac{\|a\|^2}{\|w\| \|b\|^2 }\cdot \| b - a\| \ .
	\end{equation}
Since $a - \frac{\|a\|^2}{\|b\|^2} \cdot b = K^\top w$, we also have $\|K\|  \ge { \left\| a - \frac{\|a\|^2}{\|b\|^2} \cdot b \right\| \cdot \frac{1}{\|w\|}} $.  Therefore
	\begin{equation}\label{size of K}
		\begin{aligned}
			\|K\|  &\ge \left\| \|a\| \cdot \frac{a}{\|a\|} - \frac{\|a\|^2}{\|b\|} \cdot \frac{b}{\|b\|} \right\| \cdot \frac{1}{\|w\|}  =  \left\|  \frac{\|a\|^2}{\|b\|}  \cdot \frac{a}{\|a\|} -\|a\| \cdot \frac{b}{\|b\|} \right\| \cdot \frac{1}{\|w\|}  = \frac{\|a\|}{\|b\|} \cdot \|b-a\| \cdot \frac{1}{\|w\|} \ , 
		\end{aligned}
	\end{equation}
where the first equality above follows from squaring both sides and rearranging terms using Proposition \ref{pr ab}. Combining \eqref{size of delta K} and \eqref{size of K} yields 
	$\frac{\| \Delta K \|}{\|K\|} \le \frac{\|a\|}{\|b\|} = \frac{\|\check{u} - \bar{u}\|} {\|\hat{u} - \bar{u}\|} =  \frac{\dist(\bar{u},H)}{\dist(\bar{u},\calF)}. $
	From Proposition \ref{pr H seperates}, because $H$ separates $\hat{u}$ and $\mathbb{R}^n_+$, it follows that $\dist(\bar{u},H) \le \dist(\bar{u}, \mathbb{R}^n_+)$, whereby $	\frac{\| \Delta K \|}{\|K\|} \le \frac{\dist(\bar{u},\mathbb{R}^n_+)}{\dist(\bar{u},\calF)} = \frac{1}{\theta(\bar{u})}$. Moreover, since $\Delta h := \Delta K \bar{u} - \varepsilon w$, it follows that $\|\Delta h \| \le \| \Delta K \| \cdot \|\bar{u}\| + \varepsilon \|w\| \le \frac{\|\bar{u}\|}{\theta(\bar{u})} \|K\| + \varepsilon \|w\|$.  Finally, we can add the inequalities $\theta(\bar{u}) \|\Delta K\| \le \| K \| $ and $\theta(\bar{u}) \| \Delta h \| \le \|u \| \| K\| + \theta(\bar u)  \varepsilon \|w\|$ which yields after rearranging $
	\theta(u) \le \frac{\|K\| \left( 1 + \|u\| + \eps \cdot \theta(\bar u)  \|w\| / \| K \| \right)}{\|\Delta K\| + \|\Delta h\|}\cdot$.  And since $\distinfeas(K,h) \le \|\Delta K\| + \|\Delta h\|$ we have $
	\theta(u)\le   \frac{\|K\| \left( 1 + \|u\| + \eps \cdot \theta(\bar u)  \|w\| / \| K \| \right)}{\distinfeas(K,h)}$. Taking the limit as $\varepsilon \rightarrow 0$ then proves the result in the case when $\hat y \ne 0$. 

Next we consider the case when  $\hat{y} = 0$. It follows from \eqref{eq kkt} that $\hat{u} - \bar{u} = \hat{s} $. Let $I := \{i: \hat{u}_i > 0 \}$ and $J := [n]\setminus I$. Then we have $\hat{s}_I = 0$ and $\hat{u}_I = \bar{u}_I$, and $\hat{s}_J = - \bar{u}_J$. This implies that $\hat{u} = \bar{u}^+ = P_{\mathbb{R}^n_+}(\bar{u})$, and hence $\dist(\bar{u},\calF) = \|\hat{u}-\bar{u}\|= \|\hat s\| = \dist(\bar{u},\mathbb{R}^n_+)$, and hence $\theta(\bar{u})  = \dist(\bar{u},\calF)/\dist(\bar{u},\mathbb{R}^n_+) = 1$. Now let $\Delta K = -K$, and for any $\varepsilon >0$ let $\Delta h$ be any vector satisfying $\|\Delta h\| \le \varepsilon$ and $h+\Delta h \ne 0$.  Then $(K + \Delta K, h + \Delta h) = (0, h + \Delta h)$ whereby $\soln(K + \Delta K, h + \Delta h)=\emptyset$.  Therefore $\distinfeas(K,h) \le \|\Delta K\| + \|\Delta h\| \le \|K\| + \varepsilon$ for all $\varepsilon >0$, and thus $\distinfeas(K,h) \le \|K\|$.  Finally, we have in this case that $\theta(\bar{u}) = 1 \le \frac{\|K\|}{\distinfeas(K,h)}
 \le \frac{\|K\|(1+\|\bar u \|)}{\distinfeas(K,h)}$, which completes the proof.		\end{proof}

Therefore, Theorem \ref{thm: general theta with infeasibility} and Corollary \ref{cor: theta with infeasibility} follows as special cases of Theorem \ref{thm: generalized general theta with infeasibility} and Corollary \ref{cor: generalized theta with infeasibility}, when  \eqref{pro: genecal class of standard form LP} is taken to be \eqref{pro: general primal LP}.

\section{Proofs for Section \ref{sec sharpness}}\label{app: proof of sharpness section}

\subsection{Proof of Theorem \ref{thm: sharpness and perturbation}}

In this subsection, we prove Theorem \ref{thm: sharpness and perturbation} by proving its generalization to the generic LP \eqref{pro: genecal class of standard form LP}. We let $\mu$ denote the LP sharpness of  \eqref{pro: genecal class of standard form LP}. The generalization is as follows:
\begin{theorem}\label{thm: generalized sharpness and perturbation} Consider the generic LP \eqref{pro: genecal class of standard form LP} under Assumption \ref{assump: general LP}, and let $\mu$ be the \LPsharp~of \eqref{pro: genecal class of standard form LP}. Then
	\begin{equation}\label{eq of thm: generalized sharpness and perturbation}
		\mu = \inf_{\Delta g}\left\{
		\frac{\|P_{\linV}(\Delta g)\|}{\|P_{\linV}(g)\|}:
		\opt(g + \Delta g,\calF) \ne \emptyset \ \text{ and }  \opt(g + \Delta g,\calF) \not\subset \opt(g,\calF)
		\right\} \ .
	\end{equation}
\end{theorem}

The proof of Theorem \ref{thm: generalized sharpness and perturbation} is divided into two parts, where each part proves an inequality version of \eqref{eq of thm: sharpness and perturbation} in one of the two possible directions of the inequality.  The following lemma proves the ``$\le$'' version of \eqref{eq of thm: sharpness and perturbation}.

\begin{lemma}\label{lm: sharpness upper bounded by perturbation} Consider the general LP problem \eqref{pro: genecal class of standard form LP} under Assumption \ref{assump: general LP}, and let $\mu$ be the \LPsharp~of \eqref{pro: genecal class of standard form LP}. Then
	\begin{equation}\label{eq of lm: sharpness upper bounded by perturbation}
		\mu \le \inf_{\Delta g}\left\{
		\frac{\|P_{\linV}(\Delta g)\|}{\|P_{\linV}(g)\|}:
		\opt(g + \Delta g,\calF) \ne \emptyset \ \text{ and }  \opt(g + \Delta g,\calF) \not\subset \opt(g,\calF)
		\right\} \ .
	\end{equation}
\end{lemma}

\begin{proof}
	Let $\Delta g $ satisfy $\opt(g + \Delta g,\calF) \ne \emptyset$ and $\opt(g + \Delta g,\calF) \not\subset \opt(g,\calF)$, and let $\bar{u}  \in \opt(g + \Delta g,\calF) \setminus \opt(g,\calF)$. Denote the optimal objective value hyperplane by $H^\star := \{u: g^\top u = f^\star\}$. Let $\check{u} := P_{V \cap H^\star}(\bar u) = \bar u - \frac{g^\top \bar{u} - f^\star}{\|P_{\linV}(g)\|^2}\cdot P_{\linV}(g)$ and $\hat{u} := P_{\calU^\star}(\bar u)$. 
	For simplicity, we use the notation $a :=\check{u} - \bar{u}$ and $b :=\hat{u} - \bar{u}$. From Definition \ref{def: sharpness mu} we have:
	\begin{equation}\label{eq  sharp upper bounded by perturb 1}
	 \mu \le  \frac{\dist(\bar{u},V\cap H^\star)}{\dist(\bar{u},\calU^\star)} = \frac{\| \check{u}  - \bar{u}\|}{\|\hat{u} - \bar{u}\| }  = \frac{\|a\|}{\|b\|}	   \ .  
	\end{equation} 
	Our goal then is to prove that 
	\begin{equation}\label{eq  sharp upper bounded by perturb 2}
		\frac{\| a\|}{\|b\| } \le \frac{\|P_{\linV}(\Delta g)\|}{\|P_{\linV}(g)\|}  \ ,
	\end{equation} 
 and combining \eqref{eq  sharp upper bounded by perturb 2} with \eqref{eq  sharp upper bounded by perturb 1} will yield the proof. 

	Because $\check{u} \in H^\star$ and $ \hat{u} \in H^\star$, we have $g^\top \check{u} = g^\top \hat{u}$, which implies that $g^\top a = g^\top b$. Furthermore, since $\bar{u} \in \opt(g + \Delta g,\calF)$, we have $(g + \Delta g)^\top \bar{u} \le (g + \Delta g)^\top \hat{u}$, which implies that $(g + \Delta g)^\top b \ge 0$.
Substituting $g^\top a = g^\top b$ into $(g + \Delta g)^\top b \ge 0$, we obtain $ \Delta g^\top b \ge - g^\top a $. It follows directly from Assumption \ref{assump: general LP} that $\|P_{\linV}(g)\| > 0$, and dividing both sides of this last inequality by $\|P_{\linV}(g)\|$ yields 
	 \begin{equation}\label{eq  sharp upper bounded by perturb 5} 
	 	\left(\frac{\Delta g}{\|P_{\linV}(g)\|}\right)^\top b \ge - \left(\frac{ g}{\|P_{\linV}(g)\|}\right)^\top a\ .
	 \end{equation} 
	Regarding the right-hand side of \eqref{eq  sharp upper bounded by perturb 5}, note that $\check{u} = u - \frac{g^\top \bar{u} - f^\star}{\|P_{\linV}(g)\|^2}\cdot P_{\linV}(g)$ and $a = - \frac{g^\top \bar{u} - f^\star}{\|P_{\linV}(g)\|^2}\cdot P_{\linV}(g)$, whereby:
	\begin{equation}\label{eq  sharp upper bounded by perturb 6} 
		 - \left(\frac{ g}{\|P_{\linV}(g)\|}\right)^\top a  = - \left(\frac{ P_{\linV}(g)}{\|P_{\linV}(g)\|}\right)^\top a= \|a\| \ .
	\end{equation} 
Regarding the left-hand side of \eqref{eq  sharp upper bounded by perturb 5}, since $b = \hat{u} - \bar{u} \in \linV$, it follows that $(\Delta g)^\top b = (P_{\linV}(\Delta g))^\top b\le \| P_{\linV}(\Delta g)\| \| b\|$.  Substituting this inequality and  \eqref{eq  sharp upper bounded by perturb 6} back into \eqref{eq  sharp upper bounded by perturb 5} yields \eqref{eq  sharp upper bounded by perturb 2}, which as noted earlier combines with \eqref{eq  sharp upper bounded by perturb 1} to complete the proof. 
\end{proof}

Before proving the ``$\ge$'' direction, we first establish a simple proposition.  For a convex set $\calS$ let $C_\calS$ denote the recession cone of $S$, and let $C^*_{\calS}$ denote the corresponding (positive) dual cone.
\begin{proposition}\label{pr: direction u Pstaru}
	Let $\bar u \in \calF \setminus \calU^\star$, and let $\hat{u} := P_{\calU^\star}(\bar u)$, then 
	\begin{itemize}[nosep]
		\item $(\hat{u} - \bar u)^\top (u^\star - \hat{u}) \ge 0 \text{ for any } u^\star \in \calU^\star $, and \label{lm item 1}
		\item $\hat{u} - \bar u \in C_{\calU^\star}^*$ . \label{lm item 2}
	\end{itemize}
\end{proposition}

\begin{proof}
	The first assertion follows directly from the optimality conditions for the projection of $\bar u$ onto $\calU^\star$. For the second assertion, observe that for any $v \in C_{\calU^\star}$ and any $u\in \calU^\star$ we have $u + \lambda v \in \calU^\star$ for all $\lambda \ge 0$, whereby it follows from the first assertion that $(\hat{u} - \bar u)^\top (u + \lambda v  - \hat{u}) \ge 0$ for all $\lambda \ge 0$ and hence $(\hat{u} - \bar u )^\top v \ge 0$.  Since $v$ is an arbitrary point in $C_{\calU^\star}$ it holds that $\hat{u} - \bar u \in  C_{\calU^\star}^*$ .\end{proof}

\begin{lemma}\label{lm: sharpness lower bounded by perturbation new} Consider the general LP problem \eqref{pro: genecal class of standard form LP} under Assumption \ref{assump: general LP}, and let $\mu$ be the \LPsharp~of \eqref{pro: genecal class of standard form LP}. Then
	\begin{equation}\label{eq of lm: sharpness lower bounded by perturbation-sunday}
		\mu \ge \inf_{\Delta g}\left\{
		\frac{\|P_{\linV}(\Delta g)\|}{\|P_{\linV}(g)\|}:
		\opt(g + \Delta g,\calF) \ne \emptyset \ \text{ and }  \opt(g + \Delta g,\calF) \not\subset \opt(g,\calF)
		\right\} \ .
	\end{equation}
\end{lemma}

\begin{proof} Note that by setting $\Delta g := -g$ that $ \opt(g + \Delta g,\calF)  = \calF \not\subset \opt(g,\calF) = \calU^\star$ under Assumption \ref{assump: general LP}, and therefore the right-hand side of 
\eqref{eq of lm: sharpness lower bounded by perturbation-sunday} is at most $1$.  Recall from the definition of LP sharpness that $\mu \le 1$.  Therefore in the special case when $\mu = 1$ then \eqref{eq of lm: sharpness lower bounded by perturbation-sunday} holds trivially.
	
Let us therefore consider the case $\mu < 1$.  For any given $\bar{u} \in \calF \setminus \calU^\star$, we will construct a perturbation $\Delta g$ for which $\opt(g + \Delta g,\calF) \ne \emptyset$ and $\opt(g + \Delta g,\calF) \not\subset \opt(g,\calF)$, and 
\begin{equation}\label{running}
\frac{\|P_{\linV}(\Delta g)\|}{\|P_{\linV}(g)\|} \le \frac{\dist(\bar{u},V\cap H^\star)}{\dist(\bar{u},\calU^\star)} \ , \end{equation} which then implies \eqref{eq of lm: sharpness lower bounded by perturbation-sunday}.  We proceed as follows. Let $\bar{u} \in \calF \setminus \calU^\star$ be given, and define $\check{u}:= P_{V\cap H^\star} (\bar u) = \bar u - \frac{g^\top \bar{u} - f^\star}{\|P_{\linV}(g)\|^2}\cdot P_{\linV}(g)$ and $\hat{u}:= P_{\calU^\star}(\bar{u})$. Similar to the notation used in the proof of Lemma \ref{lm: sharpness upper bounded by perturbation}, let $a := \check{u} - \bar{u}$ and $b :=\hat{u} -\bar{u}$.


To construct the perturbation $\Delta g$ we first define
	\begin{equation}\label{eq sharpness lower bounded by perturbation new 1}
		\bar{b}:= \frac{-g^\top b}{\|b\|^2} \cdot b
	\end{equation}
and construct the perturbation $\Delta g := t \cdot \bar{b} $, where $t$ is the optimal objective value of the following optimization problem:
	\begin{equation}\label{eq sharpness lower bounded by perturbation new 2}
	t \ := \  \max_{\tau} \   \tau\  \quad  \operatorname{s.t.} \ \hat u \in \opt(g + \tau \bar{b}, \calF) \ .
	\end{equation} 
We aim to show that $t \in (0,1]$.  Towards the proof of this inclusion, let $\calE1$ be the set of extreme points of $\calF$ that are in $\calU^\star$, and let $\calE2$ be the set of extreme points of $\calF$ that are not in $\calU^\star$.  Similarly, let $\calR1$ be the set of extreme rays of $\calF$ that are also extreme rays of $\calU^\star$, and let $\calR2$ be the set of extreme rays of $\calF$ that are not also extreme rays of $\calU^\star$. Note that $\calE1$ and $\calE2$ are finite sets, and $\calR1$ and $\calR2$ are also finite sets.  We can therefore rewrite \eqref{eq sharpness lower bounded by perturbation new 2} as:
	\begin{align}
		\ \ \ \ \ \text{OP}:	\ \ \ \ \ t \ := \  \max_{\tau} &  \  \  \ \tau \ \quad \\
		\operatorname{s.t.} \ & \ \ \ 	\tau  \cdot \bar{b}^\top (\hat{u} - v^i)  & \le &  -  g^\top (\hat{u} - v^i)     \ \ \ \ \ \ \text{for each } v^i \in \calE1   \label{eq sharpness lower bounded by perturbation new 3} \\
		 \ & \ \ \ 	\tau  \cdot \bar{b}^\top (\hat{u} - v^i) &  \le &  -  g^\top (\hat{u} - v^i)     \ \ \ \ \ \ \text{for each } v^i \in \calE2   \label{eq sharpness lower bounded by perturbation new 33} \\
	\ & \ \ \ 	 \tau \cdot  \bar{b}^\top r^i & \ge & -   g^\top r^i    \ \ \ \ \ \ \  \ \ \ \ \ \ \ \text{for each } r^i \in \calR1   \label{eq sharpness lower bounded by perturbation new 4} \\ 
		\ & \ \ \ 	 \tau \cdot  \bar{b}^\top r^i & \ge & -   g^\top r^i    \ \ \ \ \ \ \  \ \ \ \ \ \ \ \text{for each } r^i \in \calR2   \label{eq sharpness lower bounded by perturbation new 44}	\end{align}

First observe that $-g^\top b  = -g^\top (\hat{u} -\bar{u})>0$, whereby $\bar b$ is a positive scaling of $b$.  Also notice that $\tau = 0$ is feasible for OP, because $\hat u \in \opt(g ,\calF)$. It implies that the right-hand sides of \eqref{eq sharpness lower bounded by perturbation new 3} and \eqref{eq sharpness lower bounded by perturbation new 33} are nonnegative and the right-hand sides of \eqref{eq sharpness lower bounded by perturbation new 4} and \eqref{eq sharpness lower bounded by perturbation new 44} are nonpositive. 
Next we show that $t \le 1$. To see this, note that if $\tau >1$ then 
	$$
	 (g + \tau  \cdot \bar{b})^\top (\hat{u} - \bar u)  =  ( g + \tau  \cdot \bar{b})^\top b  = (g +  \bar{b})^\top b + (\tau - 1) \bar{b}^\top  b  = (\tau - 1) \cdot (-g^\top b) > 0 \ ,
	$$
whereby $\hat u \notin \opt(g + \tau \bar b ,\calF)$ and thus $\tau$ is not feasible for \eqref{eq sharpness lower bounded by perturbation new 2}. 
	
It thus remains to show that $t >0 $, which we will demonstrate by examining the constraints of OP.  Examining the constraints \eqref{eq sharpness lower bounded by perturbation new 3}, when $v^i \in \calE1$ the corresponding right-hand side is equal to $0$ while $  \bar{b}^\top(\hat{u} - v^i) \le 0$ (from Proposition \ref{pr: direction u Pstaru}), so these constraints are satisfied for all $\tau \ge 0$. Examining the constraints \eqref{eq sharpness lower bounded by perturbation new 33}, when $v^i \in \calE2$ the corresponding right-hand side is strictly positive so \eqref{eq sharpness lower bounded by perturbation new 33} is satisfied for all sufficiently small  $\tau > 0$. Examining the constraints \eqref{eq sharpness lower bounded by perturbation new 4}, when $r^i \in \calR1$ the corresponding right-hand side is equal to $0$ while $  \bar{b}^\top r^i    \ge 0$ (from Proposition \ref{pr: direction u Pstaru}), so these constraints are satisfied for all $\tau \ge 0$. And examining \eqref{eq sharpness lower bounded by perturbation new 44}, when $r^i \in \calR2$ the corresponding right-hand side is strictly negative, so \eqref{eq sharpness lower bounded by perturbation new 44} is satisfied for all sufficiently small  $\tau > 0$. Therefore, there exists $	\tau > 0$ that satisfies all the constraints of OP, which implies that $t > 0$.
	
Now let us show that $\opt(g + \Delta g,\calF) \ne \emptyset$ and $\opt(g + \Delta g,\calF) \not\subset \opt(g,\calF)$. It follows from \eqref{eq sharpness lower bounded by perturbation new 2} that $(g+ \Delta g)^\top (\hat{u} - u ) \le 0 $ for any $u \in \calF$ and therefore $\hat{u} \in \opt(g + \Delta g,\calF)$ and therefore $\opt(g + \Delta g,\calF) \ne \emptyset$. Now notice that when $\tau$ is optimal for \eqref{eq sharpness lower bounded by perturbation new 2} (and its equivalent formulation OP), there exists either $v^i \in \calE2$ for which the corresponding constraint in \eqref{eq sharpness lower bounded by perturbation new 33} is active, or $r^i \in \calR2$ for which the corresponding constraint in \eqref{eq sharpness lower bounded by perturbation new 44} is active (or both). In the former case, $v^i \in \opt(g + \Delta g,\calF) $ and in the latter case $\hat{u} + r^i \in \opt(g + \Delta g,\calF) $. And in either case, we have $\opt(g + \Delta g,\calF) \not\subset \opt(g,\calF)$.
	
Last of all, because $g^\top a = g^\top b$, $a \in \linV$, and $\bar{b}\in\linV$, we have
	$$
	\|P_{\linV}(\Delta g)\| = t \cdot \|\bar{b}\| = t \cdot  \frac{-g^\top b}{\|b\|} \le  \frac{-g^\top b}{\|b\|} = \frac{-g^\top a}{\|b\|} =  \frac{-P_{\linV}(g)^\top a}{\|b\|}    \le \frac{\|a\|}{\|b\|} \cdot   \|P_{\linV}(g)\| = \frac{\dist(\bar{u},V\cap H^\star)}{\dist(\bar{u},\calU^\star)}	\cdot   \|P_{\linV}(g)\| \ . $$
	This shows \eqref{running} and completes the proof. \end{proof}

Last of all, Theorem \ref{thm: generalized sharpness and perturbation} follows by combining Lemmas \ref{lm: sharpness upper bounded by perturbation} and \ref{lm: sharpness lower bounded by perturbation new}. Theorem \ref{thm: sharpness and perturbation} follows by applying Theorem \ref{thm: generalized sharpness and perturbation} to \eqref{pro: general primal LP}.

\subsection{Proof of Theorem \ref{thm: sharpness from adjacent edges}}


Our proof of Theorem \ref{thm: sharpness from adjacent edges} relies on the following two elementary propositions. For any $\eps \ge 0$ define $S_{\eps}$ to be the level set whose objective function value is exactly $\eps$ larger than the optimal objective value, namely $S_{\eps}:= \calF_p \cap \{x: c^\top x = f^\star + \eps\} $. 

\begin{proposition}\label{lm: sharpness of slices} If $\eps_2 \ge \eps_1 >0$ and $S_{\eps_i}\neq \emptyset$ for $i=1,2$, then $
		\inf_{x \in S_{\eps_2}} G(x) \ge 	\inf_{x \in S_{\eps_1}} G(x)$.
\end{proposition}
\begin{proof}
	For any given $x \in S_{\eps_2}$, let $\hat{x} : = P_{\calX^\star}(x)$ and $t:= \eps_1/\eps_2$.  Then for $x_t := t x + (1-t) \hat{x}$ it holds that $x_t \in S_{\eps_1}$.  	
	Moreover, since $x_t$ lies on the segment between $x$ and $\hat{x}$ and $\hat{x}$ is the Euclidean projection of $x$ onto the convex set $\calX^\star$, we have $\dist(x_t,\calX^\star)=\|x_t-\hat{x}\|=t\,\|x-\hat{x}\|=t\,\dist(x,\calX^\star)$.
	In addition, because $x_t\in V_p$ and $V_p\cap H_p^\star=\{u\in V_p: c^\top u=f^\star\}$, we can write $\dist(x_t, V_p\cap H_p^\star)
	= \frac{c^\top x_t - f^\star}{\|P_{\linVp}(c)\|}
	= \frac{t(c^\top x - f^\star)}{\|P_{\linVp}(c)\|}
	= t\,\dist(x, V_p\cap H_p^\star).$
	Therefore, $
	G(x_t)   = \frac{t\cdot \dist(x,V_p\cap H_p^\star)}{t \cdot \dist(x,\ \calX^\star)}  	= G(x)$. Since this equality holds for all $x \in S_{\eps_2}$, it follows that $	\inf_{x \in S_{\eps_2}} G(x) \ge 	\inf_{x \in S_{\eps_1}}  G(x) $, which proves the proposition.
\end{proof}

\begin{proposition}\label{lm:mono_sharp}
	Let $x^\star \in \calX^\star$ and $v \in \linVp$ satisfy $c^\top v > 0$.  If $t_2\ge t_1 >0$ and $x^\star + t_{i} \cdot v \in \calF_p$ for $i=1,2$, then
			$G(x^\star + t_1 \cdot v) \ge G(x^\star + t_2 \cdot v)$. 
\end{proposition}
\begin{proof}
	For $t \ge 0$ define $g_1(t) :=\dist(x^\star+t\cdot v, \ V_p \cap H_p^\star)$ and $g_2(t):=\dist(x^\star + t\cdot v,\ \calX^\star)$.  Then for $t >0$ we have $G(x^\star + t\cdot v) = g_1(t)/g_2(t)$. Notice that $g_1(t) = \frac{t\cdot c^\top v }{ \|P_{\linVp}(c)\| }$ which is a nonnegative increasing linear function of $t$ with $g_1(0) = 0$. Also notice that $g_2(t)$ is convex and nonnegative for $t \ge 0$, and $g_2(0) = 0$, whereby $g_2(t)$ is a monotonically increasing nonnegative convex function for $t \ge 0$. Therefore $g_2(t)/g_1(t)$ is monotonically increasing on $t > 0$, whereby $G(t) = g_1(t)/g_2(t)$ is monotonically decreasing on $t>0$, which proves the proposition. \end{proof}
	
\begin{proof}[Proof of Theorem \ref{thm: sharpness from adjacent edges}] First notice from the definition of $\mu$ that \begin{equation}\label{18th}
	\mu_p \le \min \left\{ R_1(\e^1), R_1(\e^2),\dots,R_1(\e^{m_1}), R_2(\f^1;\bar{\eps}),R_2(\f^2;\bar{\eps}),\dots,R_2(\f^{m_2};\bar{\eps}) \right\} \  . \end{equation}
For $\eps >0$ let us consider the level set $S_{\eps} $ and suppose that $S_{\eps} \ne \emptyset$, and let $\calE\calS_\eps$ denote the extreme points of $S_{\eps}$.  Then we claim that 
$$
\begin{aligned}
	\mu_p = \inf_{x \in \calF_p \setminus \calX^\star} G(x) &= \inf_{\eps>0} \left( \inf_{ x \in S_\eps} G(x) \right) = \lim_{\eps \to 0} \left( \inf_{ x \in S_\eps} G(x) \right) \\  
	&=\lim_{\eps \to 0} \left( \frac{\tfrac{\eps }{ \|P_{\linVp}(c)\|}}{\sup_{x\in S_\eps} \dist(x, \calX^\star)}\right)  = \lim_{\eps \to 0} \left( \frac{\tfrac{\eps }{ \|P_{\linVp}(c)\|}}{\max_{v^i\in \calE\calS_\eps} \dist(v^i, \calX^\star)}\right)  \ .
\end{aligned}
$$
Here the first equality is the definition of $\mu_p$, the second equality is a restatement of the first expression, and the third equality is due to the monotonicity property of the sharpness function on the level sets $S_{\eps}$ from Proposition \ref{lm: sharpness of slices}.  To prove the fourth equality, observe that the numerator of $G(x)$ is $\dist(x, V_p\cap H_p^\star)$ which equals the constant $\tfrac{\eps }{ \|P_{\linVp}(c)\|}$ for $x\in S_{\eps}$, and the denominator of $G(x)$ is $\dist(x, \calX^\star)$. For the fifth equality, observe that $\dist(\cdot, \calX^\star)$ is convex in $x$ and bounded from above and below on $S_\eps$, and so attains its maximum at an extreme point of  $S_\eps$.  	

Next notice that since $\calF_p$ is a polyhedron and $S_\eps$ is a level set of $\calF_p$, there exists $\bar\eps >0$ such that for all $\eps \in (0, \bar{\eps})$ the extreme points of $S_\eps$ all lie in the edges of $\calF_p$ emanating away from $\calX^\star$, namely $\calM \cap S_\eps = \calE\calS_\eps$. Therefore using the definition of $\calM = \calM_1 \cup \calM_2$ we have:
\begin{equation}\label{eq sharpness from adjacent edge 2}
	\begin{aligned}
		\mu_p &= \lim_{\eps \to 0} \left( \frac{\tfrac{\eps }{ \|P_{\linVp}(c)\|}}{\max_{v^i\in \calE\calS_\eps} \dist(v^i, \calX^\star)}\right)  = \min\left\{ \min_{\e\in\calM_1}  \left( \inf_{x \in \e: c^\top x \le f^\star + \bar{\eps}} G(x) \right) , \ 
		\min_{\f\in\calM_2}  \left( \inf_{x \in \f: c^\top x \le f^\star + \bar{\eps}} G(x) \right) 
		\right\}
		\ .
\end{aligned}
\end{equation}
It follows from Proposition \ref{lm:mono_sharp} that $G(x)$ is decreasing on any edge emanating away from $\calX^\star$.  Thus for $\e=[v^1, v^2] \in \calM_1$ with $v^1 \in \calX^\star$ and $v^2 \notin \calX^\star  $ 
we have $\inf_{x \in \e: c^\top x \le f^\star + \bar{\eps}} G(x)  \ge  G(v^2)= R_1(\e)$, and similarly for $\f=[v; r] \in \calM_2$ with $v \in \calX^\star$ and $r$ being an extreme ray of $\calF$ we have 
$\inf_{x \in \f: c^\top x \le f^\star + \bar{\eps}} G(x)  =  G(v+\bar \eps r)= R_2(\f,\bar\eps)$.
	Substituting these inequalities back into \eqref{eq sharpness from adjacent edge 2} yields	$$
	\mu_p \ge \min \left\{ R_1(\e^1), R_1(\e^2),\dots,R_1(\e^{m_1}), R_2(\f^1;\bar{\eps}),R_2(\f^2;\bar{\eps}),\dots,R_2(\f^{m_2};\bar{\eps}) \right\} \  , $$ which combined with \eqref{18th} yields the proof.
\end{proof}

\subsection*{Acknowledgements}

The authors are grateful to Kim-Chuan Toh, Haihao Lu, Jinwen Yang, and Guanghao Ye for discussions and input about the ideas in this paper and related research ideas. The authors also thank the anonymous referees for extremely helpful comments and suggestions that improved the quality of the manuscript.



\begingroup
\setstretch{0.95}  
\bibliographystyle{plain}
\bibliography{reference}

@article{burke1993weak,
  title={Weak sharp minima in mathematical programming},
  author={James V Burke and Michael C Ferris},
  journal={SIAM Journal on Control and Optimization},
  volume={31},
  number={5},
  pages={1340--1359},
  year={1993},
  publisher={SIAM}
}

@Book{wright,
  author =   {Wright, Stephen J},
  title =    {Primal-Dual Interior-Point Methods},
  PUBLISHER =    {Society for Industrial and Applied Mathematics
                  (SIAM)},
  address   =  {Philadelphia},
  year =     1997
}

@article{ZX2023-z,
	title={On the relation between {LP} sharpness and Limiting Error Ratio and its complexity implications for Restarted {PDHG}},
	author={Xiong, Zikai and Freund, Robert M},
	journal={arXiv preprint arXiv:2312.13773},
	year={2023}
}

@article{applegate2023faster,
	title={Faster first-order primal-dual methods for linear programming using restarts and sharpness},
	author={Applegate, David and Hinder, Oliver and Lu, Haihao and Lubin, Miles},
	journal={Mathematical Programming},
	volume={201},
	number={1-2},
	pages={133--184},
	year={2023},
	publisher={Springer}
}

@book{grunbaum,
  title={Convex Polytopes},
  author={Gr{\"u}nbaum, Branko and Klee, Victor and Perles, Micha A and Shephard, Geoffrey Colin},
  volume={16},
  year={1967},
  publisher={Springer}
}

@article{EpeFre00,
  title={Condition number complexity of an elementary algorithm for computing a reliable solution of a conic linear system},
  author={Epelman, Marina and Freund, Robert M},
  journal={Mathematical Programming},
  volume={88},
  number={3},
  pages={451--485},
  year={2000},
  publisher={Springer}
}

@article{chambolle2016ergodic,
	title={On the ergodic convergence rates of a first-order primal--dual algorithm},
	author={Chambolle, Antonin and Pock, Thomas},
	journal={Mathematical Programming},
	volume={159},
	number={1-2},
	pages={253--287},
	year={2016},
	publisher={Springer}
}

@article{hoffman1952approximate,
	title={On Approximate Solutions of Systems of Linear Inequalities},
	author={Hoffman, Alan J},
	journal={Journal of Research of the National Bureau of Standards},
	volume={49},
	number={4},
	year={1952},
	pages={263--265},
}

@article{pang1997error,
	title={Error bounds in mathematical programming},
	author={Pang, Jong-Shi},
	journal={Mathematical Programming},
	volume={79},
	number={1-3},
	pages={299--332},
	year={1997},
	publisher={Springer}
}

@inproceedings{mangasarian1981condition,
	title={A condition number for linear inequalities and linear programs},
	author={Mangasarian, Olvi L},
	editors={G. Bamberg and O. Opitz},
	booktitle={Proceedings of 6th Symposium {\"u}ber Operations Research, Universit{\"a}t Augsburg},
	pages={3--15},
	year={1981}
}

@article{bergthaller1992distance,
	title={The distance to a polyhedron},
	author={Bergthaller, C. and Singer, Ivan},
	journal={Linear Algebra and its Applications},
	volume={169},
	pages={111--129},
	year={1992},
	publisher={Elsevier}
}

@article{luo1994perturbation,
	title={Perturbation analysis of a condition number for linear systems},
	author={Luo, Zhi-Quan and Tseng, Paul},
	journal={SIAM Journal on Matrix Analysis and Applications},
	volume={15},
	number={2},
	pages={636--660},
	year={1994},
	publisher={SIAM}
}

@inproceedings{lewis1998error,
	title={Error bounds for convex inequality systems},
	author={Lewis, Adrian S and Pang, Jong-Shi},
	booktitle={Generalized Convexity, Generalized Monotonicity: Recent Results},
	pages={75--110},
	year={1998},
	publisher={Springer}
}

@book{nesterov1994interior,
	title={Interior-point polynomial algorithms in convex programming},
	author={Nesterov, Yurii and Nemirovskii, Arkadii},
	year={1994},
	publisher={SIAM}
}

@article{todd1990centered,
	title={A centered projective algorithm for linear programming},
	author={Todd, Michael J and Ye, Yinyu},
	journal={Mathematics of Operations Research},
	volume={15},
	number={3},
	pages={508--529},
	year={1990},
	publisher={INFORMS}
}

@article{hu2000perturbation,
	title={Perturbation analysis of global error bounds for systems of linear inequalities},
	author={Hu, Hui},
	journal={Mathematical Programming},
	volume={88}, 
	pages={277--284},
	year={2000},
	publisher={Springer}
}

@article{goffin1980relaxation,
	title={The relaxation method for solving systems of linear inequalities},
	author={Goffin, Jean-Louis},
	journal={Mathematics of Operations Research},
	volume={5},
	number={3},
	pages={388--414},
	year={1980},
	publisher={INFORMS}
}

@article{renegar1994some,
	title={Some perturbation theory for linear programming},
	author={Renegar, James},
	journal={Mathematical Programming},
	volume={65},
	number={1-3},
	pages={73--91},
	year={1994},
	publisher={Springer}
}

@TECHREPORT{son1,
AUTHOR="Sonnevend, Gy.",
TITLE="An `analytic' center for polyhedrons and new classes of global
algorithms for linear (smooth, convex) optimization",
NOTE={Preprint},
INSTITUTION={Department of Numerical Analysis, Institute of Mathematics,
E\"otv\"os University, 1088,
Budapest, Muzeum K\"orut 6-8},
YEAR=1985}

@article{chambolle2011first,
	title={A first-order primal-dual algorithm for convex problems with applications to imaging},
	author={Chambolle, Antonin and Pock, Thomas},
	journal={Journal of Mathematical Imaging and Vision},
	volume={40},
	pages={120--145},
	year={2011},
	publisher={Springer}
}

@book{ryu2022large,
	title={{L}arge-{S}cale {C}onvex {O}ptimization: {A}lgorithms \& {A}nalyses via {M}onotone {O}perators},
	author={Ryu, Ernest K and Yin, Wotao},
	year={2022},
	publisher={Cambridge University Press}
}

@article{liang2016convergence,
	title={Convergence rates with inexact non-expansive operators},
	author={Liang, Jingwei and Fadili, Jalal and Peyr{\'e}, Gabriel},
	journal={Mathematical Programming},
	volume={159},
	pages={403--434},
	year={2016}, 
	publisher={Springer}
}

@inproceedings{applegate2021practical,
	title={Practical large-scale linear programming using primal-dual hybrid gradient},
	author={Applegate, David and D{\'\i}az, Mateo and Hinder, Oliver and Lu, Haihao and Lubin, Miles and O'Donoghue, Brendan and Schudy, Warren},
	booktitle={Advances in Neural Information Processing Systems},
	 volume = {34},
	pages={20243--20257},
	year={2021},
	publisher = {Curran Associates, Inc.},
}

@article{applegate2021infeasibility,
  title={Infeasibility detection with primal-dual hybrid gradient for large-scale linear programming},
  author={Applegate, David and D{\'\i}az, Mateo and Lu, Haihao and Lubin, Miles},
  journal={SIAM Journal on Optimization},
  volume={34},
  number={1},
  pages={459--484},
  year={2024},
  publisher={SIAM}
}

@article{lu2022infimal,
	title={On the Infimal Sub-differential Size of Primal-Dual Hybrid Gradient Method},
	author={Lu, Haihao and Yang, Jinwen},
	journal={arXiv preprint arXiv:2206.12061},
	year={2022}
}

@inproceedings{polyak1979sharp,
	title={Sharp minima},
	author={Polyak, Boris Teodorovic},
	booktitle={Proceedings of the IIASA Workshop on Generalized Lagrangians and Their Applications, Laxenburg, Austria. Institute of Control Sciences Lecture Notes, Moscow},
	year={1979}
}

@article{yang2018rsg,
	title={{RSG}: Beating subgradient method without smoothness and strong convexity},
	author={Yang, Tianbao and Lin, Qihang},
	journal={Journal of Machine Learning Research},
	volume={19},
	number={1},
	pages={236--268},
	year={2018},
	publisher={JMLR. org}
}

@article{pena2018algorithm,
	title={An algorithm to compute the {H}offman constant of a system of linear constraints},
	author={Pena, Javier and Vera, Juan and Zuluaga, Luis},
	journal={arXiv preprint arXiv:1804.08418},
	year={2018}
}

@book{dantzig1963linear,
	title={Linear {P}rogramming and {E}xtensions},
	author={Dantzig, George},
	year={1963},
	publisher={Princeton University Press}
}

@misc{gurobi,
	author = {{Gurobi Optimization, LLC}},
	title = {{Gurobi} Optimizer Reference Manual},
	year = 2023,
	url = "https://www.gurobi.com"
}

@article{pdlpnews,
	author = {Mirrokni, Vahab},
	title = {Google {R}esearch, 2022 {\&} beyond: {A}lgorithmic advances},
	year = {2023},
	note = {\url{https://ai.googleblog.com/2023/02/google-research-2022-beyond-algorithmic.html}},
	howpublished = {\url{https://ai.googleblog.com/2023/02/google-research-2022-beyond-algorithmic.html}}
}

@article{lu2025geometry,
  title={On the geometry and refined rate of primal--dual hybrid gradient for linear programming},
  author={Lu, Haihao and Yang, Jinwen},
  journal={Mathematical Programming},
  volume={212},
  number={1},
  pages={349--387},
  year={2025},
  publisher={Springer Berlin Heidelberg Berlin/Heidelberg}
}

@article{lin2021admm,
	title={An {ADMM}-based interior-point method for large-scale linear programming},
	author={Lin, Tianyi and Ma, Shiqian and Ye, Yinyu and Zhang, Shuzhong},
	journal={Optimization Methods and Software},
	volume={36},
	number={2-3},
	pages={389--424},
	year={2021},
	publisher={Taylor \& Francis}
}

@article{lu2021nearly,
	title={Nearly optimal linear convergence of stochastic primal-dual methods for linear programming},
	author={Lu, Haihao and Yang, Jinwen},
	journal={arXiv preprint arXiv:2111.05530},
	year={2021}
}

@article{o2016conic,
	title={Conic optimization via operator splitting and homogeneous self-dual embedding},
	author={O’Donoghue, Brendan and Chu, Eric and Parikh, Neal and Boyd, Stephen},
	journal={Journal of Optimization Theory and Applications},
	volume={169},
	pages={1042--1068},
	year={2016},
	publisher={Springer}
}

@article{o2021operator,
	title={Operator splitting for a homogeneous embedding of the linear complementarity problem},
	author={O'Donoghue, Brendan},
	journal={SIAM Journal on Optimization},
	volume={31},
	number={3},
	pages={1999--2023},
	year={2021},
	publisher={SIAM}
}

@article{li2020asymptotically,
	title={An asymptotically superlinearly convergent semismooth {N}ewton augmented {L}agrangian method for linear programming},
	author={Li, Xudong and Sun, Defeng and Toh, Kim-Chuan},
	journal={SIAM Journal on Optimization},
	volume={30},
	number={3},
	pages={2410--2440},
	year={2020},
	publisher={SIAM}
}

@inproceedings{basu2020eclipse,
	title={Eclipse: An extreme-scale linear program solver for web-applications},
	author={Basu, Kinjal and Ghoting, Amol and Mazumder, Rahul and Pan, Yao},
	booktitle={International Conference on Machine Learning},
	pages={704--714},
	year={2020},
	organization={PMLR}
}

@article{gleixner2021miplib,
	title={{MIPLIB} 2017: data-driven compilation of the 6th mixed-integer programming library},
	author={Gleixner, Ambros and Hendel, Gregor and Gamrath, Gerald and Achterberg, Tobias and Bastubbe, Michael and Berthold, Timo and Christophel, Philipp and Jarck, Kati and Koch, Thorsten  and Linderoth, Jeff and L{\"u}bbecke, Marco and Mittelmann, Hans D.  and Ozyurt, Derya and  Ralphs, Ted K. and Salvagnin, Domenico and Shinano, Yuji},
	journal={Mathematical Programming Computation},
	volume={13},
	number={3},
	pages={443--490},
	year={2021},
	publisher={Springer}
}

@article{pena2021new,
  title={New characterizations of {H}offman constants for systems of linear constraints},
  author={Pena, Javier and Vera, Juan C and Zuluaga, Luis F},
  journal={Mathematical Programming},
  volume={187},
  pages={79--109},
  year={2021},
  publisher={Springer}
}

@article{osqp,
  title={{OSQP}: An operator splitting solver for quadratic programs},
  author={Stellato, Bartolomeo and Banjac, Goran and Goulart, Paul and Bemporad, Alberto and Boyd, Stephen},
  journal={Mathematical Programming Computation},
  volume={12},
  number={4},
  pages={637--672},
  year={2020},
  publisher={Springer}
}

@article{osqp-gpu,
  title={{GPU} acceleration of {ADMM} for large-scale quadratic programming},
  author={Schubiger, Michel and Banjac, Goran and Lygeros, John},
  journal={Journal of Parallel and Distributed Computing},
  volume={144},
  pages={55--67},
  year={2020},
  publisher={Elsevier}
}

@article{fv4,
  title={Equivalence of convex problem geometry and computational complexity in the separation oracle model},
  author={Freund, Robert M and Vera, Jorge R},
  journal={Mathematics of Operations Research},
  volume={34},
  number={4},
  pages={869--879},
  year={2009},
  publisher={INFORMS}
}

@article{fv1,
  title={Some characterizations and properties of the ``distance to ill-posedness'' and the condition measure of a conic linear system},
  author={Freund, Robert M and Vera, Jorge R},
  journal={Mathematical Programming},
  volume={86},
  number={2},
  pages={225--260},
  year={1999},
  publisher={Springer}
}

@article{lu2023cupdlp-c,
	title        = {{cuPDLP-C}: A Strengthened Implementation of {cuPDLP} for Linear Programming by {C} language},
	author       = {Lu, Haihao and Yang, Jinwen and Hu, Haodong and Huangfu, Qi and Liu, Jinsong and Liu, Tianhao and Ye, Yinyu and Zhang, Chuwen and Ge, Dongdong},
	year         = {2023},
	journal      = {arXiv preprint arXiv:2312.14832}
}

@article{lin2024pdcs,
  title={{PDCS}: A Primal-Dual Large-Scale Conic Programming Solver with {GPU} Enhancements},
  author={Lin, Zhenwei and Xiong, Zikai and Ge, Dongdong and Ye, Yinyu},
  journal={arXiv preprint arXiv:2505.00311},
  year={2025}
}

@article{lu2024restarted,
	title        = {Restarted {H}alpern {PDHG} for Linear Programming},
	author       = {Lu, Haihao and Yang, Jinwen},
	year         = {2024},
	journal      = {arXiv preprint arXiv:2407.16144}
}

@article{xiong2024role,
	title        = {{T}he Role of Level-Set Geometry on the Performance of {PDHG} for Conic Linear Optimization},
	author       = {Xiong, Zikai and Freund, Robert M},
	year         = {2024},
	journal      = {arXiv preprint arXiv:2406.01942}
}

@article{xiong2025high,
	title        = {High-Probability Polynomial-Time Complexity of Restarted {PDHG} for Linear Programming},
	author       = {Xiong, Zikai},
	year         = {2025},
	journal      = {arXiv preprint arXiv:2501.00728}
}

@article{xiong2024accessible,
	title        = {Accessible Complexity Bounds for Restarted {PDHG} on Linear Programs with a Unique Optimizer},
	author       = {Xiong, Zikai},
	year         = {2024},
	journal      = {arXiv preprint arXiv:2410.04043}
}

@article{brown1951computational,
  title={Computational suggestions for maximizing a linear function subject to linear inequalities},
  author={Brown, George W and Koopmans, Tjalling C},
  journal={Activity Analysis of Production and Allocation},
  pages={377--380},
  year={1951},
  publisher={J. Wiley New York}
}

@book{zoutendijk1960methods,
  title={Methods of Feasible Directions: A Study in Linear and Non-Linear Programming},
  author={Zoutendijk, Guus},
  year={1960},
  publisher={Elsevier, Amsterdam}
}

@article{lemke1961constrained,
  title={The constrained gradient method of linear programming},
  author={Lemke, Carlton E},
  journal={Journal of the Society for Industrial and Applied Mathematics},
  volume={9},
  number={1},
  pages={1--17},
  year={1961},
  publisher={SIAM}
}

@article{hinder2024worst,
	title        = {Worst-case analysis of restarted primal-dual hybrid gradient on totally unimodular linear programs},
	author       = {Hinder, Oliver},
	year         = {2024},
	journal      = {Operations Research Letters},
	publisher    = {Elsevier},
	volume       = {57},
	pages        = {107199}
}

@inproceedings{li2024pdhg,
  title={{PDHG}-unrolled learning-to-optimize method for large-scale linear programming},
  author={Li, Bingheng and Yang, Linxin and Chen, Yupeng and Wang, Senmiao and Mao, Haitao and Chen, Qian and Ma, Yao and Wang, Akang and Ding, Tian and Tang, Jiliang and  Sun, Ruoyu},
  booktitle={Proceedings of the 41st International Conference on Machine Learning},
  pages={29164--29180},
  year={2024}
}

@article{lu2024pdot,
  title={{PDOT}: A practical primal-dual algorithm and a {GPU}-based solver for optimal transport},
  author={Lu, Haihao and Yang, Jinwen},
  journal={arXiv preprint arXiv:2407.19689},
  year={2024}
}

@article{chen2025hprnew,
  title={{HPR-LP}: {An} implementation of an {HPR} method for solving linear programming},
  author={Chen, Kaihuang and Sun, Defeng and Yuan, Yancheng and Zhang, Guojun and Zhao, Xinyuan},
  journal={Mathematical Programming Computation},
  pages={1--28},
  year={2025},
  publisher={Springer}
}

@article{deng2025enhancednew,
  title={An enhanced alternating direction method of multipliers-based interior point method for linear and conic optimization},
  author={Deng, Qi and Feng, Qing and Gao, Wenzhi and Ge, Dongdong and Jiang, Bo and Jiang, Yuntian and Liu, Jingsong and Liu, Tianhao and Xue, Chenyu and Ye, Yinyu and Zhang, Chuwen},
  journal={INFORMS Journal on Computing},
  volume={37},
  number={2},
  pages={338--359},
  year={2025},
  publisher={INFORMS}
}

@article{huang2025restartednew,
  title={A restarted primal-dual hybrid conjugate gradient method for large-scale quadratic programming},
  author={Huang, Yicheng and Zhang, Wanyu and Li, Hongpei and Ge, Dongdong and Liu, Huikang and Ye, Yinyu},
  journal={INFORMS Journal on Computing},
  year={2025},
  publisher={INFORMS}
}

@article{lu2025cupdlp,
  title={{cuPDLP. jl}: {A GPU} implementation of restarted primal-dual hybrid gradient for linear programming in {J}ulia},
  author={Lu, Haihao and Yang, Jinwen},
  journal={Operations Research},
  volume={73},
  number={6},
  pages={3440--3452},
  year={2025},
  publisher={INFORMS}
}

@article{lu2025practical,
  title={A Practical and Optimal First-Order Method for Large-Scale Convex Quadratic Programming},
  author={Lu, Haihao and Yang, Jinwen},
  journal={Mathematical Programming},
  pages={1--38},
  year={2025},
  publisher={Springer}
}
\endgroup

\end{document}